%% file: Main_Final_Version.tex
\documentclass[11.5pt]{article}

\usepackage{jmlr2e}

\usepackage[utf8]{inputenc}

\input{Bookkeaping/1_mathcommands}

\input{Bookkeaping/3_Editing}

\usepackage{natbib}
\usepackage{comment}
\DeclareFontFamily{U}{rcjhbltx}{}
\DeclareFontShape{U}{rcjhbltx}{m}{n}{<->rcjhbltx}{}
\DeclareSymbolFont{hebrewletters}{U}{rcjhbltx}{m}{n}

\author{Reza Arabpour\footnote{Vector Institute and Department of Mathematics, McMaster University.},
Luca Galimberti\footnote{Department of Mathematics, King's College London.}, 
Anastasis Kratsios\footnote{Vector Institute and Department of Mathematics, McMaster University.}, 
Giulia Livieri\footnote{Department of Statistics, London School of Economics and Political Science.},
John Armstrong\footnote{Department of Mathematics, King's College London.}}

\newcommand{\code}{%
    \url{https://github.com/arabporr/HyperNetwork}
}

\numberwithin{equation}{section} 

\begin{document}

% JMLR
\ShortHeadings{Low-dim. approximations of conditional Volterra processes via NPC geometry}{Arabpour, Armstrong, Galimberti, Kratsios, Livieri}
\title{Low-dimensional approximations of the conditional law of Volterra processes: a non-positive curvature approach}
\author{\name Reza Arabpour\footnotemark[1]\thanks{Alphabetic ordering.}
\email arabpour@mcmaster.ca\\ 
\addr Vector Institute and Department of Mathematics, McMaster University
\AND 
\name John Armstrong\footnotemark[1]
\email john.armstrong@kcl.ac.uk\\
\addr Department of Mathematics, King's College London.
\AND
\name Luca Galimberti\footnotemark[1]~\footnotemark[2]\thanks{Primary authors, equal contribution, alphabetical ordering.}
\email luca.galimberti@kcl.ac.uk\\
\addr Department of Mathematics, King's College London
\AND
\name Anastasis Kratsios\footnotemark[1]~\footnotemark[2]~\thanks{Corresponding author.}
\email kratsioa@mcmaster.ca\\
\addr Vector Institute and Department of Mathematics, McMaster University
\AND
\name Giulia Livieri\footnotemark[1]~\footnotemark[2]
\email g.livieri@lse.ac.uk\\
\addr Department of Statistics, London School of Economics and Political Science.
}

\editor{TBD}

\maketitle

\begin{abstract}
Predicting the conditional evolution of Volterra processes with stochastic volatility is a crucial challenge in mathematical finance. While deep neural network models offer promise in approximating the conditional law of such processes, their effectiveness is hindered by the curse of dimensionality caused by the infinite dimensionality and non-smooth nature of these problems. To address this, we propose a two-step solution. Firstly, we develop a stable dimension reduction technique, projecting the law of a reasonably broad class of Volterra process onto a low-dimensional statistical manifold of non-positive sectional curvature.  Next, we introduce a sequentially deep learning model tailored to the manifold's geometry, which we show can approximate the projected conditional law of the Volterra process.  Our model leverages an auxiliary hypernetwork to dynamically update its internal parameters, allowing it to encode non-stationary dynamics of the Volterra process, and it can be interpreted as a gating mechanism in a mixture of expert models where each expert is specialized at a specific point in time.  Our hypernetwork further allows us to achieve approximation rates that would seemingly only be possible with very large networks.
\end{abstract}
\noindent
% {\bf Mathematics Subject Classification (2020):} \textbf{TBD}.

\section{Introduction}
\label{s:Intro}
The dynamic prediction of the conditional distribution $\mathbb{P}[X_{t+1}|X_{[0:t]}=x_{[0:t]}]$ of a non-Markovian Volterra process $X_{\cdot}$, conditioned on its realized path $x_{[0:t]}$ up to any given time $t$, is a fundamental problem spanning the sciences, ranging from Bayesian modelling (e.g., \cite{BarnardoSmith_BayesianStatistics_1994}) to mathematical finance (e.g., \cite{back2004incomplete}).  While there are several machine learning models for learning a process conditional distribution on its historical paths; e.g.\ \cite{bishop1994mixture, gauthier2014conditional, mirza2014conditional, borovykh2017conditional, chevyrev2022signature} and deep learning models which can approximate signed measure-valued functions~\cite{chen1995universal,korolev2022two,benth2023neural,cuchiero2023global}, the available quantitative approximation bounds for measure-valued models; e.g.\ \cite{kratsios2021universal,Acciaio2022_GHT} suggest that the measure-valued maps cannot be efficiently approximated.  This is due to two factors; firstly, they are infinite-dimensional, meaning they suffer from extreme forms of the curse of dimensionality; see \cite{lanthaler2023curse} for a lower-bound in the linear case.  Secondly, most spaces of probability measures, e.g.\ Wasserstein spaces, do not have any smooth or linear structure, which a deep learning model can naturally leverage. To the best of the authors' knowledge, there are currently no available deep-learning models which can approximate the evolving conditional distribution of most stochastic processes while also depending on a computationally feasible number of parameters.  

To address this issues, we present a two-step approach for dynamically approximating the conditional evolution of a class of discrete-time stochastic processes $X_{\cdot}$.  By allowing for an irreducible dimension-reduction-type error, our first step corrects the deficits of the geometries of most spaces of probability measures by projecting $\mathbb{P}[X_{t+1}|X_{[0:t]}=x_{[0:t]}]$ onto the (smooth) Riemannian manifold $\mathcal{N}_d$ of non-singular $d$-dimensional Gaussian measures equipped with a new Riemannian metric which has a non-positively curved geometry when the $d\ge 2$. This can be contrasted with the standard Fisher(-Rao) metric \citep{Dowty_2018__ChentsovExponentialFamilies}, from information geometry, which is only non-positively curved when $d=1$~\citep{Skovgaard_FisherRaoGeometry_NonDegenreateGaussians__1984}.  Critically, we show that (refer to Theorem~\ref{thrm:Projection_Results}), though our new projection shares an analogous characterization as ``distance'' minimizer to information-theoretic projections, e.g.~the $I$-projection \cite{csiszar2003information,CsiszarIprojectionIntro_1975}, it preserves the dependence of $\mathbb{P}[X_{t+1}|X_{[0:t]}=x_{[0:t]}]$ on the realized path $x_{[0:t]}$ in a stable way as it is Lipschitz.  Furthermore, the dependence of the projected dynamical system (see below) on the historical paths generated by $X_{\cdot}$ is found to be proportional to the persistence of the memory of the stochastic process itself  (Theorem~\ref{thrm:VanishingMemoryProperty_Qx}).

The projection of the conditional distribution $\mathbb{P}[X_{t+1}|X_{[0:t]}=x_{[0:t]}]$ onto $\mathcal{N}_d$, defines a dynamical system on the Riemannian manifold $\mathcal{N}_d$.  This long memory dynamical system is then approximated by a dynamic version of the geometric deep learning model of \cite{kratsios2021universal}.  We prove a quantitative universal approximation theorem which shows that this system can be approximated by our \textit{sequential} deep learning model (Theorem~\ref{theorem:optimal_GDN_Rates__ReLUActivation}). Our result has the following properties: 1) it shows that any infinite-memory dynamical systems between non-positively curved Riemannian manifolds, e.g.\ generalized hyperbolic spaces, can be approximated provided that its memory fades. This is similar to results in the reservoir computing literature \cite{GrigoryevaOrtegaUniversalDiscretTime2018_JMLR,grigoryeva2019differentiable}.  2) It shows that the geometric deep learning model does not suffer from the curse of dimensionality when the dynamical systems evolve according to smooth dynamics; in particular, this happens for the projected conditional law of $X_{\cdot}$ when its generator is ``smooth'' and has no hidden random state.  3) In the static setting, i.e.\ when ignoring the flow of time, our bounds are polynomially tighter than the non-Euclidean universal approximation theorems of \cite{kratsios2021universal,Acciaio2022_GHT} and in the dynamic setting, our bounds are exponentially tighter than the static setting.  4) Our constructive proof yields a federated algorithm where a sequence of independent experts approximate the dynamical system at individual points in time, after which an overarching hypernetwork \cite{ha2017hypernetworks} is used to synchronize them and create the recurrence without having to optimize (backpropagate) in time.  
While our approximation theorems hold for dynamical systems on any non-positively curved Riemannian manifold, we focus our exposition on the statistical manifold $\mathcal{N}_d$. 
We empirically verify the trainability of the proposed model and the role of each component via an ablation study evaluating the dependence of the model on its parameters. 

All code can be found at \code.

\paragraph{Outline} The paper is organized as follows.
\begin{enumerate}
    \item Section~\ref{s:Prelim} reviews notions from metric geometry required to formulate our main results.  
    \item Section~\ref{s:Prob_Setting} formalizes the setting and introduces our new perturbation of the Fisher metric on $\mathcal{N}_d$, which has a well-behaved geometry. 
    \item Section~\ref{s:Projection} contains our results on the Gaussian random projection of the conditional distribution of Volterra processes.
    \item Section~\ref{s:Approximation} contains universal approximation theorems for our model, first in the static and then in the dynamic settings.  
    \item Section~\ref{s:Numerics} contains an ablation study of our model, which confirms our theoretical results.  
\end{enumerate}
All proofs are deferred to the appendix, which also describes results from differential geometry for which we have been unable to provide a proper reference in the literature. The appendix also contains an example outlining the problem with information projection for dimension reduction of the conditional law of Volterra processes.  

We assume the reader to be familiar with the terminology of stochastic processes and optimal transport.  Both notions are widely used in machine learning, for example in denoising diffusion models \cite{sohl2015deep} and generative adversarial networks \cite{arjovsky2017wasserstein}.

\subsection{Related Work}
\label{s:Introduction__ss:RelatedWork}

Our results focus on the broad class of stochastic processes known as \textit{Volterra processes}, which represent a rich yet well-structured class of \textit{non-Markovian} stochastic differential equations (SDEs, henceforth) with latent stochastic factors. Both the discrete and the continuous versions of stochastic Volterra processes, and their generalizations~\cite{jaber2024polynomial}, play a crucial role in mathematical finance (e.g., \citep{jacquier2018vix,AbiJaberLarssonPulido_AffinVolterraProceeses_2019_AnnProb,ChristaJosef_2020_JEE_GenFellerStochVoltProcessesAffine,BondiLivieriPulido_2024_SPA}, reservoir computing (e.g., \cite{grigoryeva2019differentiable,gonon2022reservoir}), engineering (e.g., \cite{shiki2017application}), and computational biology (e.g., \cite{korenberg1996identification}). 

Our projection of the conditional law of stochastic Volterra processes exploits the existence and stability of barycenters of probability measures on Riemannian manifolds of non-positive curvature, in the sense of \cite{Aleksandrov_OG} and \cite{Revetnyak_1960}.  Various notions of barycenters, e.g.\ Wasserstein barycenters \citep{cuturi2014fast,srivastava2018scalable,puccetti2020computation,altschuler2021wasserstein,kolesov2024estimating,kolesov2024energyguided}, have been used in generative modelling as they provide a natural tool for averaging distributions.  The new element of our approach is to consider barycenters of distributions, also called the Karcher/Fr\'{e}chet mean, whose definition is rooted in Riemannian \citep{grove1973conjugate,karcher1977riemannian,kim2020nonpositive} and metric geometry \citep{SturmMartingalesNPCPt2,Sturm_2003,ohta2012barycenters,MendelNaor_Barycenter_SpectralExtension_2013,Yokota_BarycenterCat1_2017}.  These provide a non-linear generalization of the expectation of a random variable taking values in a non-linear space and are commonly used in computational geometry \citep{BiniBruno_KarcherMeanPSD_2103}, shape analysis \citep{le2001locating,fletcher2004principal}, geometric statistics \citep{LimPalfia_2020_SLOLN_Karchermean,bhattacharya2003large_AnnStatI,bhattacharya2003large_AnnStatII}, and recently in differential privacy \citep{utpala2023differentially}.  We will use them to determine the single most representative $d$-dimensional non-singular Gaussian measure, amongst a family of random Gaussian measures in $\mathcal{N}_d$ in which we encode the conditional law of the Volterra process.  

This notion of barycenter, or intrinsic expectation, is only well-defined globally for manifolds with non-positive curvature \citep{afsari2011riemannian}.  Examples include the hyperbolic spaces used in hierarchical clustering \citep{chami2020trees} or latent graph inference for GNNs \citep{de2022latent}.
Such geometries appear for various statistical manifolds  \citep{pinele2020fisher,le2021fisher} with the Fisher-Rao metric  \citep{AyJostLeSchwachhofer__InformationGeometryBook_2017}; notably the \textit{one-dimensional} non-degenerate Gaussian distributions \citep{AtkinsonMitchell_1981_SPD,CostaSummaryArticle_2015,pinele2020fisher}.  
However, as shown in \citep{Skovgaard_FisherRaoGeometry_NonDegenreateGaussians__1984}, the space $\mathcal{N}_d$ of multivariate non-degenerate Gaussian distributions fails to have non-positive curvature in dimension $2$ or greater.  Our perturbation of the standard information-theoretic geometry on the statistical manifold $\mathcal{N}_d$ 
(see Proposition~\ref{prop:NPC_FisherRao})
identifies a non-positively curved complete Riemannian geometry on $\mathcal{N}_d$ which coincides with the standard information geometry on any submanifolds of Gaussian distributions with fixed mean.
This allows us to apply the intuition from information geometry, but with the analytic benefit of a non-positive metric geometry.
Moreover our construction provides a closed-form expression for the Riemannian distance: for typical statistical families, the Riemannian distance in the Fisher-Rao metric can only be computed numerically.

Projecting the conditional distribution of the stochastic Volterra process $X_{\cdot}$, conditioned on its realized path $x_{[0:t]}$ up to any time $t$, down to $\mathcal{N}_d$ results in a (generalized) dynamical system between finite-dimensional spaces.  Since all the resulting spaces are finite-dimensional and well-structured, one can reasonably hope that this system can be approximated without the curse of dimensionality if the involved maps are smooth enough, a feature not shared by infinite-dimensional approximation problems \citep{lanthaler2023curse,Acciaio2022_GHT}. There is a well-developed literature on the approximation of dynamical systems by recurrent deep learning models such as reservoir computers \citep{GrigoryevaOrtegaUniversalDiscretTime2018_JMLR,gonon2019reservoir,grigoryeva2019differentiable,Gonon2022NNs}, recurrent neural networks \citep{LiJMLR2022,RecepHutterHelmut_2022_AHA__RNNs}, or transformers \cite{Yun2020Are}. 
However, the available universal approximation results in the literature
only apply to dynamical systems between linear input spaces and systems whose dynamics do not change in time.  We, therefore first extend the static \textit{geometric deep learning} of \cite{kratsios2021universal} to a sequential/dynamic model capable of processing \textit{sequences} of inputs and outputs in any given appropriate pair of non-positively curved Riemannian manifolds, e.g.\ on $\mathcal{N}_d$, and we then prove a universal approximation showing that it can approximate most time-inhomogeneous dynamical systems between these spaces, possibly having infinite but polynomially fading memory (Theorem~\ref{defn:HCGNs}).  

Our estimates on the parameters and guarantees are novel even in the Euclidean setting:  many of the aforementioned approximation theorems for recurrent neural networks (RNNs)  are only qualitative and no rates are provided, especially for low-regularity dynamical systems with slowly-fading infinite memory.  
Even in the static setting, our results significantly improve the rates for the available universal approximation theorems for deep learning models between Riemannian manifolds \cite{kratsios2021universal,Acciaio2022_GHT}.  Furthermore, in the static setting (Theorem~\ref{thrm:approx_HGCNN}), our approximate rates are optimal since they match those of \cite{pmlr-v75-yarotsky18a,yarotsky2020phase,LuShenYangZhang_2021_UATRegularTargets,ShenYangHaizhaoZhang_OptimalReLU_2022} in the special case where the input space is Euclidean and the output space is the real line.  In the static setting, there are other geometric deep learning models defined between Riemannian manifolds such as the geodesic convolutional model of~\citep{GCNNs_Bronstein_2015}, but no approximation theory for these models has yet been developed.

Mixture of experts (MoE) models such as DBRX~\cite{dbrx_MosaicAI2024},  Gemini~\citep{Gemini}, Switch Transformers~\cite{moe_switch_fedus2022switch}, Mixtral~\cite{moe_jiang2024mixtral}, and many others (e.g.~\cite{saad2023active,chowdhury2023patch,li2024merge,puigcerver2024from,MoERaeidStasisandfriends2024}) have taken a central role in modern deep learning, due to the need to scale-up the model complexity while maintaining a constant computational cost on the forward pass~\cite{borde2024breaking,kratsios2024mixtureInfDim}.  This is achieved via a gating mechanism which routes any given input to one of a large number of ``expert'' neural network models, which is then used to produce a prediction from that input.  Thus, only the gating network parameters and the selected ``expert'' neural network are ever activated for that input.  One can interpret our (hyper-geometric network) as a mixture of infinitely many experts, and each of them specializes in predicting at exactly one specific moment in time.  The hypernetwork in our HGN model acts as a gating mechanism which, given the current point in time, routes the input to the corresponding expert at that moment in time.

\section{Geometric Background}
\label{s:Prelim}
We briefly overview some of the relevant notions from the geometry of non-positively curved metric spaces, also called (global) NPC spaces or Hadamard spaces. For further details we
recommend the book by \cite{BridsonHaefliger_1999Book} as well as the results of \cite{Sturm_2003}.
% With slight abuse of notation, this section follows the usual convention that $d$ refers to a metric.
% \giulia{I know that it is clear from the context, but $d$ now is the distance.}
\begin{definition}[NPC space]
\label{def:NPC}
A metric space $(N,d)$ is called global NPC space if it is complete and for each $x_0, x_1 \in N$ there exists a point 
 $y\in N$ with the property that for all points $z\in N$
\[
d^2(z,y) \le \frac{1}{2} d^2(z,x_0) + \frac{1}{2}d^2(z,x_1) - \frac1{4}d^2(x_0,x_1)
.\]
\end{definition}
In particular, $(N,d)$ is contractible and hence simply connected.

A primary example of global NPC spaces is provided by Cartan-Hadamard manifolds. 

\begin{proposition}[Manifolds (Proposition 3.1 in \cite{Sturm_2003}]\label{prop:sturmone}
Let $(N, g)$ be a Riemannian manifold and let $d$ be its induced Riemannian distance. Then $(N, d)$ is a global NPC space if and only if it is complete, simply connected and of non-positive (sectional) curvature.
\end{proposition}

We will describe some results on barycenters of probability measures on metric spaces of NPC. First, we need to introduce the following notation; cfr.~\cite{Sturm_2003}, Section 4. Let $(N,d)$ be a complete metric space. Denote by $\mathcal{P}(N)$ the set of all probability measures $p$  on $(N,\mathcal{B}(N))$ with separable support $\text{supp}(p)\subset N$. For $1 \leq \theta < \infty$, $\mathcal{P}^{\theta}(N)$ denotes the set of $p \in \mathcal{P}(N)$ with $\int_N d^{\theta}(x,y)p(dy)<\infty$ for some (hence all) $x \in N$, and $\mathcal{P}^{\infty}(N)$ denotes the set of all $p \in \mathcal{P}(N)$ with bounded support. Denote by $\mathcal{W}_1$ the $(L^1-)$ Wasserstein (or Kantorovich-Rubinstein) distance on $\mathcal{P}^{1}(N)$. Let $(\Omega, \mathcal{F}, \mathbb{P})$ be an arbitrary probability space and $X:\Omega\rightarrow N$ a \textit{strongly}\footnote{Given a measurable space $(M,\mathcal{M})$ and $(N, d)$ a metric space, a map $f:M\rightarrow N$ is called \textit{strongly measurable} %iff it is the (point-wise) limit of a elementary measurable maps or, equivalently, 
iff it is $\mathcal{M}/\mathcal{B}(N)$ measurable and has separable range $f(M)$.} measurable map. It defines a probability measure $X_{*}\mathbb{P}\in\mathcal{P}(N)$, called the \textit{push forward measure} of $\mathbb{P}$ under $X$, by
\begin{equation*}
    X_{*}\mathbb{P}[A] \eqdef \mathbb{P}[X^{-1}(A)] =\mathbb{P}[\omega\in\Omega\,:\,X(\omega)\in A],\quad\quad\forall A \in \mathcal{B}(N).
\end{equation*}
In probabilistic language, a strongly measurable map $X:\Omega\rightarrow N$ is called $N$-valued \textit{random variable}, the push forward measure $X_{*}\mathbb{P}$ is called \textit{distribution} of $X$ and denoted by $\mathbb{P}_{X}$. In particular $\mathbb{P}_{X} \in \mathcal{P}^{\theta}(N) \Leftrightarrow X \in L^{\theta}(\Omega, N)$ where $L^{\theta}(\Omega, N)$ denotes the set of all (strongly) measurable maps $f:\Omega\to N$ such that $\int_\Omega d^\theta(x,f(\omega))\mathbb P(d\omega)<\infty$ for some (and hence all) $x\in N$. We have the following     
\begin{proposition}[Existence of barycenters (Proposition 4.3 in \cite{Sturm_2003})]
\label{prop:sturmtwo}
    Let $(N, d)$ be a global NPC space and fix $y \in N$. For each $q \in \mathcal{P}^{1}(N)$ there exists a unique point $z \in N$ which minimizes the continuous function $z 
\mapsto \int_{N} [d^2(z, x)-d^2(y, x)]\,q(dx)$. This point is independent of $y$: it is called barycenter (or, more precisely, $d^2$-barycenter) of $q$ and denoted by
    \begin{equation*}
        \beta(q) = \underset{z \in N}{\operatorname{argmin}} \int_{N} [d^2(z, x)-d^2(y, x)]\,q(dx).
    \end{equation*}
If $q \in \mathcal{P}^{2}(N)$, then $\beta(q) = \underset{z \in N}{\operatorname{argmin}} \int_{N} d^2(z, x)\,q(dx)$.
\end{proposition}
As an example, if $q = \delta_{x_0}$, then trivially $\beta(q)=x_0$.
For $X \in L^{1}(\Omega, N)$ we define its \textit{expectation} by 
\begin{equation*}
    \mathbb{E}[X]  \eqdef  \underset{z \in N}{\operatorname{argmin}}\,\mathbb{E}[d^2(z, X)-d^2(y, X)] = \beta(\mathbb{P}_{X}).
\end{equation*}
\noindent That is, $\mathbb{E}[X]$ is the unique minimizer of the function $z \mapsto \int_{N} [d^2(z, x)-d^2(y, x)]\mathbb{P}_{X}(\,dx)$ on $N$ (for each fixed $y \in N$). 
We recall the following:
\begin{theorem}[Fundamental contraction property; {\citep[Theorem 6.3]{Sturm_2003})}]
\label{thm:sturmthree}
    For any given $p,q\,\in\mathcal{P}^{1}(N)$ it holds
    \begin{equation}\label{eq:fundcontprop}
    d(b(p), b(q)) \leq \mathcal{W}_1(p, q).
\end{equation} 
\end{theorem}

\section{Problem Setting}
\label{s:Prob_Setting}

In this paper, for a fixed $T\in \mathbb{N}_+$ we consider a discrete-time stochastic process $X_{\cdot}=(X_t)_{t=0}^T\subset \mathbb R^{d}$, which evolves according to the following dynamics:
\begin{equation}
\tag{Volterra}
\label{eq:Volterra_X}
X_{t+1} = X_t + \operatorname{Drift}(t,X_{[0:t]}) + \operatorname{Diffusion}(t,X_{[0:t]},\mathbf{S}_{[0:t]})
W_t,\quad t=0,\dots,T-1.
\end{equation}
In the previous equation, $W_{\cdot}\eqdef (W_t)_{t=0}^{T-1}$ is a Gaussian white-noise, i.e.\ an i.i.d.\
collection of $\mathbb{R}^{d}$-valued standard normal random variables
defined on a probability space $(\Omega,\mathcal{F},\mathbb{P})$, and $\mathbf{S}_{\cdot}\eqdef (S_t)_{t=0}^{T-1}$ is a symmetric matrix-valued stochastic process independent of $W_{\cdot}$, and $X_0 = x_0\in \mathbb{R}^d$.
In this set-up, the direction and random fluctuations of the process $X_{\cdot}$, depend on its realized path and the path of a \textit{latent (unobservable) stochastic factor} $\mathbf{S}_{\cdot}$

We focus on the case of discrete-time Volterra processes with non-singular diffusion. Let 
$\operatorname{Sym}(d)$ (resp.~$\operatorname{Sym}_+(d)$) denote the set of $d\times d$ symmetric (resp.~symmetric and positive definite) matrices with real entries. 
We will call a function whose best Lipschitz constant is {at most} $L$ an $L$-Lipschitz function. 
% \john{Surely we mean which has a Lipschitz constant of $L$ or better?}
A Volterra process with non-singular diffusion is a discrete-time stochastic process with $\operatorname{Drift}:\mathbb R \times \mathbb{R}^{d(t+1)}\to \mathbb{R}^d,$ and $\operatorname{Diffusion}: \mathbb R \times \mathbb{R}^{d(t+1)}\times(\operatorname{Sym}(d))^{t+1}\to\operatorname{Sym}_+(d)$ satisfying: 
\begin{equation}
\label{eq:VolterraProcess_Dynamics}
\begin{aligned}
\operatorname{Drift}\big(t,x_{[0:t]}\big)\eqdef &   \sum_{r=0}^t   \kappa(t,r)
\,\mu(t,x_r), 
    \\
        \operatorname{Diffusion}\big(t,x_{[0:t]},s_{[0:t]}\big) 
    \eqdef &
            \exp\Biggl(
                \frac1{2}
                \,
                \sum_{r=0}^t
                \,
                \kappa(t,r)
                    \,
                        [\sigma(t,x_r)
                        +
                        s_r]
            \Biggr), 
\end{aligned}
\end{equation}
for all $t\in\mathbb N_+$, $x_{[0,t]}\in \mathbb R^{d(t+1)}$, and $s_{[0,t]}\in (\operatorname{Sym}(d))^{t+1}$. Here $\exp$ denotes the matrix exponential, $\mu:\mathbb{R}^{1+d}\rightarrow \mathbb{R}^{d}$, $\sigma:\mathbb{R}^{1+d}\rightarrow \operatorname{Sym}(d)$ are $L_{\mu}$-Lipschitz and $L_{\sigma}$-Lipschitz functions respectively and the \textit{Volterra kernel} $\kappa$ is a real-valued function satisyfing certain properties which we will decribe below.

If the process $\mathbf{S}$ is $\mathbb{P}$-a.s.\ equal to $0$ then, we observe that for each time $t\in \mathbb{N}_+$ and each path $x_{[0:t]}\in \mathbb{R}^{(1+t)d}$ realized by the process $X_{\cdot}$ up to time $t$, the distribution of $X_{t+1}$ conditioned on $X_{[0:t]}=x_{[0:t]}$ is \textit{Gaussian} with mean $x_t+\operatorname{Drift}(t,x_{[0:t]})$ and covariance equal to $\operatorname{Diffusion}(t,x_{[0:t]},\boldsymbol{0})^2=
    \exp\big( \sum_{r=0}^t \,\kappa(t,r) \, \sigma(t,x_r) \big)
$.

The \textit{(discrete-time) Volterra kernel} in Equation \eqref{eq:VolterraProcess_Dynamics}  determines the persistence of $X_{\cdot}$'s memory on the distant past. A Volterra kernel is defined to be a map
$\kappa:
            \big\{(t,r):\,t, r\in \mathbb{N} \mbox{ and } r \le t
            \big\}
        \rightarrow 
            [0,1] $ satisfying
\[
        0 < \sum_{r=0}^t\,\kappa(t,r) \le 1
    \mbox{ and }
        \kappa(t,0) =0 
\]
whenever $t>0$.  We typically require that the Volterra kernel satisfies either one of the following decay conditions 
\begin{enumerate}[label=\textcolor{blue}{\roman*.}, ref=\roman*]
    \item
    \textbf{Exponential decay:}  For some $0<\alpha<1$ and $C>0$, $\kappa(T,r) \le C \, \alpha^{T-r}$ for all integers $0\le r\le T$;
    \label{exp_decay}
    \item
    \textbf{Polynomial decay:} For some $\alpha<-1,C>0$, $\kappa(T,r) \le C(T-r)^{\alpha}$ for all integers $0\le r <T$.
    \label{pol_decay}
\end{enumerate}

We define the ``tangential version'' of the $\operatorname{Diffusion}$ map via the matrix logarithm 
\[
        \operatorname{diffusion}(t,x_{[0:t]},s_{[0:t]}) 
    \eqdef 
        \log\Big(
            \operatorname{Diffusion}(t,x_{[0:t]},s_{[0:t]})
        \Big)
    =
        \frac1{2}
        \,
        \sum_{r=0}^t
            \,
            \kappa(t,r)
            \,
            [\sigma(t,x_r)
            +
            s_r]
    ,
\]
where the uniqueness of the choice of branch, in this case the logarithm, is discussed in Appendix~\ref{s:auxlemmata__ss:Matrix_Alg}.
The intuition is that the inverse Riemannian exponential map at $\mathcal{N}(0,I_{d})$ on the submanifold of $\mathcal{N}_{d}$ of mean $0$ normal distributions coincides with the matrix logarithm\footnote{See~\eqref{prop:NPC_FisherRao__ExponentialNiceForm} in Proposition~\ref{prop:NPC_FisherRao}.}.  The stochastic diffusion factor $s_{[0:t]}$ will be a realization of a $\operatorname{Sym}(d)$-valued stochastic process $\mathbf{S}_{\cdot}\eqdef (\mathbf{S}_t)_{t=0}^{T-1}$ on $(\Omega,\mathcal{F},
\mathbb{P})$. 
We require that the stochastic $\operatorname{Sym}(d)$-valued process and the diffusive factor $\sigma$ satisfy the following conditions. 

\begin{assumption}[Uniformly bounded stochastic factor]
\label{ass:uniformboundedness__SProcecess}
We assume there is a constant $R> 0$ satisfying 
\begin{equation}
\label{eq:as_Bounded_StochasticFactor}
        \sup_{t \le T-1}\,
            \|\mathbf{S}_t\|_F
        \le R,
\quad \mathbb{P}-a.s. 
\end{equation}
where $\|\cdot\|_F$ denotes the Frobenius norm of a matrix, and that there is $M>0$ such that 
\[
    \sup_{t\le T-1,\,x\in\mathbb{R}^d } \|\sigma(t,x)\|_F\le M.
\]
\end{assumption}
 
\begin{remark} Assumption~\ref{ass:uniformboundedness__SProcecess} is required for the existence of a barycenter map.  Though this boundedness assumption is a bit strong, it can potentially be relaxed to an exponentially decaying moment's assumption in technical future work.
\end{remark}

We now provide an example concerning discretization of stochastic delay differential equations that fit into our framework.

\begin{example}[Stochastic delay differential equations]
\label{ex:sdde}

Considered a filtered probability space \\
$(\Omega,\mathcal{A},(\mathcal{A}_u)_{u\in[0,\mathcal T]},\mathbb P)$ satisfying the usual assumptions and carrying a standard one-dimensional Wiener process $W(u)$
($u$ denotes the time-variable here). Let $f:\R\to \R$ and $g:\R\to \R_{>0}$ be bounded and Lipschitz continuous functions. Fix a time delay $\tau>0$ and let $\psi(u)$ be an $\mathcal{A}_0$-measurable $C([-\tau,0],\R)$ valued random variable satisfying $\mathbb E[\norm{\psi}_\infty^2]<\infty$. Consider the following stochastic delay differential equation
\[
X(u)=X(0) + \int_0^uf(X(s-\tau))ds + \int_0^ug(X(s-\tau))dW(s)
\]
for $0\le u \le \mathcal{T}$ and with $X(u)=\psi(u)$ for $-\tau\le u \le 0$ (for further details see e.g. \cite{buckwar2000introduction}). We consider a grid with a uniform step $h$ on the interval $[0,\mathcal T]$ and $h=\mathcal{T}/N,u_t=th,t=0,\dots,N$, and assume that there is an integer $N_\tau$ such that $\tau =h N_\tau$. 
The Euler-Maruyama scheme is
\begin{equation}
\label{eq:Euler_DelayedSDE}
\tilde{X}_{t+1} = \tilde{X}_t + hf(\Tilde{X}_{t-N_\tau}) + g(\Tilde{X}_{t-N_\tau})\Delta W_{t+1},\quad 0\le t \le N-1
\end{equation}
where $\Tilde{X}_{t-N_\tau}:=\psi(u_t-\tau)$ for $t-N_\tau\le 0$, and $\Delta W_{t+1}:=W((t+1)h)-W(nh)\sim N(0,h)$ are independent.

Consider an identically zero stochastic factor $\mathbf S$, setting $\kappa(t,r)=0$ for $r\neq t-N_\tau\ge 0$ and $\kappa(t,r)=1$ for $r=t-N_\tau\ge 0$
, $\mu(x):=hf(x)$ and $\sigma(x)=2h\log g(x)$, from which we see that the discretization in~\eqref{eq:Euler_DelayedSDE} falls within our framework. 
\hfill\\
Nice applications involving stochastic differential are linked to, \textit{e.g.}, problems of optimal advertising under uncertainty for the introduction of a new product to the market (\cite{gozzi200513}), some infinite-dimensional variants of the Black-Scholes equation for the fair price of an option (\cite{fuhrman2010stochastic} and \cite{chang2007infinite}), stochastic control problem with delay arising in a pension fund models (\cite{federico2011stochastic}).
\end{example}
% }

\subsection{A Globally NPC perturbation of the Fisher geometry for Gaussian measures}
\label{sec:InformationGeometryGaussianMeasures}
Our target Riemannian (statistical) manifold of interest is the set $\mathcal{N}_d$ of non-singular Gaussian measures of a given dimension $d$, and formally defined below, endowed with a perturbation of the Fisher Riemannian metric (\cite{Skovgaard_FisherRaoGeometry_NonDegenreateGaussians__1984}), whose geometry we now review.  Let $\mathcal{N}_{d}$ be the collection of non-degenerate normal distributions of fixed dimension $d$, \textit{i.e.}
\begin{equation}\label{eq:NormalFixedDimension}
        \mathcal{N}_{d} 
    \eqdef 
        \{\mathcal{N}_{d}(\mathfrak{m},\Sigma)\,:\,
                \mathfrak{m} \in \mathbb{R}^{d} 
            \mbox{ and } 
                \Sigma \in \operatorname{Sym}_+(d)
        \},
\end{equation}
where, we recall, $\operatorname{Sym}_{+}(d)$ is the set of $d\times d$ positive definite symmetric matrices. Thus, $\mathcal{N}_{d}(\mathfrak{m},\Sigma)$ has Lebesgue density given by 
\begin{equation*}
    p_{\mathfrak{m},\Sigma}(u) =
    \frac{e^{-\frac{1}{2} (u - {\mathfrak{m}})^{\top} \Sigma^{-1} (u - {\mathfrak{m}})}}{
        \sqrt{(2\pi)^{d}\det (\Sigma)}}, \quad u\in\mathbb{R}^{d}. 
\end{equation*}
Through the identification
\begin{equation}\label{eq:identification_Nd_Vectors}
    \phi(\mathcal{N}_{d}(\mathfrak{m},\Sigma)) = ( (\mathfrak{m}_i)_{i=1}^{d},(\Sigma_{ij})_{i\le j}),
\end{equation}
where $\mathfrak{m}=(\mathfrak{m}_1,\dots ,\mathfrak{m}_d)^\top,\Sigma=(\Sigma_{ij})_{i,j=1}^d$, we see that $\mathcal{N}_{d}$ is isomorphic to an open subset $\Theta\subset\mathbb R^m$, $m= (d^2 + 3d)/2$. Hence, $\mathcal{N}_d$ can be considered as a differentiable manifold of dimension $m$, with $(\mathcal{N}_d, \phi)$ as a global chart.  We define a basis, in this coordinate system, of the set of vector fields on $\mathcal{N}_d$
\[
\frac{\partial}{\partial \mathfrak{m}_i},\quad i=1,\dots ,d, \quad \frac{\partial}{\partial \Sigma_{ij}},\quad i,j=1,\dots d,\, i\le j.
\]
We identify these vector fields with the vectors and matrices
\[
\frac{\partial}{\partial \mathfrak{m}_i} \leftrightarrow e_i\in\mathbb R^d
,\quad i=1,\dots ,d,
\]
\[
\frac{\partial}{\partial \Sigma_{ij}} \leftrightarrow E_{ij}\in \operatorname{Sym}_+(d),\quad i,j=1,\dots d,\, i\le j,
\]
where $(e_i)_i$ is the canonical basis of $\mathbb R^d$ and 
\[
E_{ij}\eqdef
\begin{cases}
    I_{(i,j)},\quad\quad\quad\quad i=j\\
    I_{(i,j)} + I_{(j,i)},\quad i\neq j,
\end{cases}
\]
where $I_{(i,j)}$ denotes the $d\times d$ matrix whose $(i,j)$-$th$ entry is equal to 1, while all the remaining entries are zero.
Therefore, each smooth vector field $A$ on $\mathcal{N}_d$ can be decomposed as
\[
A=\sum_{i=1}^d F_ie_i + \sum_{
\overset{i,j=1}{i\le j}}^d G_{ij}E_{ij}
\]
where $F_i,G_{ij}:\mathcal{N}_d\to \mathbb R$ are smooth functions. We shall call the two terms in this last equation the components of $A$ in the $\mathfrak m$- and $\Sigma$-direction, respectively.  
The space $\mathcal{N}_d$ can be turned into a Riemannian manifold by using the Fisher information as its Riemannian metric. More precisely, we define at an arbitrary point $(\mathfrak m,\Sigma)$ the following Riemannian metric $\mathcal{I}^{F}$
    \begin{equation}\label{eq:FisherRaoMetric}
    \begin{split}
        &\mathcal{I}^{F}\left(\frac{\partial}{\partial \mathfrak m_i},\frac{\partial}{\partial \mathfrak m_j} \right) \equiv \mathcal{I}^{F}(e_i,e_j) \eqdef \Sigma^{ij}, \\%,\quad i,j=1,\dots ,d,\\
        &\mathcal{I}^{F}\left(\frac{\partial}{\partial \mathfrak m_i},\frac{\partial}{\partial \Sigma_{kl}} \right) \equiv \mathcal{I}^{F}(e_i,E_{kl}) \eqdef 0, \\%\quad i,k,l=1,\dots ,d,\,k\le l,\\
        &\mathcal{I}^{F}\left(\frac{\partial}{\partial \Sigma_{ij}},\frac{\partial}{\partial \Sigma_{kl}} \right) \equiv \mathcal{I}^{F}(E_{ij},E_{kl}) \eqdef 
        \frac12\operatorname{tr}(\Sigma^{-1}E_{ij}\Sigma^{-1}E_{kl}) =
        \Sigma^{il}\Sigma^{jk} + \Sigma^{ik}\Sigma^{jl},
        % \quad i,j,k,l=1,\dots ,d,\,i\le j \text{ and }k\le l,\\
    \end{split}
    \end{equation}
    for $i,j,k,l=1,\dots ,d,\,i\le j$ and $k\le l$; where $\Sigma^{-1}=(\Sigma^{ij})_{i,j=1}^d$ denotes the inverse of $\Sigma$, and $\operatorname{tr}$ is the trace operator. From the last equation, by the bilinearity of the metric $\mathcal{I}^F$, we see that
    \begin{equation}\label{eq:FisherRaoMetric2}
    \mathcal{I}^F(X,Y) = \frac{1}{2}\operatorname{tr}(\Sigma^{-1}X\Sigma^{-1}Y), \quad X,Y\in \operatorname{Sym}(d).    
    \end{equation}

Though the Fisher metric is the canonical Riemannian metric on the smooth manifold $\mathcal{N}_d$, in the sense of Chentsov's Theorem \citep{CsiszarSansovTheorem_1984,Dowty_2018__ChentsovExponentialFamilies} from information geometry, its geometry is not amenable to intrinsically averaging in dimension above $1$.  When $d>1$ the Fisher geometry on $\mathcal{N}_d$ exhibits positive sectional curvature along tangent planes in the directions spanned by $\frac{\partial}{\partial\mathfrak{m}_i}$ and $\frac{\partial}{\partial\mathfrak{m}_j}$ for distinct $i,j=1,\dots,d$ (see \citep[Theorem 2.2]{Skovgaard_FisherRaoGeometry_NonDegenreateGaussians__1984}).  These directions of positive sectional curvature do not exist in univariate setting, since no such tangent plane exist when $d=1$, and the only tangent plane at any point is generated by $\frac{\partial}{\partial \mathfrak{m}_1}$ and $\frac{\partial}{\partial \Sigma_{1,1}}$ whose sectional curvature equals $-1/2$; this leads to the well-known identification of the Fisher geometry on $\mathcal{N}_1$ with that of the hyperbolic plane.

As we have learnt above from the results of \cite{Sturm_2003}, lack of globally NPC on the Fisher geometry of $\mathcal{N}_d$ for $d>1$ makes it unsuitable to intrinsic averaging.  Furthermore, there are no known closed-form expressions for the geodesic distance and for the Riemannian exponential map about any point for Fisher geometry on $\mathcal{N}_d$, in dimension larger than $1$.  Both these limitations make the Fisher geometry on $\mathcal{N}_d$ incompatible with the available tools for probability theory on NPC spaces and approximation theory between Riemannian manifold (\cite{kratsios2022universal}). In light of this, we perturb the Fisher metric~\eqref{eq:FisherRaoMetric} as follows
\begin{equation}
\label{eq:Riemannian_metric}
\begin{split}
    &\mathfrak I\left(\frac{\partial}{\partial \mathfrak m_i},\frac{\partial}{\partial \mathfrak m_j} \right) = \boldsymbol{\delta}_{ij}, \\ %,\quad i,j=1,\dots ,d,\\
    &\mathfrak I\left(\frac{\partial}{\partial \mathfrak m_i},\frac{\partial}{\partial \Sigma_{kl}} \right) = 0, \\ %,\quad i,k,l=1,\dots ,d,\,k\le l,\\
    &\mathfrak I\left(\frac{\partial}{\partial \Sigma_{ij}},\frac{\partial}{\partial \Sigma_{kl}} \right) = 
    \frac12\operatorname{tr}(\Sigma^{-1}E_{ij}\Sigma^{-1}E_{kl}) =
    \Sigma^{il}\Sigma^{jk} + \Sigma^{ik}\Sigma^{jl}, %,\quad i,j,k,l=1,\dots ,d,\,i\le j \text{ and }k\le l.\\
\end{split}
\end{equation}
for $i,j,k,l=1,\dots ,d,\,i\le j$ and $k\le l$, where $\boldsymbol{\delta}$ denotes the standard Riemannian metric of $\mathbb{R}^{d}$.  
Besides being a non-positively curved perturbation of the geometry induced by Fisher metric, which also has a closed-form distance function, the Riemannian metric $\mathfrak{J}$ has a second simple geometric interpretation.  Namely, $(\mathcal{N}_d,\mathfrak{J})$ is isometric as a Riemannian manifold to $\mathbb{R}^d\times \operatorname{Sym}_+(d)$ equipped with the product between the Euclidean metric $\boldsymbol{\delta}$ on $\mathbb{R}^d$ and the well-studied affine-invariant metric on $\operatorname{Sym}_+(d)$, see~\cite{SmithIEEETransSigProces_2005_CovarianceIntrinsicCramerRaoBounds,pennec2006riemannian,SchirattiAllassoniereColliotDurrleman_2017_JMLR_PSD,brooks2019riemannian,FCM_2022_AcceleratedFirstOrderMethodGDManifolds}. Some of the key geometric properties of $(\mathcal{N}_d, \mathfrak{I})$ can now be recorded.  
Below, $d_{\mathfrak I}$ will denote the distance function induced by $\mathfrak I$.

\begin{proposition}[{The geometry of $(\mathcal{N}_d, \mathfrak{I})$}]
\label{prop:NPC_FisherRao}
For any $d\ge 1$, $(\mathcal{N}_d, d_{\mathfrak{I}})$ is a complete, path-connected and simply connected, Riemannian manifold with sectional curvatures in $[-1/2,0]$. Therefore, $(\mathcal{N}_d,d_{\mathfrak{I}})$ is a global NPC space, and for any $\mathcal{N}_d(\mathfrak{m}_0,\Sigma_0),\mathcal{N}_d(\mathfrak{m}_1, \Sigma_1)\in \mathcal{N}_d$, $d_{\mathfrak{I}}$ admits the closed-form expression
\begin{equation}
\label{eq:simple_Closedform}
        d_{\mathfrak{I}}\big(\mathcal{N}_d(\mathfrak m_0,\Sigma_0),
        \mathcal{N}_d(\mathfrak m_1,\Sigma_1)\big)
    = 
        \left[\norm{\mathfrak m_0 - \mathfrak m_1 }^2_2 + 
        \frac12\sum_{i=1}^{d}\, \ln(\lambda_i)^2\right]^{1/2}
,
\end{equation}
where $0<\lambda_1\le \dots \le \lambda_d$ are the eigenvalues of $\Sigma_0^{-1}\Sigma_1$. 

For any $\mathfrak{m}\in\mathbb{R}^d$, $\Sigma\in \operatorname{Sym}_+(d)$, and any $(\tilde{\mathfrak{m}},X)\in T_{\mathcal{N}_d(\mathfrak{m},
\Sigma
)} \mathcal N_d\cong \mathbb{R}^d\times \operatorname{Sym}(d)$ the Riemannian exponential map is given by
\begin{equation}
\label{prop:NPC_FisherRao__ExponentialNiceForm}
        \operatorname{Exp}_{\mathcal{N}_d(\mathfrak{m},
        \Sigma
        )}(\tilde{\mathfrak{m}},X)
    =
        \mathcal{N}\left(
                \mathfrak{m}+\tilde{\mathfrak{m}}
            ,
                \Sigma^{1/2}
                \operatorname{exp}\big(
                    \Sigma^{-1/2}
                    X
                    \Sigma^{-1/2}
                \big)
                \Sigma^{1/2}
        \right)
    ,
\end{equation}
where $T_{\mathcal{N}_d(\mathfrak{m},
\Sigma
)}\mathcal{N}_d$ denotes the tangent space of $\mathcal N_d$ at the point $\mathcal{N}_d(\mathfrak{m},
\Sigma)$. 
\end{proposition}

\begin{proof}
See Appendix \ref{app:proof_proposition_NPC_FisherRao}.
\end{proof}

\subsubsection{{Comparison with other geometries on \texorpdfstring{$\mathcal{N}_d$}{Nd}}}
We may further motivate the proposed Riemannian geometry on $\mathcal{N}_d$ by contrasting it with other standard geometries on the set $\mathcal{N}_d$, occurring in optimal transport and information geometry.  Our discussion is summarized concisely in Table~\ref{tab:Non-Singular_Gaussian}. 

\begin{table}[htp!]
\caption{Comparison of topological, metric, and computational properties between different geometries of non-singular Gaussian measures from information geometry and optimal transport.}
\label{tab:Non-Singular_Gaussian}
\centering
\ra{1.3}
\resizebox{\columnwidth}{!}{%
\begin{tabular}{c|p{2cm}p{3.5cm}p{2.5cm}p{3cm}l}
    \textbf{Geometry} &  \textbf{Complete} & \textbf{NPC (Barycenters)}  &  \textbf{Closed-Form Dist.} & \textbf{Isometric Copies $(\mathcal{N}_d^{\mathfrak{m}},\mathcal{I})$} & \textbf{Reference}\\
    \midrule
    $(\mathcal{N}_d,d_{\mathfrak{J}})$ & $\good$ & $\good$ & $\good$ & $\good$ & Equation~\eqref{eq:Riemannian_metric}\\%\textbf{This Paper} \\
    \arrayrulecolor{lightgray}\hline
    $(\mathcal{N}_d,d_{\mathcal{I}})$ & $\good$ & $\bad$ ($\good$ $d=1$) &  $\bad$ & $\good$ &  \cite{Skovgaard_FisherRaoGeometry_NonDegenreateGaussians__1984} \\
    $(\mathcal{N}_d,\mathcal{W}_2)$ & $\bad$ & ? & $\good$ & $\bad$ & \cite{takatsu2011wasserstein} \\
    $(N_d,g_{SL})$ & $\good$ & $\good$ & $\good$ & $\bad$ & 
    \cite{GeometryOfMultivariateNormal_Lie_Canada}\\
    \bottomrule
\end{tabular}
}
\end{table}

In what follows, we linearly identify the spaces via the symmetrization map $\operatorname{sym}$ defined by
    \begin{equation}
    \label{eq:vecsym_def}           \operatorname{sym}:\mathbb{R}^{d(d+1)/2} \rightarrow \operatorname{Sym}(d)
            ,\quad
            v  \mapsto
            \begin{pmatrix}
                v_1 & v_2 & \dots & v_{d} \\
                v_2 & v_{d+1} & & . \\
                \vdots & & \ddots & \vdots \\
                v_d & \dots & & v_{d(d+1)/2}
            \end{pmatrix}
    \end{equation}
    whose inverse is often denoted by $\operatorname{vec}$ (see \cite[Section 10.2.2]{CookingMatrices_v2105}).  
We observe that the linear identification is not an isometry, but it is a bi-Lipschitz map, when $\operatorname{Sym}(d)$ is equipped with Fr\"{o}benius inner-product $\langle \Sigma_1,\Sigma_2\rangle\eqdef \operatorname{tr}(\Sigma_1\Sigma_2)$ and $\mathbb{R}^{d(d+1)/2}$.

Table~\ref{tab:Non-Singular_Gaussian} compares the Riemannian geometry $(\mathcal{N}_d,\mathfrak{J})$ with alternative geometries found in the literature. Specifically, it contrasts $(\mathcal{N}_d,\mathfrak{J})$ with the standard Fisher information geometry as studied in \cite{Skovgaard_FisherRaoGeometry_NonDegenreateGaussians__1984}, the Riemannian geometry inducing the $2$-Wasserstein metric ($\mathcal{W}_2$) on $\mathcal{N}_d$ as presented in \cite{takatsu2011wasserstein}, and a non-positive curvature geometry on $\mathcal{N}_d$, denoted $g_{SL}$, constructed in \cite{GeometryOfMultivariateNormal_Lie_Canada} by identifying $\mathcal{N}_d$ with a subset of a special linear group, modulo the action of a special orthogonal group. The columns of Table~\ref{tab:Non-Singular_Gaussian} should be self-explanatory apart from the penultimate column. Here we say that the geometry admits isometric copies of $(\mathcal{N}_d,\mathcal{I})$ if for each $\mathfrak{m}\in \mathbb{R}^d$, the submanifold  $\mathcal{N}^{\mathfrak{m}}_d \subset \mathcal{N}_d$ whose points have mean $\mathfrak{m}$ is isometric to $(\mathcal{N}_d^{\mathfrak{m}},\mathcal{I})$.

Finally, we remark that the the manifolds $(\mathcal{N}_d,W_2)$, $(\mathcal{N}_d,\mathfrak{I})$, and $(\mathcal{N}_d,\mathcal{J})$ are diffeomorphic; in particular, their topologies coincide.
Nevertheless, we emphasize that all three structures are different as Riemannian manifolds and as metric spaces.

\section{Gaussian random projections}
\label{s:Projection}

The law of the process $X_{\cdot}$ in~\eqref{eq:Volterra_X}, conditioned on any realized path up to a given time $t\in \mathbb{N}$, is an arbitrary probability distribution on $\mathbb{R}^{d}$ with a priori no favourable 
% \giulia{Maybe we should explain a little bit?} 
structure. For instance, the space of probability distribution on $\mathbb{R}^d$, e.g.\ with the L\'{e}vy-Prokhorov metric, see e.g.~\citep[Theorem 8.3.2]{bogachev2007measure__BookII}, is neither a finite-dimensional as a topological space nor is it a vector space (even though it is convex).
This lack of structure implies that its elements cannot be approximated by simpler probability measures, such as finitely supported measures, without facing the curse of dimensionality (see, e.g.~\citep[Theorem 6.2]{GradLuschgy_QuantizationProbMeasure_2000}). 
\paragraph{Intuitive idea} However, if we had access to the unobserved realized path of the ``stochastic factor process'' $\mathbf{S}_{[0:t]}$, we could leverage the fact that $X_{t+1}$ is Gaussian {when conditioned on the realized path $X_{[0:t]}$ and belongs to $\mathcal{N}_{d}$, the smooth manifold of non-singular $d$-dimensional Gaussian measures. This motivates the family of probability measures 
$\{\Lambda_{x_{[0:t]},s_{[0:t]}};\, x_{[0:t]}\in \mathbb R^{(1+t)d},\, s_{[0:t]}\in \operatorname{Sym}(d)^{1+t}
\}\subset \mathcal{N}_{d}$, defined by
\begin{equation}
\label{eq:Q_prederandomized}
    \Lambda_{x_{[0:t]},s_{[0:t]}}[\cdot] \eqdef 
        \mathbb{P}[X_{t+1}\in \cdot \vert X_{[0:t]}=x_{[0:t]},\,\mathbf{S}_{[0:t]}=s_{[0:t]}\big],\quad x_{[0:t]}\in \mathbb R^{(1+t)d},\, s_{[0:t]}\in \operatorname{Sym}(d)^{1+t}
\end{equation}
but since the process $\mathbf{S}_{\cdot}$ is latent, we need to reintroduce its ``random effects'' into the family of probability measures $\Lambda_{x_{[0:t]},s_{[0:t]}}$. Traditionally, this is achieved by integrating against the law of $\mathbf{S}_{[0:t]}$. Informally, we instead reinsert the observed path $\mathbf{S}_{[0:t]}$ as the conditional state $s_{[0:t]}$, \textit{i.e.}:
\begin{equation}
\label{eq:randomized_encoding__informal}
            \mathbb{Q}_{x_{[0:t]}}(\omega) 
        \overset{\text{informal}}{=}
            \Lambda_{x_{[0:t]},\mathbf{S}_{[0:t]}(\omega)},\quad \omega\in \Omega 
 .
\end{equation}
A rigorous construction of~\eqref{eq:randomized_encoding__informal}, formalized below in Definition~\ref{defn:Projection}, hinges on our Riemannian geometry on $\mathcal{N}_d$ and it enjoys Lipschitz stability, meaning that it is a Lipschitz function of path realized by the process $X_{\cdot}$.  
In some cases, the dependence of $\mathbb{Q}_{x_{[0:t]}}$ on the realized path of $X_{\cdot}$ is even smooth, meaning that deep learning model can approximate it with a feasible number of parameters.  This is generally not true with geometric deep learning models on non-smooth metric spaces, see \cite{Acciaio2022_GHT}. 
The results of \cite{Sturm_2003} guarantee that there is a unique point $\Pi_{x_{[0:t]}}$ in $\mathcal{N}_{d}$ which best represents the intrinsic average behaviour of $\mathbb{Q}_{x_{[0:t]}}$
\begin{equation}
\label{eq:projection}
    \Pi_{x_{[0:t]}}
    \eqdef 
        \operatorname{argmin}_{u\in \mathcal{N}_{d}}\, 
            \int_{v\in \mathcal{N}_{d}}
                d_{\mathfrak{J}}^2(u,v) 
            \text{Law}(\mathbb{Q}_{x_{[0:t]}}\in dv)
    ,
\end{equation}
where $d_{\mathfrak{J}}$ is the \textit{geodesic distance} induced by our perturbation of the Fisher Riemannian metric on $\mathcal{N}_{d}$. 
Unlike information projections, the map defined by~\eqref{eq:projection} enjoys the following stability properties and can even be smooth in certain cases.

\begin{result_intro}[Gaussian projections Theorem]
\label{informal_summary_A}
Under mild conditions, if the $\operatorname{Drift}$ and $\operatorname{Diffusion}$ maps in~\eqref{eq:Volterra_X} are $L$-Lipschitz, then the map $x_{[0:t]}\mapsto\Pi_{x_{[0:t]}}$ is well-defined, $\mathcal{O}(L)$-Lipschitz, and it solves the following minimization problem 
\[
        \Pi_{x_{[0:t]}}
    =
        \underset{\mu \in \mathcal{N}_d}{\operatorname{argmin}}
        \,
            \mathbb{E}
            \big[
                d^2_\mathfrak J
                (\mu, 
                \mathbb{Q}_{x_{[0:t]}}
                )
            \big]   
\]
Furthermore, if the latent factor process $\mathbf{S}_{\cdot}=0$ then
\[
        \Pi_{x_{[0:t]}}
    =
        \mathbb{P}[X_{t+1}\in \cdot \vert X_{[0:t]}=x_{[0:t]}]
\]
for each time $t\in \mathbb{N}_+$ and every realized path $x_{[0:t]}\in \mathbb{R}^{(1+t)d}$.
\end{result_intro}

We now formalize our notion of projection of the conditional law of the stochastic process~\eqref{eq:Volterra_X}, called a \textit{Gaussian random projection}.  The terminology should not be understood in the sense of the Johnson-Lindenstrauss lemma, see \citep[Theorem 12.2.1]{MatousekWonderfulBook_2002}, but rather in the sense of \cite{OhtaStochasticRetracts_2009}. That is, as a Lipschitz measure-valued function with the property that a certain class of inputs are fixed in certain sense (see Corollary~\ref{cor:DeterministicHighRegularity}).

\subsection{Encoding the conditional distribution of a Volterra process in \texorpdfstring{$\mathcal{N}_d$}{Nd}}}
\label{s:Projection__ss:EncodingAsMeasureValuedProcess}
Our argument centers around the geometry of the regular conditional distribution (RCD) function sending any $(x_{[0:t]},s_{[0:t]})$ to 
\[
    \mathbb{P}[X_{t+1}\in \cdot \vert X_{[0:t]}=x_{[0:t]},\mathbf{S}_{[0:t]}=s_{[0:t]}]
.
\]
Our first observation is to note that the RCD's image lies in $\mathcal{N}_{d}$. 

The map $\tilde{\phi}:\mathcal{N}_d \rightarrow \mathbb{R}^{d}\times \mathbb{R}^{d(d+1)/2}$, defined by $\mathcal{N}_d(\mathfrak m, \Sigma) \stackrel{\tilde{\phi}}{\mapsto} (\mathfrak m, \operatorname{vec}(\log (\Sigma)))$, is a smooth chart.  This is because it is the composition of the maps $1_{\mathbb{R}^d}\times \operatorname{vec}$, 
$\mathbb{R}^d\times \operatorname{Sym}(d)$
, $\log$ \footnote{See \citep[Proposition 3.3.5]{HilgertNeeb_2012_StructureGeometryLie}.} 
, $1_{\mathbb{R}^d}\times \log$ 
and $\phi$, defined in~\eqref{eq:identification_Nd_Vectors}, each of which are diffeomorphisms.
Thus, 
\begin{equation}
\label{eq:chart}
\begin{aligned}
    \psi:
    \mathbb{R}^{d(d+3)/2}& \rightarrow \mathcal{N}_{d}\\
    (u,v) & \mapsto \tilde{\phi}^{-1}(u,v) = \mathcal{N}_{d}(u,\exp(\operatorname{sym}(v)))
\end{aligned}
\end{equation}
defines a smooth global chart on $\mathcal{N}_{d}$. For every $t \in \mathbb{N}_+$, the maps 
\[
\begin{aligned}
    \mathbb{R}^{(1+t)d}\ni x_{[0:t]} & \mapsto \operatorname{Drift}(t,x_{[0:t]})\in \mathbb{R}^{d}
\\
\mathbb{R}^{(1+t)d}\times \operatorname{Sym}(d)^{(1+t)}\ni (x_{[0:t]},s_{[0:t]})& \mapsto \operatorname{Diffusion}(t,x_{[0:t]},s_{[0:t]})\in \operatorname{Sym}_+(d)
\end{aligned}
\]
are locally Lipschitz continuous.  
Thus, for every $t \in \mathbb{N}_+$, the map
\[
\mathbb{R}^{(1+t)d}\times \mathbb R^{(1+t)d(d+1)/2}
\ni (x_{[0:t]},y_{[0:t]}) \mapsto \operatorname{Diffusion}(t,x_{[0:t]},\operatorname{sym}( y_{[0:t]}))\in \operatorname{Sym}_+(d),
\]
where $\operatorname{sym}$ is applied componentwise,
is locally Lipschitz continuous.
Thus, for every $t \in \mathbb{N}_+$ the map $\Psi_t: \mathbb{R}^{(1+t)d}\times \mathbb{R}^{(1+t)\,d(d+1)/2}  \rightarrow \mathcal{N}_{d}$ defined for any $x_{[0:t]}\in \mathbb{R}^{(1+t)d}$ and $y_{[0:t]}\in \mathbb{R}^{(1+t)d(d+1)/2}$ by
\begin{equation}
\label{eq:Psi_t_Definition}
\begin{aligned}
        \Psi_t(x_{[0:t]},y_{[0:t]}) 
    & 
    \eqdef
        \psi\big(
                x_t + \operatorname{Drift}(t,x_{[0:t]})
            ,
                    \operatorname{vec}\circ\log\circ
                    \operatorname{Diffusion}(t,x_{[0:t]},\operatorname{sym}(
                    y_{[0:t]}))
        \big)\\
    & = 
        \mathcal{N}_{d}\left(
            x_t + \operatorname{Drift}(t,x_{[0:t]})
            ,
            \exp\left(
                \sum_{r=0}^t\kappa(t,r)\left[\,\sigma(t,x_r) + \operatorname{sym}(y_r)\right]
            \right)
        \right)
\end{aligned}
\end{equation}
must also be locally-Lipschitz when 
$\mathcal{N}_d$ is equipped with the Riemannian distance function $d_{\mathfrak{J}}$. Note that $\operatorname{Diffusion}(t,x_{[0:t]},\operatorname{sym}(s_{[0:t]}))$ is positive definite.  Therefore we can assume that its eigenvalues are such that $0<\lambda_{1}\le \dots \le\lambda_{d}$.
Furthermore, if $\mu$ and $\sigma$ are smooth, then each $\Psi_t$ is also smooth by composition.
Consequently, we may rigorously define the $\mathcal{N}_{d}$-valued stochastic process as the composition of $x_{[0:t]}$ and $\mathbf{S}_{[0:t]}$ and the locally Lipschitz continuous map $\Psi_t$, namely $\Psi_t( x_{[0:t]}, \operatorname{vec}\circ\mathbf{S}_{[0:t]})$.

We note that $\|e^A\|_{op}\le e^{\|A\|_{op}}$ for any $d\times d$ matrix $A$; hence, for some universal constants $c_1,c_2>0$ depending only $d$, $\|e^A\|_{F}\le c_1e^{c_2\|A\|_{F}}$; therefore the boundedness of $S_{\cdot}$ and the image of $\sigma$ in  Assumption~\ref{ass:uniformboundedness__SProcecess} imply that the law of $\Psi_t(x_{[0:t]},\operatorname{vec}\circ\mathbf{S}_{[0:t]})$ is compactly supported on $\mathcal{N}_d$.  Therefore, $\Psi_t(x_{[0:t]},\operatorname{vec}\circ \mathbf{S}_{[0:t]})$ is integrable over the global NPC space $(\mathcal{N}_d,d_{\mathfrak{J}})$, in the sense that it belongs to $L^1(\Omega,\mathcal{N}_d)$, and the barycenter of its law is well-defined .

\begin{definition}[Gaussian random projection of $X_{\cdot}$]
\label{defn:Projection}
Given a Volterra process $X_{\cdot}$ in~\eqref{eq:VolterraProcess_Dynamics}, we define the \textit{Gaussian Random projection} of $X_{\cdot}$ as the map, sending any $x_{[0:t]}\in\R^{(1+t)d}$ to the measure $\Pi_{x_{[0:t]}}\in\mathcal{N}_{d}$ defined by
\begin{equation}
\label{rmk: barycenter of a Dirac}
\begin{aligned}
        \Pi_{x_{[0:t]}}
    \eqdef
        \beta\left(
            \operatorname{Law}\left(
                \Psi_t\left(x_{[0:t]},\operatorname{vec} ( \mathbf{S}_{[0:t]}) 
            \right)
        \right)\right) \in \mathcal{N}_d
\end{aligned}
\end{equation}
where $\beta$ is given in Proposition \ref{prop:sturmtwo}. 
For the sake of brevity, we will henceforth write $\Pi_{x_{[0:t]}} = \beta(\mathbb{Q}_{x_{[0:t]}})$, where $\mathbb{Q}_{x_{[0:t]}}\eqdef \operatorname{Law}\left(\Psi_t(x_{[0:t]},\operatorname{vec}( \mathbf{S}_{[0:t]}))\right)$. 
\end{definition}
In particular, notice that if we have $ \mathbf{S}_{[0:t]}=s_{[0:t]}$, then 
\begin{equation*}
    \Pi_{x_{[0:t]}} = \mathcal{N}_{d}(x_t +\operatorname{Drift}(t,x_{[0:t]}), \operatorname{Diffusion}(t, x_{[0:t]}, s_{[0:t]})^2) 
\end{equation*}
where $\operatorname{Diffusion}(t, x_{[0:t]}, s_{[0:t]})^2=\exp\left(
                \sum_{r=0}^t\kappa(t,r)\left[\,\sigma(t,x_r) + s_r\right]
            \right)$.
We now have the following.
\begin{theorem}[{Optimality and Lipschitz Stability of the Gaussian Random Projection}]
\label{thrm:Projection_Results}
Fix $t \in \mathbb{N}$ and a compact $\mathcal{K}_{[0:t]}\subseteq \mathbb{R}^{(1+t)\,d}$.  If~\eqref{eq:VolterraProcess_Dynamics} and Assumption~\ref{ass:uniformboundedness__SProcecess} hold then, $\Pi:x_{[0:t]}\mapsto \beta(\mathbb{Q}_{x_{[0:t]}})$ is well-defined and Lipschitz and
\[
\Pi_{x_{[0:t]}}
    =
        \underset{
            {u\in \mathcal{N}_{d}}
        }{
            \operatorname{argmin}
        }\,
            \int_{v\in \mathcal{N}_{d}}
                d_{\mathfrak{J}}^2(u,v) 
            \mathbb{Q}_{x_{[0:t]}}( dv).
\]
Furthermore, the Lipschitz constant of $\Pi_{\cdot}$ 
depends only on $t$ and $R$, as defined in Assumption~\ref{ass:uniformboundedness__SProcecess}.
\end{theorem}

\begin{proof}
As regards the Lipschitz part, we refer to Proposition \ref{lem:Stability_Estimates}. The remainder of the proof follows from the discussion near~\eqref{rmk: barycenter of a Dirac}. 
\end{proof}

The smooth structure on $\mathcal{N}_d$ has the advantage of granting meaning to the smoothness of the map $\Pi_{\cdot}$ when $\mu$ and $\sigma$ are smooth and $\mbox{\textbf{S}}_{[0:t]}$ is almost surely constant.   
Smoothness is beneficial from an approximation perspective since neural networks can effectively approximate smooth Lipschitz functions between Riemannian manifolds. This is demonstrated in Theorem~\ref{theorem:optimal_GDN_Rates__ReLUActivation} below. In contrast, even in the Euclidean setting, non-smooth Lipschitz functions between Euclidean spaces on simple, compact subsets cannot be accurately approximated by feasible neural networks, as shown in \citep[Theorem 5]{yarotsky2017error}.

The following Corollary provides a natural sufficient condition for the smoothness of $\Pi_{\cdot}$.  Namely, if the maps $\mu$ and $\sigma$, which define the drift and diffusion maps of the Volterra process, are smooth, and that $\mathbf{S}_{\cdot}$ is constantly $0$ then $\Pi_{\cdot}$ is also smooth.

\begin{proposition}[{Smoothness}]
\label{cor:DeterministicHighRegularity} 
Suppose that $\mathbf{S}_t=0$ for every $t\in \mathbb{N}$, Assumption~\ref{ass:uniformboundedness__SProcecess} holds and that $\mu$ and $\sigma$ are of class $C^{\infty}$.  Then for every $t\in \mathbb{N}$, the following map is $C^{\infty}$  
    \[ 
    \begin{aligned}
        \Pi_{\cdot}:\,
            (\mathbb{R}^{(1+t)d},\|\cdot\|) 
        & \rightarrow 
            (\mathcal{N}_{d},d_{\mathfrak{J}}) 
        \\
            x_{[0:t]} 
        & \mapsto \Pi_{x_{[0:t]}}.
    \end{aligned}
    \]
\end{proposition}
\begin{proof}
    See Appendix~\ref{app:proof_thrm:VanishingMemoryProperty_Qx}. 
\end{proof}

Next, we study how the Volterra kernel dictates the rate at which the Gaussian random projection forgets the distant history of any realized path of a Volterra process.

\subsection{Memory decay of 
\texorpdfstring{$\Pi_{\cdot}$}{the Gaussian Projection operation} given a decaying Volterra kernel}
\label{s:Projection__ss:MemoryDecay}
In order for the map $x_{[0:t]}\mapsto \Pi_{x_{[0:t]}}$ to be approximable by any model which only processes the path realized by $X_{\cdot}$ up to the current time, we need that the stochastic Volterra process does not exhibit a persistent memory of its distant past. As one may expect, a stochastic Volterra process exhibits no persistent long-term memory if its Volterra kernel does not place large weights on long-past states realized by the process $X_{\cdot}$.  
Similar phenomena have been observed in the RNN literature when approximating linear state-space systems with exponentially decaying impulse response in \citep{RecepHutterHelmut_2022_AHA__RNNs} or in \citep[Theorem 11]{li2022approximation} when quantifying the approximability of linear functionals by linear continuous-time RNNs. 
Our next result shows that if the Volterra kernels decay at least polynomially, then one has a quantitative analogue of the fading memory property in reservoir computing \citep[Definition 3]{GrigoryevaOrtegaUniversalDiscretTime2018_JMLR}, and of the approximable complexity property of \cite[Definition 4.9]{Acciaio2022_GHT}.  

\begin{theorem}[Memory Decay Rate of Gaussian Projection of Volterra Process]
\label{thrm:VanishingMemoryProperty_Qx}
\hfill\\
Suppose 
Assumption~\ref{ass:uniformboundedness__SProcecess} holds
and the Volterra-kernel $\kappa$ either satisfies~\eqref{exp_decay} or~\eqref{pol_decay}.
For every $T\in \mathbb{N}_+$, each compact $\mathcal{K}_{[0:T]}\subseteq \mathbb{R}^{(1+T)d}$, there is a constant $C_{M,R,\mu}>0$ such that for every pair of integers $\memory$ and $t$, with $0\le \memory < t\le T$, there is a Lipschitz function $\tilde{f}^{(t,\memory)}:(\mathbb{R}^{(1+\memory)d},\|\cdot\|)\rightarrow (\mathcal{N}_d,d_{\mathfrak{J}})$ satisfying
\[
        d_{\mathfrak{J}}\big(
            \Pi_{x_{[0:t]}},
            \tilde{f}^{(t,\memory)}(x_{[t-\memory:t]})
        \big)
    \le 
        C_{M,R,\mu}
        \,
        \operatorname{diam}(\mathcal{K}_{[0:T]})
        \,
        \begin{cases}
            \frac{C\alpha}{\alpha-1}\,(\alpha^t-\alpha^{\memory})
            &
            \mbox{if (i) holds}
            \\
            C\,(\memory+1)^{\alpha}\,(t-\memory) 
            &
            \mbox{if (ii) holds}
        \end{cases}
\]
for every $x_{[0:T]}\in \mathcal{K}_{[0:T]}$. 
\end{theorem}
\begin{proof}
    See Appendix \ref{app:proof_thrm:VanishingMemoryProperty_Qx}
\end{proof}

Our next objective is to approximate the map $x_{[0:T]}\mapsto (\Pi_{x_{[0:t]}})_{t=0}^T$ while respecting the forward flow of information in time. We will accomplish this by developing a more general approximation theory for causal maps between Riemannian manifolds which are global NPC spaces.

\section{Universal approximation of manifold-valued processes}
\label{s:Approximation}

In this section, we introduce our \textit{sequential} geometric deep learning model, called the HGN and defined below in~\eqref{eq:GCN_Summary}.  We then use it to prove a universal approximation theorem guaranteeing that the map $x_{[0:t]}\to \Pi_{x_{[0:t]}}$ can be causally (in time) approximated to arbitrary precision if the memory of the Volterra process decays at-least polynomially.  

\begin{result_intro}[HGNs can approximate projected RCDs]
\label{thrm:IntroVersionapprox_HGCNN}
For every time-horizon $T$, and every suitable process $X_{\cdot}$ defined by~\eqref{eq:Volterra_X}, for each $\varepsilon>0$ and every compact set of inputs $\mathcal{K}_{[0:T]}\subseteq \mathbb{R}^{(1+T)\,d}$, there exists a HGN 
$F:(\mathbb{R}^{d})^T\rightarrow \mathcal{M}^T$, as in Figure~\ref{fig:Model}, satisfying
\[
    \underset{x_{[0:T]\in \mathcal{K}_{[0:T]}}}{\max}
    \underset{t=T-\memory,\dots,T}{\max}
        \,\,
        d_g\big(
        % (
        \Pi_{x_{[0:t]}},
        % )_{t=0}^T,
            F(x_{\cdot})_t
        \big)
    <
        \varepsilon
    .
\]
The number of parameters and structure of the HGN depends quantitatively on $\varepsilon,\memory,\mathcal{K}_{[0:T]}$, and on the $\operatorname{Drift}$ and $\operatorname{Diffusion}$, and on the Volterra kernel $\kappa$, where $\memory\in \mathbb{N}$ is the \textit{history} of the HGN.
\end{result_intro}
The main result of this section (Theorem~\ref{thrm:approx_HGCNN}) is a significant generalization of Informal Result~\ref{thrm:IntroVersionapprox_HGCNN}, as it applies to infinite-memory dynamical systems between any pair of well-behaved non-positively curved complete Riemannian manifolds.  We now present our main approximation theoretic results.  

\subsection{Standing Assumptions for Universal Approximation Results}
\label{s:Approximation__ss:SettingApprox}
In this section, we present our model and build on the theory of causal maps, as formulated in~\cite{Acciaio2022_GHT}, between a source and target connected Riemannian manifolds $(\mathcal{N},h)$ and $(\mathcal{M},g)$.  
\begin{assumption}[Regularity of input and output manifolds]
\label{ass:NPC}
Consider two $C^\infty$ Riemannian manifolds $(\mathcal{N},h)$ and $(\mathcal{M},g)$ of dimensions $D\in\mathbb N$ and $d\in\mathbb N$ respectively. We assume that:
\begin{enumerate}
    \item[\textit{(i)}] $(\mathcal{N},h)$ and $(\mathcal{M},g)$ are geodesically complete and simply connected.
    \item[(ii)] Their sectional curvatures are bounded in $[-\kappa,0]$ for some $\kappa>0$.
\end{enumerate}
\end{assumption}
We denote the geodesic distance on these spaces respectively by $d_{h}$ and $d_g$.  Under Assumption~\ref{ass:NPC}, $(\mathcal{N},d_{h})$ and $(\mathcal{M},d_g)$ are complete separable metrics spaces of NPC in the sense of \cite{Sturm_2003}; see Proposition \ref{prop:sturmone}. 
\begin{example}
The Euclidean space $(\mathbb{R}^{d},\boldsymbol{\delta})$ satisfies Assumption~\ref{ass:NPC} with $\kappa=0$.
\end{example}
\begin{example}
By Proposition~\ref{prop:NPC_FisherRao}, $(\mathcal{N}_d,h)$ satisfies Assumption~\ref{ass:NPC} with $\kappa=-1/2$.
\end{example}

\subsection{Static Case}
\label{s:Approximation__ss:Static}
We begin by defining the static version of our main (dynamical) model, in which time is not yet considered.  This geometric deep-learning model extends the geometric deep network of \cite{kratsios2022universal}, thus we use the same name.  From the technical standpoint, our main approximation theorem for the static case (Theorem~\ref{theorem:optimal_GDN_Rates__ReLUActivation}) is a significant technical improvement of the main result of \cite{kratsios2021universal} as it has optimal rates, matching the optimal bounds of \cite{ShenYangHaizhaoZhang_OptimalReLU_2022} in the classical Euclidean setting when the target function is Lipschitz and/or smooth; this is an exponential improvement of the guarantees in \cite{kratsios2021universal} for the $\operatorname{ReLU}$ case.  Moreover, we obtain probabilistic approximation guarantees for Borel measurable and for general functions being approximated on finite subsets of the input space.  The static case will serve as the first main technical tool for our main approximation theorem for the dynamic case, (Theorem~\ref{thrm:approx_HGCNN}),

\begin{figure}[H]
    \centering
  \includegraphics[width=1\linewidth]{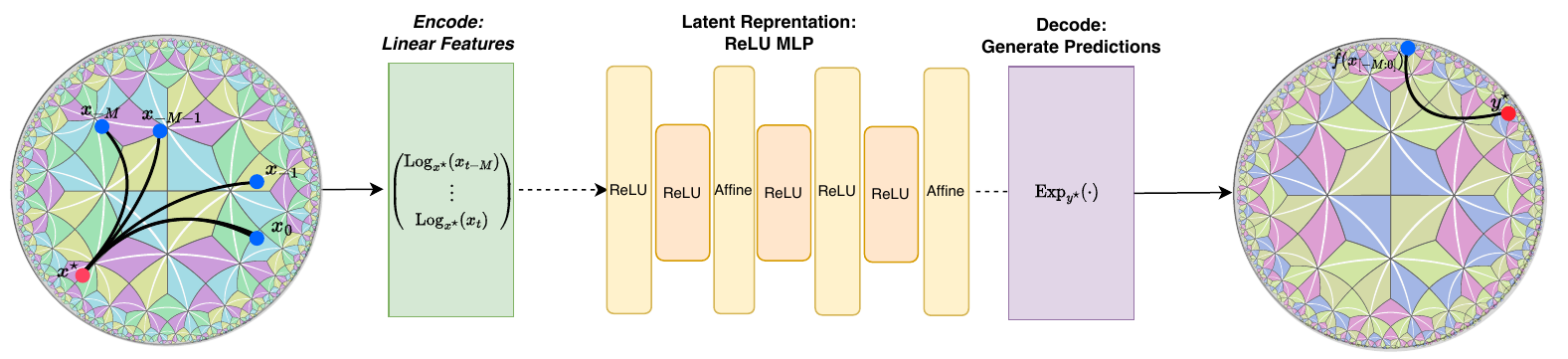}
    \caption{\textbf{The GDN Model:}~%
    GDN model process an input in $x_{[-M:0]}\in \mathcal{N}^{1+\memory}$, interpreted as sequential points $x_{-M},\dots,x_0$ inputs in $\mathcal{M}$, in three steps: an encoding, transformation, and decoding phase.  First, it linearized (purple) the inputs in $\mathcal{N}^{1+\memory}$ along products of geodesics emanating from a set of reference points $x_0^{\star},\dots,x^{\star}_M$ in $\mathcal{N}^{1+\memory}$. 
    It then transforms the linearized features and maps them to a vector $v$ in the tangent space of $\mathcal{M}$ using a standard ReLU-MLP (yellow).  
    In the decoding phase (green), the model maps $v$ to a point $\hat{f}(x_{[-M:0]})$ on $\mathcal{M}$ by travelling geodesics in $\mathcal{M}$ emanating from a reference point $y^{\star}$ therein with initial velocity $v$.}
    \label{fig:GCNN}
\end{figure}

\noindent 
We now formalize the static version of our model, illustrated in Figure~\ref{fig:GCNN}, and extend the geometric deep network (GDN) of \cite{kratsios2022universal}.
Fix $J,\memory \in \mathbb{N}$, with $J>0$, and a multi-index $[\mathbf{d}]  \eqdef  (d_0, \ldots, d_{J})$ with $d_0=(1+\memory)D$ and $d_J=d$, and let $
P([\mathbf{d}])  \eqdef  
%%% Exp and Log
(1+\memory)D
+
d
+
%%% MLP
\sum_{j=0}^J\,
d_{j+1}(1+d_j)
$. Weights, biases, and base-point parameters are encoded as a unique parameter $\theta \in \mathbb{R}^{P([\mathbf{d}])}$ with 
\begin{equation}\label{eq:identification}
    % \mathbb{R}^{P([\mathbf{d}])}\,\reflectbox{$\in$}
    %\,
    \theta
    \,\,
    \leftrightarrow
    \,\,
    \Big(
    %% MLP
        (A^{(j)},b^{(j)})_{j=0}^J
    ,
    %% Expansion Points
    %%% Exp
        (\tilde{x}_m)_{m=0}^{\memory}
    ,
    %%% Log
        \tilde{y}
    \Big)
\end{equation}
where for $j=0,\dots,J$, $A^{(j)}\in \mathbb{R}^{d_{j+1}\times d_j}$, $b^{(j)}\in \mathbb{R}^{d_{j+1}}$, $\tilde{y}\in \mathbb{R}^d$, and $\tilde{x}_m\in \mathbb{R}^D$ for $m=0,\dots,\memory$.
We henceforth denote componentwise composition of the $\operatorname{ReLU}\eqdef \max\{0,\cdot\}$ activation function with any vector $u\in\mathbb{R}^n$ for any $n\in \mathbb{N}_+$, by $\operatorname{ReLU}\bullet x \eqdef \big(\operatorname{ReLU}(x_i)\big)_{i=1}^n$. 
Given $\memory \in \mathbb{N}$, a GDN model is a (parameterized) map $f_{\theta}:
\mathbb{R}^{P([\mathbf{d}])}
\times 
\mathcal{N}^{1+\memory}\rightarrow \mathcal{M}$ sending any $(\theta,x_{[0:\memory]})\in \mathbb{R}^{P([\mathbf{d}])}\times \mathcal{N}^{1+\memory}$ to 
\begin{equation}
\tag{Static}
\label{eq:GCN_Summary}
\begin{aligned}
f(\theta,x_{[0:\memory]})\eqdef 
f_{\theta}(x_{[0:\memory]}) & = \operatorname{Exp}^g_{\bar{y}}\circ \iota^{-1}_{\mathcal{M},\bar{y}}(A^{(J)}x^{(J)} + b^{(J)})
\\
x^{(j+1)}& \eqdef 
            \operatorname{ReLU}\bullet
            (A^{(j)} x^{(j)} + b^{(j)})\quad\text{for\,\,}j= 0, \ldots, J-1,\\
    x^{(0)}
\eqdef &
    \big(\operatorname{Log}_{\bar{x}_m}^h\circ \iota_{\mathcal{N},\bar{x}_m}(x_m)\big)_{m=0}^{\memory},
\end{aligned}
\end{equation}
where the \textit{encoding/feature} and \textit{decoding/readout} maps are defined by
the product of Riemannian exponential map at pre-specified points and its inverse up to identification with the relevant Euclidean spaces via choices of linear isomorphisms
$\iota_{\mathcal{N},\bar{x_m}}:T_{\bar{x}_m}\mathcal{M}\to \mathbb{R}^d$ and $\iota_{\mathcal{M},y}:T_y\mathcal{M}\to \mathbb{R}^d$. 

We encode the trainable parameters $\bar{x}_{[0:\memory]}=(\bar{x}_m)_{m=0}^{\memory}\in\mathcal N^{1+\memory}$ and $\bar{y}\in\mathcal{M}$ vectorially via
\begin{equation}
\label{eq:HopfRinow_Encoding}
\begin{aligned}
\bar{x}_{[0:\memory]} \eqdef %&
(\operatorname{Exp}_{x^{\star}}^h(\tilde{x}_{m}))_{m=0}^{\memory}
\,\,\,\mbox{ and }\,\,\,
\bar{y} \eqdef %& 
\operatorname{Exp}_{y^{\star}}^g(\tilde{y})
,
\end{aligned}
\end{equation}
where $\tilde{x}_{[0,\memory]}\in \mathbb{R}^{(1+\memory)D}\cong (T_{x^{\star}}\mathcal{N})^{1+\memory}$, $\tilde{y}\in \mathbb{R}^{d}\cong T_{y^{\star}}\mathcal{M}$ for arbitrary, but a-priori fixed, points $x^{\star}\in \mathcal{N}$ and $y^{\star}\in \mathcal{M}$. 
The set of all GDNs with representation~\eqref{eq:GCN_Summary}-\eqref{eq:HopfRinow_Encoding} is denoted by $\mathcal{GDN}_{[\mathbf{d}],\memory}$.

The following universal approximation theorem significantly improves on \citep[Theorem 9]{kratsios2022universal} in three ways.  
First, when the target function is smooth, we obtain dimension-free approximation rates, unlike \citep[Theorem 9]{kratsios2022universal}, which only gives cursed approximation rates regardless of the regularity of the target function (on infinite sets).
Second, it exponentially improves on the loose rate given therein when approximating $\alpha$-H\"{o}lder functions; indeed, when $(\mathcal{N},h)$ is a Euclidean space and $(\mathcal{M},g)$ is the real line with Euclidean metric, our result matches the optimal approximation rates of \citep[Theorem 2.4]{ShenYangHaizhaoZhang_OptimalReLU_2022}.  
Third, when the compact set on which the function is being approximated is finite, we obtain approximation guarantees without assuming any continuity of the target function using the fact that such functions have smooth extensions.

\begin{theorem}[Approximation capacity of GDNs for smooth and H\"{o}lder concept classes]
\label{theorem:optimal_GDN_Rates__ReLUActivation}
Fix a depth parameter $J \in \mathbb{N}_+$.  Suppose that $(\mathcal{N},h)$ and $(\mathcal{M},g)$ satisfy Assumptions~\ref{ass:NPC}, $\memory \in \mathbb{N}_+$, and $0<\varepsilon\le 1$.  Fix reference points $\bar{x}\in \mathcal{N}^{1+\memory}$ and $\bar{y}\in \mathcal{M}$ as in~\eqref{eq:GCN_Summary}. 
For every locally $\alpha$-H\"{o}lder map (resp.~$k$-times continuously differentiable, resp.\ arbitrary if $\mathcal{K}$ is finite) $f:\mathcal{N}^{\memory+1} \rightarrow \mathcal{M}$, for some $0<\alpha \le 1$ (resp. $k\in \mathbb{N}_+)$, and each non-empty compact $\mathcal{K}_{[0:\memory]}\subseteq \mathcal{N}^{\memory+1}$ there is a GDN $\hat{f}:\mathcal{N}^{\memory+1}\rightarrow \mathcal{M}$ satisfying
\begin{equation*}
% \label{eq:lemma:optimal_GDN_Rates__ReLUActivation}
    \sup_{x_{[0:\memory]} \in \mathcal{K}_{[0:\memory]}}\,
        d_g\big(f(x_{[0:\memory]}),\hat{f}(x_{[0:\memory]})\big)<\varepsilon.
\end{equation*}
and depth and width as recorded in Table~\ref{tab:casestudies}.
\end{theorem}
\begin{proof}
See Appendix \ref{s:Proof_Static_Case}. 
\end{proof}

By randomizing, we obtain a global probabilistic universal approximation theorem for general Borel functions.  
This result can be interpreted as a non-Euclidean generalization of \citep[Theorem 2.3]{hornik1989multilayer} where the authors quantify the accuracy with which MLPs can approximate Borel functions between Euclidean spaces, which are not necessarily in some $L^p$-space, using a weaker notion than uniform convergence on compact sets.

\begin{table}[H]%[t]%
    \centering
	\ra{1.3}
    \caption{\textbf{Complexity estimates for the GDN in Theorem~\ref{theorem:optimal_GDN_Rates__ReLUActivation}}}
    \label{tab:casestudies__BigO}
    \begin{tabular}{@{}lll@{}}
    \cmidrule[0.3ex](){1-2}
    \textbf{$C^k$-Times Cnt.\ Diff.\ } ($k\in \mathbb{N}_+$) &\\
    \midrule
        Depth & $
            %%% Cost of Deep Parallelization
            \tilde{\mathcal{O}}\big(
                D(J+k^2+
                \memory
                d)
            \big)
            $\\
        Width Parameter &  $
                %%% CONSTANT
                \tilde{\mathcal{O}}\big(
                    J(\sqrt{D}/\varepsilon)^{(1+\memory)/(2k)}
                \big)
                $
        \\    \arrayrulecolor{lightgray}
    \cmidrule[0.3ex](){1-2}
    \textbf{Locally $\alpha$-H\"{o}lder} ($0<\alpha\le 1$) & \\
    \arrayrulecolor{lightgray}\hline
        Depth & $\mathcal{O}(JD)$\\
        Width Parameter & $\mathcal{O}\big(
        \memory
        d\,V(J^{-2}\varepsilon)^{\frac1{\alpha(1+\memory)d}}\big)
            $
    \\
    \arrayrulecolor{lightgray}
    \cmidrule[0.3ex](){1-2}
    % \textbf{Hyperparam.} & 
    \textbf{No. Reg. - Finite $N\eqdef \mathcal{K}_{[0:\memory]}$} & \\
    \arrayrulecolor{lightgray}\hline
    Depth & $\mathcal{O}(N^2\,J)$ \\
    Width Parameter & $\mathcal{O}(\varepsilon^{-\frac{1}{2N\,J}} + N)$\\
    \bottomrule
    \end{tabular}
    % }
    \caption*{Explicit expressions for all constants are given in Table~\ref{tab:casestudies}.}
\end{table}

\begin{corollary}[GDN are universal approximators of Borel functions]
\label{cor:Universality}
Let $(\mathcal{N},h)$ and $(\mathcal{M},g)$ be Riemannian manifolds satisfying Assumptions~\ref{ass:NPC}.  Let $\memory\in \mathbb{N}_+$ and $X_0,\dots, X_{\memory}$ are i.i.d.\ $\mathcal{N}$-valued random variables defined on a common probability space $(\mathcal{S},\mathcal{A},\mu)$. 
For any Borel function $f:\mathcal{N}^{1+\memory}\rightarrow \mathcal{M}$ and every approximation error $0<\varepsilon\le 1$, there is a GDN $\hat{f}:\mathcal{N}^{1+\memory}\rightarrow \mathcal{M}$ satisfying the probabilistic bound
\[
        \mu\big(
            d_h\big(
                f(X_{[0:\memory]}),\hat{f}(X_{[0:\memory]})
            \big) < \varepsilon
        \big)
    > 
        1- \varepsilon
.
\]
\end{corollary}
\begin{proof}
    See Appendix \ref{s:Proof_Static_Case}. 
\end{proof}

Next, we consider the dynamic version of our results, and our main approximation theorem.
\subsection{Dynamic/Sequential universality} \label{s:Approximation__ss:Dynamic}

Before presenting the dynamic universal approximation theorems for GDNs, we formalize their architecture summarized graphically in Figure~\ref{fig:Model}.
Briefly, the blue loop is an auxiliary feed-forward neural network defined on the GDN model's parameter space.  This auxiliary feedforward neural network, called \textit{hypernetwork} (e.g., \cite{von2019continual}), synchronizes the parameters of several GDNs each of which independently/in parallel approximates the map $x_{[0:T]}\mapsto(\Pi_{x_{[0:t]}})_{t=0}^T$ for a unique $t$. The hypernetwork synchronizes these GDNs, memorizing the parameters of an optimal GDN at time $t+1$ from those defining an optimal one at time $t$, for each $t=0,\dots,T$.  Since memorization requires exponentially fewer parameters than approximation, then the HGN can approximate the causal map $x_{[0:T]}\mapsto(\Pi_{x_{[0:t]}})_{t=0}^T$ using roughly the same number of parameters as one GDN needed to approximate the map at any one step in time.  This allows us to obtain our main ``dynamic'' approximation theorem.

\begin{figure}[htp!]%[H]
\centering
  \includegraphics[width=1\linewidth]{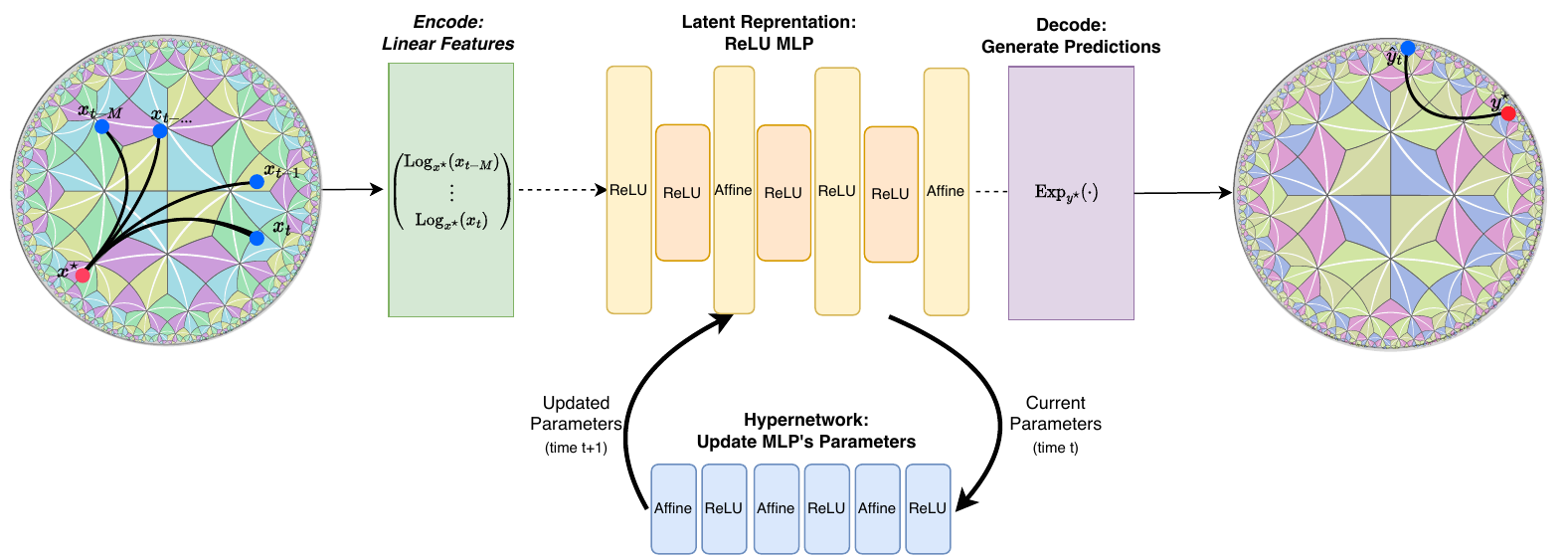}
  \caption{\textbf{The HGN Model:} The green layer encodes sequence segments in the input manifold into distances relative to a reference/landmark point $x^{\star}$ therein.  These linearized features are then processed through a ReLU MLP, illustrated by the yellow repeated applying fully-connected affine (also called linear) layers interspersed with ReLU activation functions orange.  Finally, the purple decodes the vector $v$ generated by the downsampled CNN into a manifold-valued prediction, by travelling along a geodesic emanating from a reference/landmark point $y^{\star}$ therein with initial velocity $v$.  
  \hfill\\
  \textbf{The HGN Model:} Applies the GDN model while iteratively updating its internal parameters, at each time step, using an (blue) auxiliary ReLU network, called a \textit{hypernetwork}.}
  \label{fig:Model}
\end{figure}

\begin{definition}[(HGN) Hypergeometric Network]
\label{defn:HCGNs}
A map $F: \mathcal{N}^{\mathbb{Z}}\rightarrow \mathcal{M}^{\mathbb{Z}}$ 
is called a \textit{hyper-geometric network (HGN)} if there are $\memory,d^{\star},d,D\in \mathbb{N}_+$, a multi-index $[\mathbf{d}]
\eqdef
(d_0=(1+\memory)D,\dots,d_J=d)$, a \textit{hypernetwork} given by a ReLU MLP $h:\mathbb{R}^{d^{\star}}\rightarrow \mathbb{R}^{d^{\star}}$, an \textit{initial parameter} $z\in \mathbb{R}^{d^{\star}}$, and a linear readout map $\tilde{L}:\mathbb{R}^{d^{\star}}\rightarrow \mathbb{R}^{P([\mathbf{d}])}$ representing $F$ as
\[
\begin{aligned}
    F\big(x_{\cdot}\big)_t
        & = 
    f_{\tilde{L}(z_t)}\big(
            x_{[t-\memory:t]}
    \big)
    \\
    z_{t}  & = \begin{cases}
        h(z_{t-1}) & \mbox{ if } t>0\\
        z & \mbox{ if } t\le 0
    \end{cases}
,
\end{aligned}
\]
for every $t\in \mathbb{Z}$ and each sequence $x_{\cdot} \in \mathcal{N}^{\mathbb{Z}}$, where
$f_{\theta}\in \mathcal{GDN}_{[\mathbf{d}],\memory}$.
The quadruple $(z,\tilde{L},h,[\mathbf{d}])$ is called a HGN with \textit{history} $H$ and latent dimension $d^{\star}$.  The causal map $F$ is called the representation of the HGN $(z,\tilde{L},h,[\mathbf{d}])$.
\end{definition}
When clear from the context, we refer to the representation of a HGN simply as a HGN.

We now formalize what it means for a map $f:\mathcal{N}^{\mathbb{N}}\rightarrow \mathcal{M}^{\mathbb{N}}$ to flow information forward in time while not overemphasizing the arbitrarily distant past.  When $\mathcal{N}$ and $\mathcal{M}$ are the Euclidean space, this definition is a quantitative version of the fading memory property for reservoir systems (e.g., \cite{lukovsevivcius2009reservoir,Gonon2022NNs}) and it is a variant of the approximable complexity property of \cite{Acciaio2022_GHT} and the related notion introduced in \citep[Definition 9]{galimberti2022designing}. 
We use $C^{\lambda}_{\alpha}(\mathcal{N}^{\memory},\mathcal{M})$ to denote the space of $\alpha$-H\"{o}lder functions from $\mathcal{N}^{\memory}$ to $\mathcal{M}$ whose $\alpha$-H\"{o}lder coefficients is at most $\lambda$-Lipschitz.
We use $C^{k,\lambda}(\mathcal{N}^{\memory},\mathcal{M})$ to denote the space of all $C^k$-functions $f:\mathcal{N}^{\memory} \to \mathcal{M}$ for which there exists $x_1,\dots,x_{\memory}\in \mathcal{N}^{\memory}$ and some $y\in \mathcal{M}$ such that 
\[
\operatorname{Log}_y^g\circ f\circ 
\Big(\prod_{m=1}^{\memory}\operatorname{Exp}_{x_m}^h\Big)
\]
has at most $\lambda$-Lipschitz $k^{th}$ partial derivatives. 

\begin{definition}[Causal maps of finite virtual memory between Hadamard manifolds]
\label{def:causal_maps}
\hfill\\
    A map $f\,:\,\mathcal{K} \rightarrow \mathcal{M}^{\mathbb{Z}}$ from a compact subset $\mathcal{K} \subseteq \mathcal{N}^{\mathbb{Z}}$ is called a causal map with virtual memory $r \ge 0$, if for every ``compression level'' $\varepsilon>0$ and each ``time-horizon'' $T\in \mathbb{N}_+$ there is a positive integer $\memory
    \eqdef \memory(\varepsilon,T,\mathcal{K})
    \in O(\varepsilon^{-r})$ with $\memory\le T$ and functions $f_1,\dots,f_T \in C(\mathcal{N}^{\memory}, \mathcal{M})$
     satisfying
    \begin{equation}
    \label{eq:causal_maps}
        \max_{t \in [[T]]}
        \sup_{x \in \mathcal{K}} \,
            d_g\big(
                    f(x)_t
                , 
                    f_t(x_{(t-\memory: t]})
            \big)<\varepsilon.
    \end{equation}
We say that $f$ is $(r, k,\lambda)$-smooth if there are $k\in \mathbb{N}_+$ and $\lambda>0$ such that each $t\in\mathbb N_+$ the map $f_{t}$ belongs to $C^{k,\lambda}(\mathcal{N}^{\memory}, \mathcal{M})$.  
\hfill\\
\noindent
We say that $f$ is $(r, \alpha, \lambda)$-H\"{o}lder $k\in \mathbb{N}_+$ and $0<\alpha\le 1$ such that for each $t\in \mathbb{N}_+$ the map $f_{t}$ belongs to $C_{\alpha}^{\lambda}(\mathcal{N}^{\memory},\mathcal{M})$. 
\end{definition}

%%%%
In order for causal maps to be uniformly approximable on compact sets, they need to possess a basic level of continuity.  Therefore, we focus on the following two classes of ``high regularity'' and ``low regularity'' causal maps of finite virtual memory as approximable target functions, respectively.  These definitions were adapted from the infinite-dimensional linear analogues in \cite{galimberti2022designing} to the finite-dimensional non-vectorial setting considered herein. 

We may now state our main \textit{universal approximation theorem}, which shows that HGN can approximate causal maps of finite virtual memory.  Quantitative rates are obtained when the target function is either smooth or H\"{o}lder in the above sense.  
Our result is quantitative and depends only on the regularity and memory of the target causal map being approximated, as well on the structure of the compact path space over which the uniform spatio-temporal approximation is guaranteed.

\begin{table}[htp!]%[H]
    \centering
	\ra{1.3}
    \caption{\textbf{Complexity estimates for the HGN in Theorem~\ref{thrm:approx_HGCNN}}}
    \label{tab:HGCNN__BigO}
    % \resizebox{\columnwidth}{!}{%
    \begin{tabular}{@{}ll@{}}
    \cmidrule[0.3ex](){1-2}
    \textbf{Hyperparame.} & \textbf{Estimate} \\
    \hline
    \textbf{Depth} & $\tilde{\mathcal{O}}(T^{3/2})$ 
    \\
        \textbf{Width} & $\mathcal{O}\big(P([\bar{J}])+Q)\, T\big)$ 
    \\
    \textbf{N. Parameters} & 
    $
        \tilde{\mathcal{O}}\Big(
        T^4(P([{J}])+Q)^3
        \Big)
    ,
    $
    \\
    \textbf{Latent Dimension} %($d^{\star}=Q+P([\mathbf{d}]$) 
    & $\Theta ( P([\mathbf{d}]) )$
    \\
    \bottomrule
    \end{tabular}
    %}
    \caption*{The depth, width, and value of $P([\mathbf{d}])$ is as in Table~\ref{tab:casestudies} and as determined by the regularity of $f$ (and the cardinality of $\mathcal{K}_{[0:T]}$) in Theorem~\ref{theorem:optimal_GDN_Rates__ReLUActivation}, with the value of $\varepsilon$ instead set to be $\varepsilon/2$.}
\end{table}

\begin{theorem}[Universal approximation of causal maps between Hadamard manifolds]
\label{thrm:approx_HGCNN}
\hfill\\
Under Assumption~\ref{ass:NPC}, 
let $\mathcal{K}$ be a non-empty compact subset of $\mathcal{N}^{\mathbb{Z}}$, and let $f:\mathcal{K}\rightarrow \mathcal{M}^{\mathbb{Z}}$ be a causal map of finite virtual memory.
Fix depth and memory hyperparameters $J,Q\in \mathbb{N}_+$ and $0<\delta<1$. \hfill\\
\noindent 
Suppose either that 1) $f$ is $(r,k,\lambda)$-smooth ($r,\lambda \ge 0$ and $k\in \mathbb{N}_+$), or 2) $f$ is $(r, \alpha, \lambda)$-H\"older ($r,\lambda\ge 0$ and $\alpha \in (0,1]$) or 3) that each $\mathcal{K}_t$ is finite for $t\in \mathbb{Z}$ and $f$ has a finite virtual memory $r\ge 0$.  
\hfill\\
\noindent
For every ``approximation error'' $\varepsilon>0$ and each ``time-horizon'' $T_{\delta,Q} \eqdef \lfloor \delta^{-Q}\rfloor$, there is an integer $0\le \memory \lesssim 
\min\{T_{\delta,Q},\varepsilon^{-r}\}
$, and an HGN $F:\mathcal{N}^{\mathbb{Z}}\rightarrow \mathcal{M}^{\mathbb{Z}}$ with history $\memory$ and latent dimension $d^{\star}=Q+D$ satisfying the uniform approximation guarantee
\begin{equation*}
    \max_{t = 0,\dots,T_{\delta,Q}}\,
        \sup_{x \in \mathcal{K} }\,\,
            d_{g}\big(
                        f(x)_t
                    ,
                        F(x)_t
                \big) 
    <  
        \varepsilon
.
\end{equation*}
Moreover, the structure and number of parameters determining $\hat{h}$ are recorded in Tables~\ref{tab:casestudies__BigO} and~\ref{tab:HGCNN__BigO}.

\end{theorem}

\begin{proof}
    See Appendix \ref{s:Proof_Dynamic_Case}.
\end{proof}

\section{Ablation Study}
\label{s:Numerics}

We now complement our theoretic analysis of the \textit{Gaussian random projection} and the \textit{HGN} with a numerical analysis of these new tools.  The primary purpose of this section is to show how such a pipeline can be implemented and to explain the role of each component of our model and how each of these interacts with the Volterra process $X_{\cdot}$.  
Since there are no available benchmarks for approximating dynamical systems on Riemannian manifolds, which are guaranteed to be universal, we instead perform an ablation study to better understand the dependence of each of these various factors determining the process $X_{\cdot}$.  Additional details are provided in Appendix~\ref{a:Experiment_Details}.

\paragraph{Experiment Setup}
% \label{s:Numerics__ss:Setup}

Consider a family of i.i.d.\ random Bernoulli variables $(B_t)_{t=0}^{\infty}$ taking values in $\{0,1\}$ with equal probabilities of each state.  Fix a ``randomness'' parameter $\lambda\ge0$ and define the random matrices
\[
        \mathbf{S}_t 
    \eqdef 
        \lambda B_t\cdot I_d
\]
where $t\in \mathbb{N}$.
Fix a weight $w\in (0,1]$, Lipschitz functions $\mu:\mathbb{R}^{d}\rightarrow\mathbb{R}^d$, $\varsigma:\mathbb{R}^d\rightarrow (0,2]$, and a $d\times d$ symmetric positive-definite matrix $\sigma$.  Consider the $d$-dimensional stochastic process 
\begin{equation}
\label{eq:Volterra__ExperimentsS}
\begin{aligned}
        X_{t+1}
    & = 
            X_t
        +
            \operatorname{Drift}(X_{[t-1:t]})
        +
            \operatorname{Diffusion}(X_{t},\mathbf{S}_t)
            \,
            W_t
    \\
        \operatorname{Drift}(z,x) & \eqdef w \mu(x) + (1-w) \,\mu(z)
    \\
        \operatorname{Diffusion}(x,s) & \eqdef 
        \varsigma (x)\cdot \sigma + s
,
\end{aligned}
\end{equation}
for $t\in \mathbb{N}$, where $(W_t)_{t=0}^{T}$ are i.i.d.\ $d$-dimensional standard normal random variables independent of $(B_t)_{t=0}^{T}$, and both $X_{-1},X_{0}$ are $d$-dimensional standard normal random variables.  The diffusion component of the process $X_t$, conditionally on $X_t$, randomly moves between $\varsigma(X_t)\cdot\sigma$ and $\varsigma(X_t)\cdot\sigma+\lambda I_d$ with equal probabilities, independently of the driving Brownian motion. 
\noindent For any $T\in \mathbb{N}_+$, path $x_{[-1:T]}\in \mathbb{R}^{(2+T)d}$, and integer $0\le t\le T$, the $\mathcal{N}_d$-valued random variable $\mathbb{Q}_{x_{[-1:t]}}$ is distributed according to 
\begin{equation}
\label{eq:Q_states}
\begin{aligned}
        \mathbb{P}\big[
                \mathbb{Q}_{x_{[-1:t]}}
            =
                \mathcal{N}_d(
                        x_{t}
                        +
                        \operatorname{Drift}(x_{[t-1:t]})
                    ,
                        \varsigma(x_t)^2
                        \cdot
                        \sigma^2
                    )
        \big]
    =&
        \frac12
\\
        \mathbb{P}\big[
                \mathbb{Q}_{x_{[-1:t]}}
            =
                \mathcal{N}_d(
                        x_{t}+\operatorname{Drift}(x_{[t-1:t]})
                    ,
                        (
                            \varsigma(x_t)
                            \cdot
                            \sigma
                            +
                            \lambda
                            I_d
                        )^2
                    )
        \big]
    =&
        \frac12
    .
\end{aligned}
\end{equation}
By \citep[Proposition 5.5]{Sturm_2003} and the product Riemannian structure on $(\mathcal{N}_d,\mathfrak{J})$ we have that the barycenter of $\operatorname{Law}(\mathbb{Q}_{x_{[0:t]}})$ is the Cartesian product of the barycenters of its components, up to identification of $(\mathcal{N}_d,\mathfrak{J})$ with $(\mathbb{R}^d\times \operatorname{Sym}_+(d),\boldsymbol{\delta}\oplus g)$.  Using, see \citep[page 1701]{BiniBruno_KarcherMeanPSD_2103}, for the expression of the barycenter between two-points in $(\operatorname{Sym}_+(d),g)$ we find that
\begin{equation}
\label{eq:GaussianProjection__Experiments}
        \beta(\mathbb{Q}_{x_{[-1:t]}})
    =
        \mathcal{N}_d\big(
                x_t+\operatorname{Drift}(x_{[t-1:t]})
            ,
            \,\,
                \varsigma(x_t)
                \cdot
                \sigma^2(
                    \sigma^{-2} 
                    (\lambda I_d + \varsigma(x_t) \cdot \sigma)^2 
                )^{1/2}
        \big)
.
\end{equation}

Next, we confirm that the HGN model can indeed approximate the map $x_{[-1:t]}\to \beta(\mathbb{Q}_{x_{[-1:t]}})$ in practice.  
Furthermore, we inspect the dependence of each of the components of our framework on the parameters defining the Volterra process~\eqref{eq:Volterra__ExperimentsS}; namely, the drift $\mu$, the diffusion $\sigma$, $\varsigma$, the ``randomness of the stochastic factor process''  $\lambda>0$, and the effect of non-Markovianity $w$.

\paragraph{HGN Training Pipeline}
The HGN model is trained as follows.  First, we sample several $N\in \mathbb{N}_+$ paths segments $\{x^{(n)}_{[-1:T]} \}_{n=1}^N$ up to time $T \in \mathbb{N}_+$ and we train a GDN to predict $y^{(n)}_1\eqdef \beta(\mathbb{Q}_{x^{(n)}_{[-1:1]}})$ given each sampled path, by minimizing the intrinsic mean squared error (IMSE)
\[
    \ell_1(\theta)\eqdef \sum_{n=1}^N\, d_g(f_{\theta}(x_{[-1:1]}^{(n)}),y^{(n)}_1)^2
\]
where $\theta$ parameterized a set of GDNs of pre-specified depth and width.  The IMSE is optimized using the native ADAM optimizer built into Pytorch until a suitable GDN parameter $\theta_1$ is obtained.

Then, for every subsequent time $t$, each we train a GDN by rolling the training window forward and minimizing the corresponding GDN
\begin{equation}
\label{eq:IMSE_timet}
        \ell_t(\theta)
    \eqdef
        \sum_{n=1}^N\, d_g(f_{\theta}(x_{[-t-2:t]}^{(n)}),y^{(n)}_t)^2
\end{equation}
where $y^{(n)}_t\eqdef \beta(\mathbb{Q}_{x^{(n)}_{[-t-2:t]}})$.  To avoid instability due to the several symmetries present in the parameter space of most MLPs, see e.g.~\cite{ainsworth2023git,sharma2024simultaneous}, and thus of our GDNs, we initialize the optimization of each GDN at time $t+1$ using the optimized parameters $\theta_t$ obtained by minimizing $\ell_t$ at time $t$.  Additionally, this implicitly encodes a transfer-learning effect, whereby the GDN responsible for predicting at time $t$ encodes the pre-trained structure in previous times.  We note that when training the first GDN, $20$ ADAM epochs are used while subsequent GDNs are sequentially fine-tuned using $10$ ADAM epochs.  

Once we have trained each ``expert'' GDN $\{f_{\theta_t}\}_{t=0}^T$, specialized only on approximating $\beta(\mathbb{Q}_{\cdot})$ at each time $t$, the HGN can be trained by simply minimizing the ``hyper-MSE'' $\ell_{\text{hyper}}$ in the common parameter space $\mathbb{R}^p$ of the GDNs.  Namely, 
\[
    \ell_{\text{hyper}}(\vartheta) \eqdef \sum_{t=0}^{T-1}\, \|h_{\vartheta}(\theta_t)-\theta_{t+1}\|^2
    ,
\]
where $\vartheta$ encodes the parameters of a hypernetwork of fixed depth and width.
The HGN is then fully encoded into the pair $(\theta_0,h)$.   
Additional details, and pseudo-code is contained in Appendix~\ref{a:Experiment_Details} and our code can be found at our Github page~\code.

\subsection{Ablation Results}
\label{ss:Numerics}

We study the sensitivity of the HGN and GDN models to the principle characteristics dictating the stochastic evolution of $X$.  We subsequently study the effect of encoding a large number GDN ``experts'' into a single HGN.

We fixed one base problem (a process with a specific set of parameters) and $30$ additional variations used during our ablation study, each with a similar set of parameters but with exactly one hyperparameter different (e.g.\ drift, volatility, etc...) perturbed during each ablation.  The training set consists of the first $t=0,\dots,159$ time steps and the test set consists of the final $160,\dots,200$ time steps of the process $X_{\cdot}$.   In each result we report $95\%$ empirical confidence intervals.
All experiment details on the computational resources used are in Appendix~\ref{s:ExperimentDetails}.

\paragraph{Sensitivities to aspects of $X_{\cdot}$}

We begin by ablating the sensitivity of the HGN and GDN models to: the simplicity/complexity of the drift $(\mu)$, 
the level of randomness $(\lambda)$ in the stochastic factor $\mathbf{S}_{\cdot}$, 
the effect of large/small fluctuations $(\varsigma)$ in the diffusion,
the dimension $(d)$ of the process $X_{\cdot}$, and the level of non-Markovianity/persistence of memory $(w)$ of the process $X$.  In each case, we report the intrinsic mean squared error for the GDN and HGN models and the confidence intervals formed from one standard deviation of the loss distribution about the (mean) intrinsic mean squared error across all time steps in the test set. 

\begin{table}[H]
\centering
\caption{\textbf{Drift Ablation:} Sensitivity to the structure of the drift $(\mu)$ of $X$.}
\label{tab:loss_mu}
\resizebox{1\textwidth}{!}{% %
\begin{tabular}{@{}lccccc@{}}
\toprule
\textbf{$\mu$} & \textbf{GDN Loss Mean} & \textbf{GDN Loss 95\% C.I.} & \textbf{HGN Loss Mean} & \textbf{HGN Loss 95\% C.I.} \\ \midrule
% 0.01
$\frac1{100}$
& $1.27\times{10^{-6}}$ & $[1.18, 1.36]\times{10^{-6}}$ & $4.78\times{10^{-4}}$ & $[3.89, 5.67]\times{10^{-4}}$ \\
%0.1
$\frac1{10}$
& $2.21\times{10^{-6}}$ & $[2.03, 2.39]\times{10^{-6}}$ & $4.39\times{10^{-2}}$ & $[3.65, 5.13]\times{10^{-2}}$ \\
% (0.01 - $x$) 0.5 
$\frac1{2}\,(\frac1{100}-x)$
& $1.58\times{10^{-3}}$ & $[1.55, 1.60]\times{10^{-3}}$ & $1.58\times{10^{-3}}$ & $[1.55, 1.60]\times{10^{-3}}$ \\
% exp(-$x$) + cos(0.01 $x$) 
$e^{-x} + \cos(\frac{x}{100})$
& $1.04\times{10^{-5}}$ & $[0.88, 1.19]\times{10^{-5}}$ & $1.41\times{10^{+1}}$ & $[1.29, 1.53]\times{10^{+1}}$ \\ \bottomrule
\end{tabular}
}% %
\end{table}

Table~\ref{tab:loss_mu} shows that the HGN and GDN models can predict Volterra processes whose drift is both simple, e.g.\ constant, or complicated, e.g.\ exhibiting osculations $\cos(x/100)$ and decay such as $e^{-x}$.  Nevertheless, as one would expect, the more complicated drifts are more difficult to learn for both models; as reflected by larger test set errors.  Moreover, as the drift becomes more complicated the gap between the test set performance of the HGN and GDN models grows as the parameters of the GDN become increasingly difficult to predict for the hypernetwork in the HGN model.

\begin{table}[H]
\centering
\caption{\textbf{Random Factor Ablation:} Sensitivity to the randomness $(\lambda)$ in the stochastic factor process $\mathbf{S}$.}
\label{tab:loss_lambda}
\resizebox{1\textwidth}{!}{% %
\begin{tabular}{@{}lccccc@{}}
\toprule
\textbf{$\lambda$} & \textbf{GDN Loss Mean} & \textbf{GDN Loss 95\% C.I.} & \textbf{HGN Loss Mean} & \textbf{HGN Loss 95\% C.I.} \\ \midrule
0 & $3.47\times{10^{-7}}$ & $[3.36, 3.58]\times{10^{-7}}$ & $3.49\times{10^{-7}}$ & $[3.41, 3.58]\times{10^{-7}}$ \\
0.1 & $1.58\times{10^{-3}}$ & $[1.55, 1.60]\times{10^{-3}}$ & $1.58\times{10^{-3}}$ & $[1.55, 1.60]\times{10^{-3}}$ \\
0.25 & $2.32\times{10^{-5}}$ & $[2.19, 2.46]\times{10^{-5}}$ & $7.67\times{10^{-4}}$ & $[6.87, 8.46]\times{10^{-4}}$ \\
0.5 & $8.94\times{10^{-5}}$ & $[8.25, 9.63]\times{10^{-5}}$ & $4.22\times{10^{-3}}$ & $[3.75, 4.69]\times{10^{-3}}$ \\
0.75 & $2.25\times{10^{-4}}$ & $[2.07, 2.44]\times{10^{-4}}$ & $1.34\times{10^{-2}}$ & $[1.21, 1.47]\times{10^{-2}}$ \\
0.9 & $3.30\times{10^{-4}}$ & $[3.00, 3.60]\times{10^{-4}}$ & $1.69\times{10^{-2}}$ & $[1.55, 1.83]\times{10^{-2}}$ \\
1 & $4.30\times{10^{-4}}$ & $[3.91, 4.69]\times{10^{-4}}$ & $2.12\times{10^{-2}}$ & $[1.94, 2.30]\times{10^{-2}}$ \\ \bottomrule
\end{tabular}
}% %
\end{table}

Table~\ref{tab:loss_lambda} shows that all models have increasingly larger challenges when predicting from processes with large levels of randomness $(\lambda)$ in the stochastic factor $\mathbf{S}_{\cdot}$ influencing their diffusion component.  This is because the larger $\lambda$ is, the more spread out both states~\eqref{eq:Q_states} of the random variable $\mathbb{Q}_{x_{[-1:t]}}$ becomes and, consequentially, the more information is lost when computing the intrinsic averaging using $\beta$.  As anticipated, highly random stochastic factors produce a larger gap in the test set performance of the GDN and the HGN models.  In contrast, less randomness in the stochastic factor yields a smaller test set performance gap between the GDN and HGN models.

\begin{table}[H]
\centering
\caption{\textbf{Dimension Ablation:} Sensitivity to Dimension $(d)$ of the Volterra Process $X$.}
\label{tab:loss_dimension}
\resizebox{1\textwidth}{!}{% %
\begin{tabular}{@{}lccccc@{}}
\toprule
$d$ & \textbf{GDN Loss Mean} & \textbf{GDN Loss 95\% C.I.} & \textbf{HGN Loss Mean} & \textbf{HGN Loss 95\% C.I.} \\ \midrule
2 & $3.84\times{10^{-6}}$ & $[3.32, 4.36]\times{10^{-6}}$ & $4.27\times{10^{-5}}$ & $[3.83, 4.71]\times{10^{-5}}$ \\
5 & $4.44\times{10^{-6}}$ & $[4.14, 4.75]\times{10^{-6}}$ & $9.56\times{10^{-5}}$ & $[0.88, 1.04]\times{10^{-4}}$ \\
10 & $1.58\times{10^{-3}}$ & $[1.55, 1.60]\times{10^{-3}}$ & $1.58\times{10^{-3}}$ & $[1.55, 1.60]\times{10^{-3}}$ \\
20 & $1.77\times{10^{-3}}$ & $[1.73, 1.81]\times{10^{-3}}$ & $1.76\times{10^{-3}}$ & $[1.72, 1.80]\times{10^{-3}}$ \\ \bottomrule
\end{tabular}
}% %
\end{table}

Table~\ref{tab:loss_dimension} confirms the effect of dimensionality on the expressive power of the HGN and GDN models, described in Theorems~\ref{theorem:optimal_GDN_Rates__ReLUActivation} and~\ref{thrm:approx_HGCNN} and detailed in Table~\ref{tab:casestudies__BigO}.  
Importantly, the performance of the HGN consistently mirrors that of the GDN model in dimensions $2$ to $20$.  
Since roughly the same number of parameters are used in each case, then, naturally, the performance of both models is higher in low dimensions than in higher dimensions.

\begin{table}[H]
\centering
\caption{\textbf{Non-Markovianity Ablation:} Sensitivity to the persistence of memory $(w)$ in $X$.}
\label{tab:loss_memory}
\resizebox{1\textwidth}{!}{% %
\begin{tabular}{@{}lccccc@{}}
\toprule
\textbf{Memory} & \textbf{GDN Loss Mean} & \textbf{GDN Loss 95\% C.I.} & \textbf{HGN Loss Mean} & \textbf{HGN Loss 95\% C.I.} \\ \midrule
0 & $3.34\times{10^{-8}}$ & $[2.82, 3.85]\times{10^{-8}}$ & $7.97\times{10^{-6}}$ & $[6.63, 9.31]\times{10^{-6}}$ \\
0.1 & $9.31\times{10^{-5}}$ & $[8.64, 9.98]\times{10^{-5}}$ & $9.32\times{10^{-5}}$ & $[0.86, 1.00]\times{10^{-4}}$ \\
0.25 & $1.02\times{10^{-4}}$ & $[0.94, 1.10]\times{10^{-4}}$ & $1.02\times{10^{-4}}$ & $[0.94, 1.10]\times{10^{-4}}$ \\
0.5 & $1.58\times{10^{-3}}$ & $[1.55, 1.60]\times{10^{-3}}$ & $1.58\times{10^{-3}}$ & $[1.55, 1.60]\times{10^{-3}}$ \\
\bottomrule
\end{tabular}
}% %
\end{table}

Both the HGN and GDN models perform nearly identically for all degrees of memory persistence, from Markovianity to higher levels of non-Markovian memory.  This confirms that the hypernetwork can reliably predict GDN parameters regardless of the degree of memory, as in Theorem~\ref{thrm:approx_HGCNN}.

\begin{table}[H]
\centering
\caption{\textbf{Diffusion Ablation:} Sensitivity to the size of the fluctuations $(\varsigma)$ in the diffusion of $X$.}
\label{tab:loss_varsigma}
\resizebox{1\textwidth}{!}{% %
\begin{tabular}{@{}lccccc@{}}
\toprule
\textbf{$\varsigma$} & \textbf{GDN Loss Mean} & \textbf{GDN Loss 95\% C.I.} & \textbf{HGN Loss Mean} & \textbf{HGN Loss 95\% C.I.} \\ \midrule
0.005 & $4.53\times{10^{-6}}$ & $[4.19, 4.88]\times{10^{-6}}$ & $1.39\times{10^{-4}}$ & $[1.25, 1.54]\times{10^{-4}}$ \\
0.01 & $5.02\times{10^{-6}}$ & $[4.67, 5.37]\times{10^{-6}}$ & $1.46\times{10^{-4}}$ & $[1.27, 1.66]\times{10^{-4}}$ \\
0.05 & $1.01\times{10^{-5}}$ & $[0.94, 1.08]\times{10^{-5}}$ & $4.04\times{10^{-4}}$ & $[3.55, 4.54]\times{10^{-4}}$ \\
0.1 & $1.90\times{10^{-5}}$ & $[1.79, 2.02]\times{10^{-5}}$ & $7.30\times{10^{-4}}$ & $[6.43, 8.17]\times{10^{-4}}$ \\
\arrayrulecolor{black!30}\midrule
1 & $1.17\times{10^{-3}}$ & $[1.09, 1.24]\times{10^{-3}}$ & $3.33\times{10^{-2}}$ & $[2.97, 3.70]\times{10^{-2}}$ \\
10 & $4.13\times{10^{-1}}$ & $[3.69, 4.57]\times{10^{-1}}$ & $4.56\times{10^{+0}}$ & $[4.14, 4.99]\times{10^{+0}}$ \\
100 & $2.80\times{10^{+3}}$ & $[2.74, 2.86]\times{10^{+3}}$ & $3.27\times{10^{+3}}$ & $[3.18, 3.35]\times{10^{+3}}$ \\
1000 & $3.10\times{10^{+5}}$ & $[3.01, 3.18]\times{10^{+5}}$ & $3.15\times{10^{+5}}$ & $[3.07, 3.23]\times{10^{+5}}$ \\
\arrayrulecolor{black}\bottomrule
\end{tabular}
}% %
\end{table}

Table~\ref{tab:loss_varsigma} shows that both models can reliably predict regardless of the diffusion component of the process $X_{\cdot}$ has large or small fluctuations.  As expected, the reliability of the HGN predictions deteriorates when $\varsigma$ increases, as can be seen by an increase in the standard deviation of the loss.

\begin{table}[H]
\centering
\caption{\textbf{Curvature Ablation:} Sensitivity to the size of the fluctuations $(\varsigma)$ in the diffusion of $X$.}
\label{tab:loss_varsigma__highcurvature}
\resizebox{1\textwidth}{!}{% %
\begin{tabular}{@{}lccccc@{}}
\toprule
\textbf{$\varsigma$} & \textbf{GDN Loss Mean} & \textbf{GDN Loss 95\% C.I.} & \textbf{HGN Loss Mean} & \textbf{HGN Loss 95\% C.I.} \\ \midrule
0.000001 & $2.34\times{10^{-3}}$ & $[2.30, 2.38]\times{10^{-3}}$ & $2.32\times{10^{-3}}$ & $[2.28, 2.36]\times{10^{-3}}$ \\
0.0001 & $1.52\times{10^{-3}}$ & $[1.48, 1.57]\times{10^{-3}}$ & $1.54\times{10^{-3}}$ & $[1.50, 1.58]\times{10^{-3}}$ \\
0.001 & $1.58\times{10^{-3}}$ & $[1.55, 1.60]\times{10^{-3}}$ & $1.58\times{10^{-3}}$ & $[1.55, 1.60]\times{10^{-3}}$ \\
\arrayrulecolor{black}\bottomrule
\end{tabular}
}% %
\end{table}

One may a priori expect that a very small variance hyperparameter $\varsigma$ (near $0$) would further improve the performance of the GDN and HGN models.  As illustrated by Table~\ref{tab:loss_varsigma__highcurvature}, this is not the case since Gaussian measures with nearly singular covariance matrices (which happens when $\varsigma\approx 0$) are near the \textit{(Gromov) boundary%
\footnote{This concept generalizes the circle/boundary of the hyperbolic disc model used in the visualization in Figure~\ref{fig:Model}.  Importantly, all points in this so-called Gromov boundary are infinitely far from points genuinely inside the non-positively curved Riemannian manifold.  See~\citep[Chapter II.8]{BridsonHaefliger_1999Book} for details.}~%
``at infinity''} of our non-positively curved Riemannian manifold $(\mathcal{N}_d,\mathfrak{J})$.  The trouble here is that the Lipschitz constant of the exponential map/layer based at a Gaussian measure far from the boundary is very large (see Lemma~\ref{lem:EstimateLocLipExp}) which translates to constants in our approximation guarantees in Theorem~\ref{theorem:optimal_GDN_Rates__ReLUActivation} (see Table~\ref{tab:casestudies}).  Consequentially, significantly more parameters are required for the HGN to achieve a comparable approximation accuracy when the target Gaussian measure is near the Gaussian measure used as the base point of the exponential layer of the GDN (as in Table~\ref{tab:loss_varsigma}).  This shows that the constants in our main result concretely impact practical implementations of the GDN and the HGN models.

The next set of ablation studies will continue to examine the efficacy of the hypernetwork in encoding and predicting GDN parameters in the test set.  

\subsubsection{Ablation of the Hypernetwork Encoding}
Theorem~\ref{thrm:approx_HGCNN} guarantees that the hypernetwork can effectively encode a large number of GDNs.  
In particular, doing so suggests that the HGN model can be recursively rolled, allowing us to predict well into the future.  Two questions naturally arise: 1) In practice, is a hypernetwork encoding of a sequence of GDN models legitimately trainable? 2) Does the hypernetwork continue to generate well-performing GDN models out-of-sample in future times?  This section yields an affirmative \textit{yes} to both of these questions; thus showing the feasibility and reliability of the hypernetwork in the HGN model.

In this next experiment, we test this by comparing three different degrees of hypernetwork encoding.  These experiments are run with a subset of the same configurations from the last section; which we annotate in each figure caption.
Each of the figures~\ref{fig:gap_HGN} and~\ref{fig:no_gap__HGN} plot the test set performance of each model, which begins at time $t=160$ and ends at time $t=200$.

1) (GDN) no hypernetwork is used and only a different GDN ``expert'' is used to generate predictions at any time, trained at time $t-1$.
2) (HGN $1$ step) at every time $t$, the hypernetwork loads the predictions of the GDN at time $t-1$ and uses them to predict a GDN, which is then used to predict at time $t$.  In this case, the hypernetwork component of the HGN is only ever used to make one-step-ahead predictions.
3) (HGN) loads the GDN parameters at time $160$ and then sequentially predicts the parameters at every subsequent time $t$ using its predicted parameters at time $t-1$; up until the terminal time $t=200$.
All three models perform nearly identically, as illustrated by the log-scale losses.  This shows that the hypernetwork encoding of the GDN models is as practically effective as in Theorem~\ref{thrm:approx_HGCNN}.

\begin{figure}[H]
    \centering
    \includegraphics[width=1\linewidth]{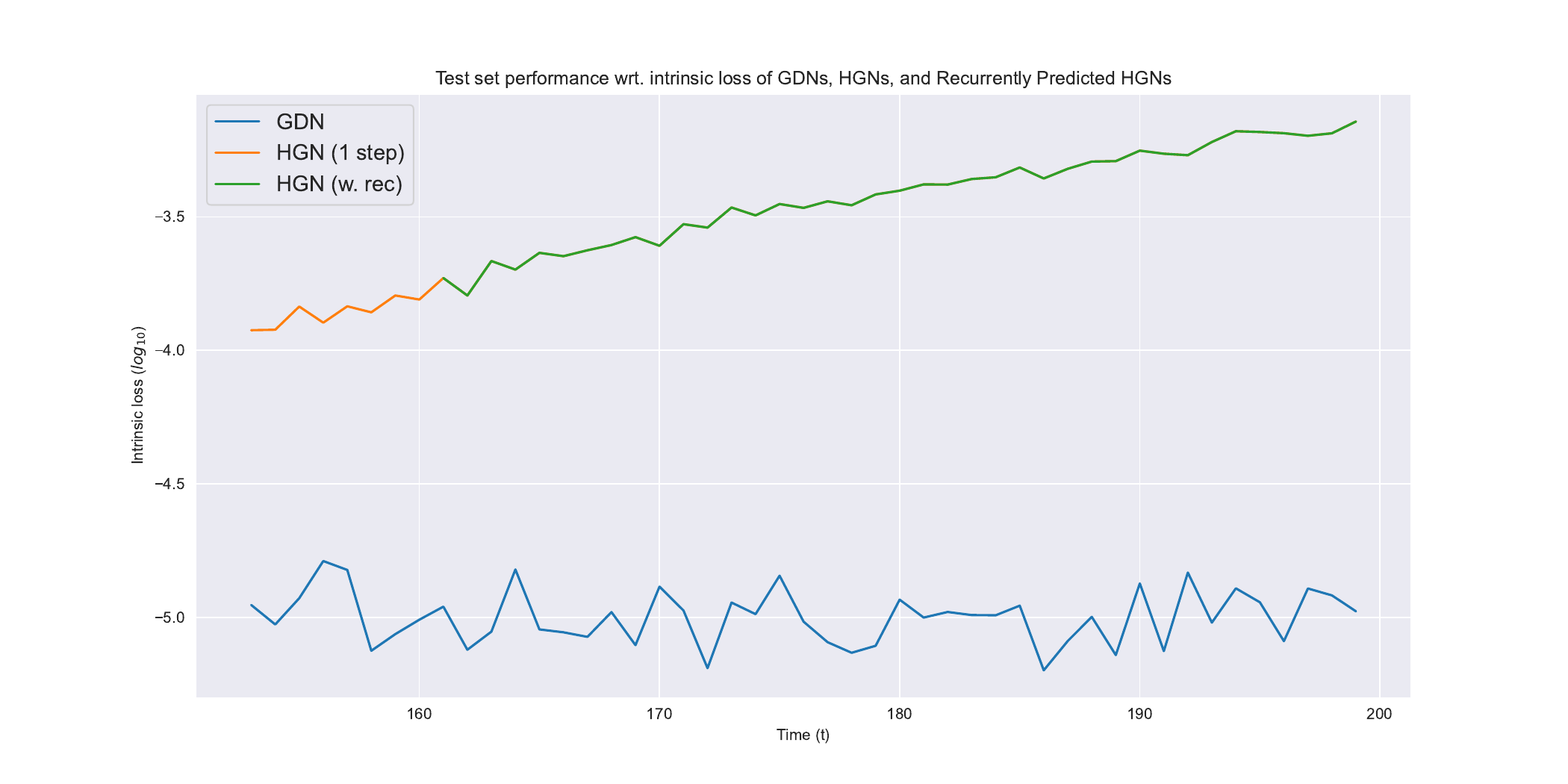}
    \caption{\textbf{Situation I - Nearly Logarithmic Degradation of HGN Accuracy:} The HGN performance slowly (logarithmically) departs from that of the GDN as time rolls forward.
    This is typically what is observed in most of our experiments.}
    \label{fig:gap_HGN}
% [0, (1000, 200, 10, 0.1, 0.5, 'torch.eye(self.d)', 'lambda X, ind: (0.01 - X) * 0.5', 'lambda X: torch.ones(X.shape[0]) * 0.05')]
\end{figure}%

Our experiments uncover two types of behaviours which the HGN can exhibit.  In the former case, seen in Figure~\ref{fig:gap_HGN}, there is a small but growing gap between the test set performance of the HGN and the GDN model which increases as time flows forward.  Furthermore, this gap is roughly the same for both the $1$-step and recurrent HGN models.  

\begin{figure}[H]
        \centering
    \includegraphics[width=1\textwidth]{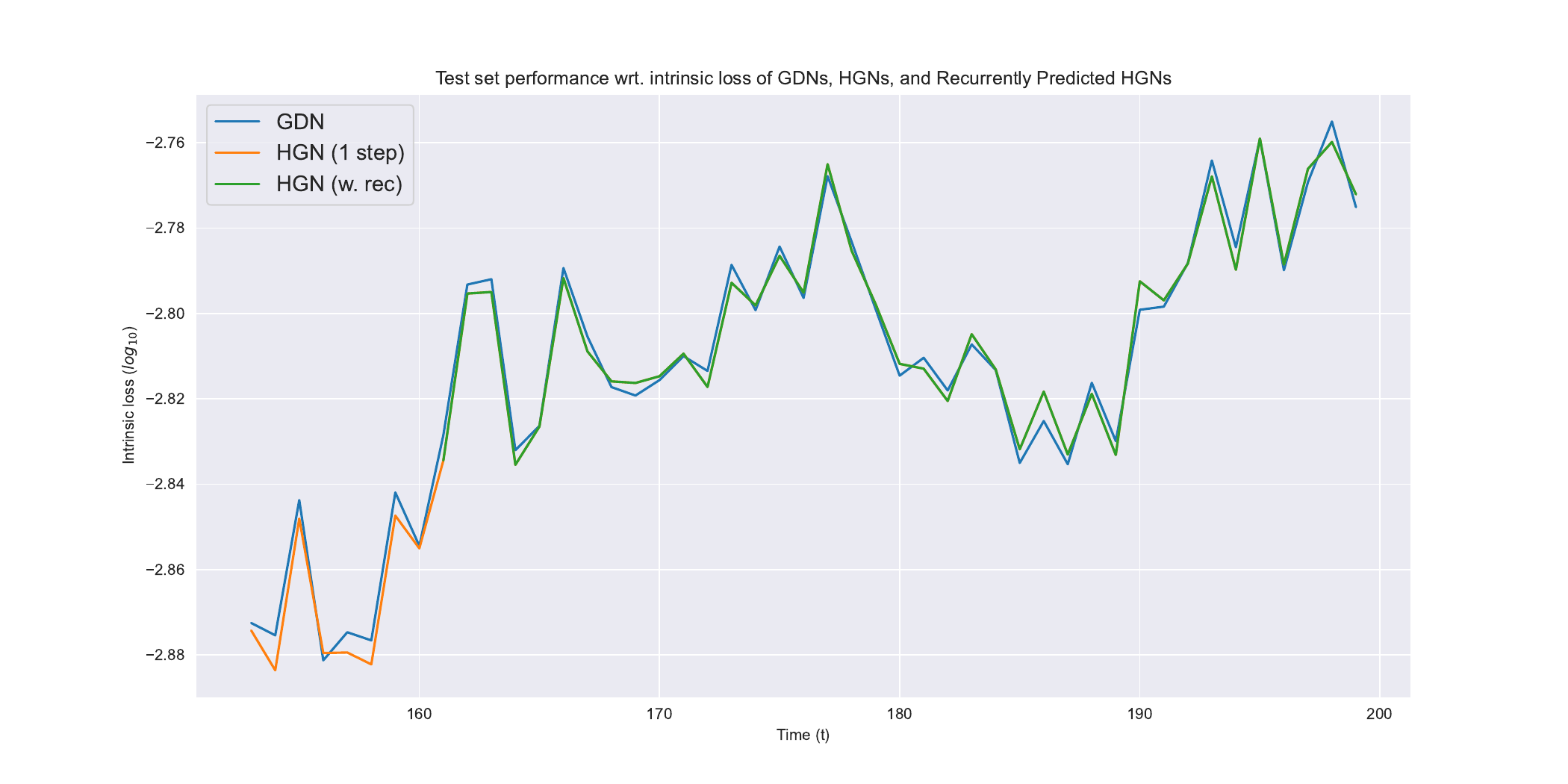}
    \caption{\textbf{Situation II - Nearly Perfect GDN Prediction by HGN:} The HGN continues to nearly perfectly predict the performance of the GDN as time rolls forward.
    This occurs in a subset of experiments where the GDN parameters do not change significantly between time steps.}
    \label{fig:no_gap__HGN}
    % [4, (1000, 200, 10, 0.1, 0.5, 'torch.eye(self.d)', 'lambda X, ind: (0.01 - X) * 0.5', 'lambda X: torch.ones(X.shape[0]) * 0.001')]
\end{figure}

Occasionally, the GDN training suffers from exploding gradients during training and one can re-run the stochastic gradient descent algorithm when this happens.  We note that there is nothing apriori particular about instances when this happens (see Appendix~\ref{ss:ExplodingGrads}).

\section{Conclusion}

We presented a framework for obtaining low-dimensional approximations of the conditional distribution (non-Markovian) stochastic Volterra processes in discrete time.  First, we develop a tool, the Gaussian projection, for projecting the condition of such processes down onto the smooth manifold $\mathcal{N}_d$ of non-singular Gaussian measure with a perturbation $\mathfrak{J}$ of its standard information geometry with favourable geometric and computational properties (see Table~\ref{tab:Non-Singular_Gaussian}). Like classical tools for dimension reduction of probability measures (e.g.\ information projections) the Gaussian projection is a projection-type optimization problem; however, unlike those tools the Gaussian projection is a Lipschitz operation (Theorem~\ref{thrm:Projection_Results} and Proposition~\ref{cor:DeterministicHighRegularity}) and can even be a \textit{smooth} map under additional conditions (Proposition~\ref{cor:DeterministicHighRegularity}).

Using these insights, we then constructed a sequential geometric deep learning model which is compatible with the non-positive curvature of $(\mathcal{N}_d,\mathfrak{J})$.  We showed that this model is a universal approximator of generalized dynamicals systems (Theorem~\ref{thrm:approx_HGCNN}), possibly with long memory, and we obtained quantitative approximation rates therefor (Theorem~\ref{theorem:optimal_GDN_Rates__ReLUActivation}).  
In the static case our model reduces the GDN model of~\cite{kratsios2022universal}; in this case, our main universal approximation theorem yields improved rates (compared to~\citep[Theorem 9]{kratsios2022universal}) when approximating both H\"{o}lder, as well as better rates when approximating functions on finite set (compare Table~\ref{tab:casestudies__BigO} against \citep[Theorem 20]{kratsios2022universal}).
We additionally obtain new probabilistic approximation guarantees for any Borel measurable function in that setting (Corollary~\ref{cor:Universality}).  

We numerically illustrated the HGN model showing its practical viability.  We conducted an ablation study, confirming our main theoretical results.  Further, our ablation study shows that certain geometric constants appearing in our parametric complexity estimates (Table~\ref{tab:casestudies__BigO}), depending on the sectional curvature of the underlying space, are legitimately non-negligible in practice.

\subsection*{Future Research}
In future work, we would like to extend our analysis to general stochastic processes, going beyond the stochastic Volterra setting.  We would also like to explore the impact of projecting onto different information-like geometries when approximation the conditional law of various processes and how to choose such geometries if one has information on the structure of these processes.

\section{Acknowledgment and Funding}
\label{s:AcknowledgmentandFunding}
AK acknowledges financial support from an NSERC Discovery Grant No.\ RGPIN-2023-04482.  AK also acknowledges that resources used in preparing this research were provided, in part, by the Province of Ontario, the Government of Canada through CIFAR, and companies sponsoring the Vector Institute~\href{https://vectorinstitute.ai/partnerships/current-partners/}{https://vectorinstitute.ai/partnerships/current-partners/}.

The authors would like to extend a very special thanks to Deane Yang from Courant Institute for his helpful references on Riemannian Centers of Mass.  They would also like to thank McKenzie Yuen-Kong Wang from McMaster University for his references on the regularity of the exponential map.  Both sets of references significantly helped our derivations.  

%%% REferences

% \newpage
\appendix

\section{Explicit Tables}
\label{s:Explicit_Tables}
This section records the versions of Tables~\ref{tab:casestudies},~\ref{tab:casestudies__dynamic}, and~\ref{tab:HGCNN} with explicit constants.  

\subsection{Static Case}

\begin{table}[htp]%[H]
    \centering
	\ra{1.3}
    \caption{\textbf{Detailed Complexity estimates for the GDN in Theorem~\ref{theorem:optimal_GDN_Rates__ReLUActivation}.}}
    \label{tab:casestudies}
    \resizebox{\columnwidth}{!}{%
    \begin{tabular}{@{}lll@{}}
    \cmidrule[0.3ex](){1-2}
	% \textbf{Hyperparam.} & 
    \textbf{$k$-Times Continuously Differentiable $f$ } & \textbf{$k\in \mathbb{N}_+$} \\
    \midrule
        Depth ($J)$ & $
            %%% Cost of Deep Parallelization
            D\Big(1 
                +
                %%%% Depth of Individual Network
                18 k^2(J+2) \log_2(4J) + 2 (1+\memory)d
            \Big)
            $\\
        Width Parameter ($w$) &  $
                %%% CONSTANT
            %%% Cost of Deep Parallelization
            C_0
            +
            %     %%%% Depth of Individual Network
            %     17k^{(1+\memory)d+1}3^{(1+\memory)d}\, 
            C_{1}
                %%% RATE
                \Big(2 + 
                    %N
                    %% Definition of N
                    \Big\lceil 
                        J\Big(
                            \varepsilon^{-1}
                            C_{\bar{f}} D^{1/2}
                            \Big)^{(1+\memory)/(2k)}
                    \Big\rceil 
                \Big)\log_2\Big(8
                    %N
                    \Big\lceil 
                        J\Big(
                            \varepsilon^{-1}
                            C_{\bar{f}} D^{1/2}
                            \Big)^{(1+\memory)d/(2k)}
                    \Big\rceil 
                \Big)
                $
        \\    \arrayrulecolor{lightgray}
    \cmidrule[0.3ex](){1-2}
    % \textbf{Hyperparam.} &  
    \textbf{Locally $\alpha$-H\"{o}lder} & \textbf{$0<\alpha\le 1$}\\
    \arrayrulecolor{lightgray}\hline
        Depth ($J)$ & $D
                \Big(
                % 1+
                    %%%% Depth of Individual Network
                    11 J + C_2
                \Big)$\\
        Width Parameter ($w$) & $
            %%% Cost of Deep Parallelization
                    C_0
                +
            %%%% Width of Individual Network
                C_1\,
                \max\big\{
                    (1+\memory)d
                    \big\lfloor
                    \big(
                    % Expression for N 
                    %%%%
                            \big\lceil
                            V\big(
                                J^{-2}\,
                                (
                                    C_2 
                                    \varepsilon
                                    /C_f)^{d/\alpha}
                            \big)
                        \big\rceil
                    %%%%
                    \big)^{1/((1+\memory)\,d)}
                    \big\rfloor
                ,
                    2 +
                    % Expression for N 
                        %%%%
                                \big\lceil
                                V\big(
                                    J^{-2}\,
                                    (
                                        C_2
                                        \varepsilon
                                        /C_f
                                    )^{1/(\alpha(1+\memory)d)}
                                \big)
                            \big\rceil
                        %%%%
            \big\}
            $
    \\
    \arrayrulecolor{lightgray}
    \cmidrule[0.3ex](){1-2}
    % \textbf{Hyperparam.} & 
    \textbf{No. Reg. - Finite $\mathcal{K}_{[0:\memory]}$} & \textbf{$N\eqdef \#\mathcal{K}_{[0:\memory]}$} \\
    \arrayrulecolor{lightgray}\hline
    Depth ($J)$ & $D(28 N^2 J)$ \\
    Width Parameter ($w$) & $
                            \lceil (\varepsilon^{-1}\,C)^{1/(2NJ)}\rceil - 1
                            \big)
                            +2N-8
                                $\\
    \bottomrule
    \end{tabular}}
    \end{table}
    \paragraph{\textbf{Details about Constants in Table~\ref{tab:casestudies}:}}%\hfill\\%%  Smooth Case
    \textbf{(1)} When $f$ is of class $C^k$, the constants are $C_0\eqdef d(1+\memory)(D-1)$, $
            C_1
        \eqdef 
            %%%% Depth of Individual Network
            17k^{d(1+\memory)+1}3^{d(1+\memory)}$ and 
            \[
                C_{\bar{f}}\eqdef \max_{i=1,\dots,D}\, 
                    \| (\rho^{-1}\circ f\circ \varphi^{-1}\circ W)_i\|_{C^k([0,1]^{(1+\memory)\,d})}
                /
                    \max\{1,\operatorname{Lip}\big(\rho\vert f\circ \varphi( \overline{\operatorname{Ball}(\mathcal{K}_{[0:\memory]},1} )\big)\}
            \]
            where $W$ is any bijective affine self-map on $\mathbb{R}^{(1+\memory)\,d}$ mapping $W(\overline{\operatorname{Ball}(\varphi(\mathcal{K}_{[0:\memory]}),1)})$ into $[0,1]^{(1+\memory)\,d}$. 
    \hfill\\
    %%  Holder Case
    \textbf{(2)} When $f$ is $\alpha$-H\"{o}lder, the constants are $C_0\eqdef (1+\memory)d(D-1)$, $C_1\eqdef 3^{(1+\memory)d+3}$, $C_2\eqdef 19 + 2(1+\memory)d$, $C_3\eqdef (131 ((1+\memory)dD)^{1/2})$, 
    \[
        C_f \eqdef \operatorname{Lip}_{\alpha}( \rho^{-1}\circ f \circ \varphi^{-1}  \vert \phi(\overline{\operatorname{Ball}(\mathcal{K}_{[0:\memory]},1)}),
    \]
    and $V:[0,\infty)\rightarrow [0,\infty)$ is the ``special function'' defined as the inverse of $t\mapsto t^2\log_3(t+2)$. %
    \hfill\\
    \textbf{(3)} When $\mathcal{K}$ is finite, $
    C=
    (1+\memory)d(D-1)
            +
        9
        \big(
\operatorname{diam}(f(\mathcal{K}_{[0:\memory]}))
        \,
        \Big(
            \frac{
                (1+\memory)d
                }{2(1+\memory)d+2}
        \Big)^{1/2}
        \,
        DN
    $.

\subsection{Dynamic Case}
Next, we provide detailed expressions for the constants in the dynamic case.

\begin{table}[H]
    \centering
	\ra{1.3}
    \caption{\textbf{Complexity estimates for the HGN in Theorem~\ref{thrm:approx_HGCNN}}}
    \label{tab:casestudies__dynamic}
    \resizebox{\columnwidth}{!}{%
    \begin{tabular}{@{}lll@{}}
    \cmidrule[0.3ex](){1-2}
	% \textbf{Hyperparam.} & 
    \textbf{$k$-Times Continuously Differentiable $f$ } & \textbf{$k\in \mathbb{N}_+$} \\
    \midrule
        Depth ($J)$ & $
            %%% Cost of Deep Parallelization
            D\Big(1 
                +
                %%%% Depth of Individual Network
                18 k^2(J+2) \log_2(4J) + 2 (1+\memory)d
            \Big)
            $\\
        Width Parameter ($w$) &  $
                %%% CONSTANT
            %%% Cost of Deep Parallelization
            C_0
            +
            %     %%%% Depth of Individual Network
            %     17k^{(1+\memory)d+1}3^{(1+\memory)d}\, 
            C_{1}
                %%% RATE
                \Big(2 + 
                    %N
                    %% Definition of N
                    \Big\lceil 
                        J\Big(
                            \varepsilon^{-1}
                            C_{\bar{f}} D^{1/2}
                            \Big)^{(1+\memory)/(2k)}
                    \Big\rceil 
                \Big)\log_2\Big(8
                    %N
                    \Big\lceil 
                        J\Big(
                            \varepsilon^{-1}
                            C_{\bar{f}} D^{1/2}
                            \Big)^{(1+\memory)d/(2k)}
                    \Big\rceil 
                \Big)
                $
        \\    \arrayrulecolor{lightgray}
    \cmidrule[0.3ex](){1-2}
    % \textbf{Hyperparam.} &  
    \textbf{Locally $\alpha$-H\"{o}lder} & \textbf{$0<\alpha\le 1$}\\
    \arrayrulecolor{lightgray}\hline
        Depth ($J)$ & $D
                \Big(
                % 1+
                    %%%% Depth of Individual Network
                    11 J + C_2
                \Big)$\\
        Width Parameter ($w$) & $
            %%% Cost of Deep Parallelization
                    C_0
                +
            %%%% Width of Individual Network
                C_1\,
                \max\big\{
                    (1+\memory)d
                    \big\lfloor
                    \big(
                    % Expression for N 
                    %%%%
                            \big\lceil
                            V\big(
                                J^{-2}\,
                                (
                                    C_2 
                                    \varepsilon
                                    /C_f)^{d/\alpha}
                            \big)
                        \big\rceil
                    %%%%
                    \big)^{1/((1+\memory)\,d)}
                    \big\rfloor
                ,
                    2 +
                    % Expression for N 
                        %%%%
                                \big\lceil
                                V\big(
                                    J^{-2}\,
                                    (
                                        C_2
                                        \varepsilon
                                        /C_f
                                    )^{1/(\alpha(1+\memory)d)}
                                \big)
                            \big\rceil
                        %%%%
            \big\}
            $
    \\
    \arrayrulecolor{lightgray}
    \cmidrule[0.3ex](){1-2}
    % \textbf{Hyperparam.} & 
    \textbf{No. Reg. - Finite $\mathcal{K}_{[0:\memory]}$} & \textbf{$N\eqdef \#\mathcal{K}_{[0:\memory]}$} \\
    \arrayrulecolor{lightgray}\hline
    Depth ($J)$ & $D(28 N^2 J)$ \\
    Width Parameter ($w$) & $
                                \lceil (\varepsilon^{-1}\,C)^{1/(2NJ)}\rceil - 1
                                \big)
                                +2N-8
                            $\\
    \bottomrule
    \end{tabular}}
    \end{table}
    \paragraph{\textbf{Details about Constants in Table~\ref{thrm:approx_HGCNN}}}
    \hfill\\
    %%  Smooth Case
    \textbf{(1)} \textbf{When $f$ is $(r,k,\lambda)$-smooth:} the constants are $C_0\eqdef d(1+\memory)(D-1)$, $
            C_1
        \eqdef 
            %%%% Depth of Individual Network
            17k^{d(1+\memory)+1}3^{d(1+\memory)}$ and $
                C_{\bar{f}}
        \eqdef 
            \displaystyle\max_{
            \underset{i=1,\dots,D}{t=T-\memory,\dots,T}
            }\, 
                    \frac{
                        \| (\rho^{-1}\circ 
                        f_t
                        \circ \phi^{-1}\circ W_t)_i\|_{C^k([0,1]^{d})}
                    }{
                    \max\{1,
                        \operatorname{Lip}\big(
                            \rho\vert 
                            f_t
                            \circ 
                            \varphi( 
                                \mathcal{K}^{\star:T}
                                )
                        \big)
                    \}
                    }
            $; where for each $t=\memory,\dots,T$, $W_t:\mathbb{R}^{(1+\memory)d}\ni 
        \frac1{C_{\bar{x}}\sqrt{M}\,2r}\, 
        \sqrt{
            \frac{2d(1+\memory)+2}{d(1+\memory)}
        }
        \,(x-\tilde{x}_t)
        \in \mathbb{R}^{(1+\memory)d}
        $ is affine.
    \hfill\\
    %%  Holder Case
    \textbf{(2)} \textbf{When $f$ is $(r,\alpha,\lambda)$-H\"{o}lder:} the constants are $C_0\eqdef (1+\memory)d(D-1)$, $C_1\eqdef 3^{(1+\memory)d+3}$, $C_2\eqdef 19 + 2(1+\memory)d$, $C_3\eqdef (131 ((1+\memory)dD)^{1/2})$, $C_f \eqdef \operatorname{Lip}_{\alpha}( \rho^{-1}\circ f \circ \varphi^{-1}  \vert \phi(\overline{\operatorname{Ball}(\mathcal{K}_{[0:\memory]},1)})$, 
    and $V:[0,\infty)\rightarrow [0,\infty)$ is the ``special function'' defined as the inverse of $t\mapsto t^2\log_3(t+2)$. %
    \hfill\\
    \textbf{(3)} \textbf{When $\mathcal{K}_{[0:T]}$ is finite:} $
    C\eqdef 
    (1+\memory)d(D-1)
            +
        9
        \big(
\operatorname{diam}(f(\mathcal{K}_{[0:\memory]}))
        \,
        \Big(
            \frac{
                (1+\memory)d
                }{2(1+\memory)d+2}
        \Big)^{1/2}
        \,
        % C_{N,D}
        DN
    $
    and $N\eqdef \binom{\max_{t=0,\dots,T}\, \mathcal{K}_t + M -1}{M}$.

We complete this section by providing more resolution on the rates in Table~\ref{tab:HGCNN__BigO}, with exact expressions wherever possible.
\begin{table}[H]
    \centering
	\ra{1.3}
    \caption{\textbf{Complexity estimates for the HGN in Theorem~\ref{thrm:approx_HGCNN}}}
    \label{tab:HGCNN}
    \resizebox{\columnwidth}{!}{%
    \begin{tabular}{@{}ll@{}}
    \cmidrule[0.3ex](){1-2}
    \textbf{Hyperparame.} & \textbf{Estimate} \\
    \hline
    \textbf{Depth} & $
        \mathcal{O}\Biggr(
        T\biggl(
        1+\sqrt{T\log(T)}\,\Big(1+\frac{\log(2)}{\log(T)}\,
        \biggl[
            C + \frac{\Big(\log\big(T^2\,2^{1/2}\big)-\log(\delta)\Big)}{\log(2)}
        \biggr]_+\Big)
        \biggr)
        \Biggr)$ 
    \\
        \textbf{Width} & $(P([\bar{J}])+Q)\, T + 12$ 
    \\
    \textbf{N. Parameters} & 
    $
        \mathcal{O}\Bigl(
        T^3(P([\bar{J}])+Q)^2\,\left(1+(P([\bar{J}])+Q)
        \sqrt{T\log(T)}\,\left(1+\frac{\log(2)}{\log(T)}\,
            \left[
            C_d+
                \frac{\Big(\log\big(T^2\,2^{1/2}\big)-\log(\delta)\Big)}{\log(2)}
            \right]_+\right)\right)
        \Bigr)
    ,
    $
    \\
    \textbf{Latent Dimension} & 
    $d^{\star} \ge \max\{12,P([\mathbf{d}])+1\}$ (free parameter)
    % $d^{\star}-P([\bar{J}])$ \anastasis{fix this} 
    \\
    \bottomrule
    \end{tabular}}
    \caption*{The values of $(J)$ and $(w)$ are as in Table~\ref{tab:casestudies} and $P([\bar{J}])$ is determined accordingly by Theorem~\ref{theorem:optimal_GDN_Rates__ReLUActivation}, with the value of $\varepsilon$ set to be $\varepsilon/2$.  These depend only on the regularity of the map $f$ and the cardinality of $\mathcal{K}_{[0:T]}$.}
\end{table}

\section{A problem with information projections for regime-switching processes}
\label{app:InfoprojProblems}
Consider the special case of~\eqref{eq:Volterra_X} where $X_{\cdot}$ evolves according to the regime-switching type hidden Markov model, with dynamics given by
\[
        X_t 
    = 
        \sum_{s=0}^t\, S_s \, W_s
,
\]
where $\{\mathbf{S}_s,W_s\}_{s=0}^{T}$ are independent random variables where the ``regimes'' are encoded by the identically distributed Bernoulli $\{0,1\}$ random variables $S_s$ with equal probabilities and the $(W_s)_{s=0}^{T}$ are standard Gaussian random variables. Let $x \in \mathbb{R}$ be a realized state of the process at time 0; the process $X_{.}$ is defined on the measurable space $(\mathbb{R}^{\mathbb{N}},\mathcal{B}(\mathbb{R}^{\mathbb{N}}))$ where $\mathcal{B}(\mathbb{R}^{\mathbb{N}})$ is the Borel $\sigma$-algebra on countable product $\mathbb{R}^{\mathbb{N}}$. These processes are such that the corresponding RCD being projected does not admit a Lebesgue density. More precisely, $\mathbb{P}[X_1\in \cdot\vert X_0=x]$ does not admit a well-defined $I$-projection on the set of $d$-dimensional Gaussian probability measures $\mathcal{N}_{d}$ with non-singular covariance, since for an $I$-projection to exist and be unique, the measure $\mathbb{P}[X_1\in \cdot\vert X_0=x]$ must admit a (Lebesgue) density.

\section{Auxiliary differential geometric results} 
\label{s:auxlemmata}
This appendix contains the required differential geometric results on which many of our proofs are built.

\subsection{Results from Riemannian geometry}
\label{s:auxlemmata__ss:RiemannianGeo}

The following result is due to \citep{Rauch_1951_ConstributionToDiffGeoInLarge_AnnMath_1951}, but we make use of the formulation which is essentially found in \citep[C.1 and the remark following (A4.2)]{Karcher_1977_RiemannianCenterofMass__MollifierSmoothing__CPAM}.  We record the additional minor details.

In what follows, for any $d$-dimensional Riemannian manifold $(\mathcal{M},g)$ and any point $p\in \mathcal{M}$, we fix a linear isomorphism $\iota_p:\mathbb{R}^d\rightarrow T_p\mathcal{M}$ which formalizes the identification of the tangent space of $\mathcal{M}$ at $p$ with the $d$-dimensional Euclidean space.

\begin{lemma}[\citep{Rauch_1951_ConstributionToDiffGeoInLarge_AnnMath_1951,Karcher_1977_RiemannianCenterofMass__MollifierSmoothing__CPAM}]
\label{lem:EstimateLocLipExp}
Let $(\mathcal{M},g)$ be a $d$-dimensional Cartan-Hadamard manifold with sectional curvatures bounded between $[-\kappa,0]$ for some constant $\kappa>0$.
For every $\bar{r}>0$ and every $p\in \mathcal{M}$, we have that 
\[
\operatorname{Lip}\left(\operatorname{Exp}_p \circ \iota_p \vert \overline{\operatorname{Ball}_{(\mathbb{R}^d, \|\cdot\|)}(0,\bar{r}/c_p)} \right)\le 2c_p\, \frac{\operatorname{sinh}(\sqrt{\kappa}\bar r)}{\sqrt{\kappa}\bar r} 
,
\]
where $\operatorname{Lip}\left(\operatorname{Exp}_p{\circ \iota_p} \vert \overline{\operatorname{Ball}_{(\mathbb{R}^d, \|\cdot\|)}(0,\bar{r}/c_p)} \right)$ denotes the best local Lipschitz constant of $\operatorname{Exp}_p\circ \iota_p$ on the the closed ball about the origin of radius $\bar{r}/c_p$ and $c_p\eqdef \abs{\iota_p}_{op}$ denotes the operator norm of $\iota_p$.
\end{lemma}
\begin{proof}\textbf{of Lemma}~\ref{lem:EstimateLocLipExp}
%%%
First of all, let us recall the following general fact: given $(\mathfrak X,g)$ and $(\mathfrak Y,h)$ Riemannian manifolds, with $\mathfrak X$ connected, and $F:\mathfrak X\to\mathfrak Y$ a $C^1$ map, we have that the following are equivalent:
\begin{enumerate}
    \item $d_\mathfrak Y(F(x_1),F(x_2)) \le C d_\mathfrak X(x_1,x_2) $ with $x_1,x_2\in\mathfrak X$;
    \item $\abs{d F_x}_{op} \le C$ with $x\in\mathfrak X$;
\end{enumerate}
where  $\abs{d F_x}_{op} \eqdef \sup_{\abs{V}_x=1} \abs{dF_x(V) }_{F(x)} $ and where  $\abs{\cdot}_x$ simply denotes $\sqrt{g(\cdot,\cdot)_x}$, and similarly  $\abs{\cdot}_{ F(x)}=\sqrt{h(\cdot,\cdot)_{F(x)}}$.

Fix $p\in\mathcal{M}$, and for any $w \in T_p\mathcal{M}$, identify $T_w(T_p\mathcal{M})$ with $T_p\mathcal{M}$ itself.

Under our hypothesis, we can apply \citep[C.1]{Karcher_1977_RiemannianCenterofMass__MollifierSmoothing__CPAM}: therefore, for any $A,v\in T_p\mathcal{M}$ with $\abs{A}_p=\abs{v}_p=1$ and $A\bot v$ (i.e. orthgonal), and $r>0$ it holds
\begin{equation}\label{eq: Rauch}
    \abs{d(\operatorname{Exp}_p)_{rv}(A)}_{\operatorname{Exp}_p(rv)} \leq \frac{\operatorname{sinh}(\sqrt{\kappa}r)}{\sqrt{\kappa}r} 
\end{equation}
where again $\abs{\cdot}_p$ simply denotes $\sqrt{g(\cdot,\cdot)_p}$, and similarly for $\abs{\cdot}_{\operatorname{Exp}_p(rv)}$. 

Let us now show an analogous bound for the general case in which not necessarily it holds $A\bot v$. In an obvious notation, write $A= A^\parallel + A^\bot$, where $A^\bot$ is orthogonal to $v$ and $A^\parallel$ is parallel to $v$. Then, by linearity of the differential, we have, for arbitrary $\abs{A}_p=\abs{v}_p=1$:
\[
\begin{split}
    \abs{d(\operatorname{Exp}_p)_{rv}(A)}_{\operatorname{Exp}_p(rv)} &\leq 
    \abs{d(\operatorname{Exp}_p)_{rv}(A^\parallel)}_{\operatorname{Exp}_p(rv)} +
    \abs{d(\operatorname{Exp}_p)_{rv}(A^\bot)}_{\operatorname{Exp}_p(rv)} \\
    &\le
    \abs{A^\parallel}_p 
    \abs{d(\operatorname{Exp}_p)_{rv}(v)}_{\operatorname{Exp}_p(rv)} +
    \abs{A^\bot}_p \frac{\operatorname{sinh}(\sqrt{\kappa}r)}{\sqrt{\kappa}r},
\end{split}
\]
where in the second line we have used \eqref{eq: Rauch}. By  \citep[C.1]{Karcher_1977_RiemannianCenterofMass__MollifierSmoothing__CPAM} again, we have that\\
$\abs{d(\operatorname{Exp}_p)_{rv}(v)}_{\operatorname{Exp}_p(rv)}=\abs{v}_p=1$. Besides, $1\le \frac{\operatorname{sinh}(\xi)}{\xi}$ for any $\xi>0$. Hence, we may write,
\[
\abs{d(\operatorname{Exp}_p)_{rv}(A)}_{\operatorname{Exp}_p(rv)} \leq
\left[\abs{A^\parallel}_p + \abs{A^\bot}_p \right] \frac{\operatorname{sinh}(\sqrt{\kappa}r)}{\sqrt{\kappa}r} \le 2 \frac{\operatorname{sinh}(\sqrt{\kappa}r)}{\sqrt{\kappa}r}.
\]
In conclusion, we have proved that for any $p\in\mathcal{M},r>0$ and $\abs{A}_p=\abs{v}_p=1$, it holds
\[
\abs{d(\operatorname{Exp}_p)_{rv}(A)}_{\operatorname{Exp}_p(rv)}  \le 2 \frac{\operatorname{sinh}(\sqrt{\kappa}r)}{\sqrt{\kappa}r},
\]
i.e. 
\[
\abs{d(\operatorname{Exp}_p)_{rv} }_{op} \le 2 \frac{\operatorname{sinh}(\sqrt{\kappa}r)}{\sqrt{\kappa}r}.
\] 
Since $(0,\infty)\ni\xi\mapsto\frac{\operatorname{sinh}(\xi)}{\xi}$ is (strictly) increasing, given $\Bar{r}>0$, we have $\abs{d(\operatorname{Exp}_p)_{rv} }_{op} \le 2 \frac{\operatorname{sinh}(\sqrt{\kappa}\bar r)}{\sqrt{\kappa} \bar r}$ for $r\leq \bar r$, i.e. 
\[
\abs{d(\operatorname{Exp}_p)_{w} }_{op} \le 2 \frac{\operatorname{sinh}(\sqrt{\kappa}\bar r)}{\sqrt{\kappa}\bar r},\quad p\in\mathcal{M},\,w\in T_p\mathcal{M},\,\abs{w}_p\le \bar r.
\]

Therefore, we conclude that 
\begin{equation}
\label{eq:Tpnormbound_Rauch}
d_M(\operatorname{Exp}_p(w_1), \operatorname{Exp}_p(w_2)) \le 2 \frac{\operatorname{sinh}(\sqrt{\kappa}\bar r)}{\sqrt{\kappa}\bar r} \abs{w_1-w_2}_p, \quad p\in\mathcal{M},\, w_1,w_2\in T_p\mathcal{M},\, \abs{w_1}_p,\abs{w_2}_p\leq \Bar{r}.
\end{equation}

Let $c_p\eqdef \abs{\iota_p}_{op}$ and suppose now that $v_1,v_2\in \overline{\operatorname{Ball}_{\mathbb{R}^d,\|\cdot\|}(0,\bar{r}/c_p)}$.  Then, $\iota_p(v_1), \iota_p(v_2)\in \overline{\operatorname{Ball}_{T_p\mathcal{M},g}(0,\bar{r})}$.  Whence,~\eqref{eq:Tpnormbound_Rauch} implies that
\[
        d_M\big(
                \operatorname{Exp}_p(\iota_p(v_1))
            ,
                \operatorname{Exp}_p(\iota_p(v_2))
        \big)
    \le 
        2\frac{\operatorname{sinh}(\sqrt{\kappa}\bar{r})}{\sqrt{\kappa}\bar{r}}
        \abs{
                \iota_p(v_1)
            -
                \iota_p(v_2)
        }_p
    \le 
        2\frac{\operatorname{sinh}(\sqrt{\kappa}\bar{r})}{\sqrt{\kappa}\bar{r}}
        c_p
        \,
        \|
                v_1
            -
                v_2
        \|
    .
\]
This yields the conclusion.
\end{proof}
Lemma~\ref{lem:EstimateLocLipExp} and the Cartan-Hadamard Theorem (\citep[Corollary 6.9.1]{Jost_2017_RGGA}), which states that the Riemannian exponential map is a global diffeomorphism with $1$-Lipschitz inverse on a simply connected complete Riemannian manifold of non-positive sectional curvature, imply local-regularity estimates for the feature map $\varphi$, the readout map $\rho$, and their inverses; where $\varphi$ and $\rho$ are defined as follows
\begin{equation}
\label{eq:feature_readout_definitions}
\begin{aligned}
\rho(y) & \eqdef \operatorname{Exp}^g_{\bar{y}}\circ \iota^{-1}_{\mathcal{M},\bar{y}}(y)
\\
\varphi(x_1,\dots,x_m) & \eqdef 
    \big(\operatorname{Log}_{\bar{x}_m}^h\circ \iota_{\mathcal{N},\bar{x}_m}(x_m)\big)_{m=0}^M
,
\end{aligned}
\end{equation}
where $x_1,\dots,x_m\in \mathbb{N}^M$ and $y\in \mathcal{M}$.

\begin{lemma}[Regularity of feature map ($\varphi$)]
\label{lem:regularity_phi}
Suppose that $(\mathcal{N},h)$ satisfies Assumption~\ref{ass:NPC}, $M\in \mathbb{N}$, $\bar{x}_{[0:\memory]}\in \mathcal{N}^{\memory}$, and $\varphi$ is as in~\eqref{eq:feature_readout_definitions}.  Then $\varphi:\mathcal{N}^{1+\memory}\rightarrow\mathbb{R}^{(1+\memory)\,d}$ is
\begin{enumerate}
    \item[\textit{(i)}] smooth and $C_{\bar{x}}$-Lipschitz (globally) for some $C_{\bar{x}}>0$ depending only on $\bar{x}$ (see below),
    \item[\textit{(ii)}] $\varphi^{-1}$ exists, it is smooth, and for any $\mathcal{K}_{[0:\memory]}\eqdef \prod_{i=0}^{\memory}\, \mathcal{K}_i \subseteq \mathbb{R}^{(1+\memory)\,d}$ where $\mathcal{K}_0,\dots,\mathcal{K}_{\memory}$ are non-empty compact sets, we have
    \[
            \operatorname{Lip}(\varphi^{-1}\vert \mathcal{K}_{[0:\memory]})
        \le 
            \max_{\tilde{m}=0,\dots,\memory}
                c_{\bar{x}_m}\,
                \frac{\operatorname{sinh}(\kappa^{1/2}
                    c_{\bar{x}_m}\, \|K_m\|%(\operatorname{diam}(K_m)^2+\|K_m\|^2)^{1/2}
                    )
                }{\kappa^{1/2}
                    c_{\bar{x}_m}\, \|K_m\|%(\operatorname{diam}(K_m)^2+\|K_m\|^2)^{1/2}
                }
    ,
    \]
    where $c_{\bar{x}_m}\eqdef \|\iota_{\bar{x}_m}\|_{op}$ for $m=0,\dots,\memory$.
\end{enumerate}
Where for any non-empty subset $A\subseteq \mathbb{R}^{d}$ we define $\|A\|\eqdef \sup_{u\in A}\,\|u\|$ where $C_{\bar{x}}\eqdef \max\limits_{m=0,\dots,\memory}|\iota_{\bar{x}_m}^{-1}|_{op}$. 
\end{lemma}
\begin{proof}\textbf{of Lemma}~\ref{lem:regularity_phi}
Since $\mathcal{N}$ is simply connected, complete, and has non-positive sectional curvatures then the first statement in the Cartan-Hadamard theorem, as formulated in \citep[Corollary 6.9.1]{Jost_2017_RGGA}, guarantees that each $\operatorname{Log}^h_{\bar{x}_m}\eqdef (\operatorname{Exp}^h_{\bar{x}_m})^{-1}:\mathcal{N} \to T_{\bar{x}_m}\mathcal{N}$, for $m=0,\dots,\memory$ is well-defined and smooth.  Moreover, the second statement of \citep[Corollary 6.9.1]{Jost_2017_RGGA} implies that each $\operatorname{Log}^h_{\bar{x}_m}$ is a $1$-Lipschitz map (globally on $(\mathcal{N},{d_h})$).  
Consequentially, for every $x_{[0:\memory]},\tilde{x}_{[0:\memory]}\in \mathcal{N}^{\memory}$ it follows that 
\allowdisplaybreaks
\begin{align*}
        \|
                \varphi(x_{[1:\memory]})
            -
                \varphi(\tilde{x}_{[1:M]})
        \|
    = &
        \biggl(
            \sum_{m=0}^{\memory}\,
                \big\|
                    \operatorname{Log}_{\bar{x}_m}{\circ \iota^{-1}_{\bar{x}_m}} (x_r)
                -
                    \operatorname{Log}_{\bar{x}_m}{\circ \iota^{-1}_{\bar{x}_m}} (\tilde{x}_r)
                \big\|^2_2
        \biggr)^{1/2}
\\
    \le &
    {
        \biggl(
            \sum_{m=0}^{\memory}\,
                [ |\iota^{-1}_{\bar{x}_m}| \,d_h(x_r,\tilde{x}_r)]^2
        \biggr)^{1/2}
    }
\\
    \le &
    {
        C_{\bar{x}}
        \,
        \biggl(
            \sum_{m=0}^{\memory}\,
                d_h(x_r,\tilde{x}_r)^2
        \biggr)^{1/2}
    }
\\
    = &
        {C_{\bar{x}}}
    \,
       d_{h:2}(x_{[1:\memory]},\tilde{x}_{[1:M]})
    .
\end{align*}
This yields claim \textit{(i)}, {where $d_{h:2}$ is the $(1+\memory)$-fold $2$-product metric on $\mathcal{N}^{1+\memory}$ of the geodesic distance $d_h$ on $(\mathcal{N},h)$}.

Let $u_{[0:\memory]},\tilde{u}_{[0:\memory]}\in (\mathbb{R}^{d})^{1+\memory}\cong \mathbb{R}^{{(1+\memory)}\,d}$.  Under the sectional curvature lower-bound in Assumption~\ref{ass:NPC} (iii), we may apply Lemma~\ref{lem:EstimateLocLipExp} to deduce that, for any $r_{{0}},\dots,r_{\memory}>0$ and each $m=0,\dots,\memory$, the restriction of the map $\operatorname{Exp}_{\bar{x}_{m}}$ to the closed Euclidean ball $\overline{\operatorname{Ball}_{(\mathbb{R}^{d},\ell^2_{d})}(0,c_{\bar{x}_m}r_m/c_{\bar{x}_m})}$ is Lipschitz with Lipschitz constant at-most $c_{\bar{x}_m}\,\frac{\operatorname{sinh}(\kappa^{1/2}c_{\bar{x}_m}r_m)}{\kappa^{1/2}c_{\bar{x}_m}r_m}$.  Therefore, we have that
\allowdisplaybreaks
\begin{align}
\notag
        d_{h:2}(\varphi^{-1}(u_{[0:\memory]}),\varphi^{-1}(\tilde{u}_{[0:\memory]}))
    = &
        \biggl(
            \sum_{m=0}^{\memory}
            \,
            d_h\big(
                    \operatorname{Exp}_{\bar{x}_m}(u_m)
                ,
                    \operatorname{Exp}_{\bar{x}_m}(\tilde{u}_m)
                \big)^2
        \biggr)^{1/2}
    \\
    \notag
    \le &
        \biggl(
            \sum_{m=0}^{\memory}
            \,
            \Big(
                c_{\bar{x}_m}
                \frac{\operatorname{sinh}(\kappa^{1/2}c_{\bar{x}_m} r_m)}{\kappa^{1/2}c_{\bar{x}_m} r_m}
                \,
                \|u_m-\tilde{u}_m\|
            \Big)^2
        \biggr)^{1/2}
    \\
    \notag
    \le &
        \biggl(
            \sum_{m=0}^{\memory}
            \,
            \biggl[
                \max_{\tilde{m}=0,\dots,\memory}
                c_{\bar{x}_m}\,
                \Big(
                    \frac{\operatorname{sinh}(\kappa^{1/2}r_{\tilde{m}})}{\kappa^{1/2}r_{\tilde{m}}}
                \Big)
            \biggr]^2
            \,
            \|u_m-\tilde{u}_m\|^2
        \biggr)^{1/2}
\\
    \notag
    \le &
        \biggl[
                \max_{\tilde{m}=0,\dots,\memory}
                    c_{\bar{x}_m}\,
                    \frac{\operatorname{sinh}(\kappa^{1/2}r_{\tilde{m}})}{\kappa^{1/2}r_{\tilde{m}}}
        \biggr]
        \,
        \biggl(
            \sum_{m=0}^{\memory}
            \,
            \|u_m-\tilde{u}_m\|^2
        \biggr)^{1/2}
\\
    \notag
    = &
        \biggl[
                \max_{\tilde{m}=0,\dots,\memory}
                    c_{\bar{x}_m}\,
                    \frac{\operatorname{sinh}(\kappa^{1/2}r_{\tilde{m}})}{\kappa^{1/2}r_{\tilde{m}}}
        \biggr]
        \,
        \|u_{[0:\memory]}-\tilde{u}_{[0:\memory]}\|
    .
\end{align}
Note that each {$K_m \subseteq\overline{\operatorname{Ball}_{(\mathbb{R}^{d},\ell^2_{d})}(0,\|K_m\| ) )}$}, for $m=0,\dots,\memory$. 
Upon setting $r_m\eqdef {\|K_m\|}$, for each $m=0,\dots,\memory$, we obtain the conclusion.
\end{proof}

\begin{remark}
\label{rem:badbasepointchoice=largernetworks}
Lemma~\ref{lem:regularity_phi} shows that a poor choice of base-points $\bar{x}_m$ when defining the feature map $\varphi$ will require a larger network to compensate for the unnecessarily large Lipschitz constant of $\varphi^{-1}$.  
\end{remark}
%%%
Since $\rho$ is $\varphi^{-1}$ when interchanging the roles of $\mathcal{M}$ and $\mathcal{N}$ in Lemma~\eqref{lem:regularity_phi}, in the special case where $M=1$, then Lemma~\ref{lem:regularity_phi} directly implies the following.  
%%%
\begin{lemma}[Regularity of readout map ($\rho$)]
\label{lem:regularity_rho}    
Suppose that $(\mathcal{M},g)$ satisfies Assumption~\ref{ass:NPC} \textit{(i)} and (iii).  Fix a $\bar{y}\in \mathcal{M}$ and define $\rho$ as in~\eqref{eq:feature_readout_definitions}.  Then, $\rho:\mathbb{R}^{D}\rightarrow \mathcal{M}$ satisfies
\begin{enumerate}
    \item[\textit{(i)}] $\rho$ is smooth and on any non-empty compact $K\subseteq \mathbb{R}^{D}$ its local-Lipschitz constant is 
    \[
            \operatorname{Lip}(\rho\vert K)
        \le 
            c_{\bar{y}}
                \frac{
                    c_{\bar{y}}
                    \operatorname{sinh}(\kappa^{1/2}
                    \operatorname{diam}(K)
                    )
                }{
                    c_{\bar{y}}
                    \kappa^{1/2}
                    \operatorname{diam}(K)
                }
    ,
    \]
    where $c_{\bar{y}}\eqdef 2\,\|\iota_{\bar{y}}\|_{op}$.  
    \item[\textit{(ii)}] $\rho^{-1}$ exists, it is smooth, and it is ${C_p}$-Lipschitz (globally) where 
    {where $C_{\bar{p}}\eqdef |\iota_{p}^{-1}|_{op}$.}
\end{enumerate}
\end{lemma}
\begin{proof}
    Point \textit{(i)} follows from Lemma~\ref{lem:EstimateLocLipExp}, by setting $\bar{r}=c_{\bar{y}}\,\operatorname{diam}(K)$ and noting that $x\mapsto \operatorname{sinh}(\kappa^{1/2} x)/(\kappa^{1/2} x)$ is monotonically increasing on $[0,\infty)$.  Point \textit{(ii)} follows from the second statement of the Cartan-Hadamard Theorem \citep[Corollary 6.9.1]{Jost_2017_RGGA} {and upon recalling that we have linearly isomorphically identified the Hilbert spaces $\iota_p:(T_p(\mathcal{M}),\|\cdot\|_{g_p}) \to (\mathbb{R}^D,\|\cdot\|)$ (which needn't be an isometry); whence $\operatorname{Lip}(\iota_p^{-1}) = C_{\bar{p}}\eqdef |\iota_{p}^{-1}|_{op} >0$ (and need not equal to $1$).}
\end{proof}

\subsection{Results from matrix theory}
\label{s:auxlemmata__ss:Matrix_Alg}
This appendix contains a short proof of some folklore results in matrix theory.

Recall that, for a negative integer $t$ and an invertible matrix $\Sigma$, the matrix $\Sigma^{t}$ is defined to be $(\Sigma^{-1})^{|t|}$.  That is, $\Sigma^{t}$ is the $t$-fold product of $\Sigma^{-1}$ with itself.  

The next lemma justifies the definition of the real exponent of a symmetric positive definite matrix.  
\begin{lemma}[Powers of symmetric positive-definite matrices]
\label{lem:Exp_log_Identity}
For each $\Sigma \in \operatorname{Sym}_+(d)$ and each $t\in \mathbb{R}$, $\exp(t\log(\Sigma))$ is a well-defined element of $\operatorname{Sym}_+(d)$.  Furthermore, if $t\in \mathbb{Z}$ then
\[
    \exp(t\log(\Sigma)) 
        = 
    \Sigma^t.
\]
\end{lemma}
\begin{proof}\textbf{of Lemma~\ref{lem:Exp_log_Identity}}
By \citep[Proposition 3.3.5]{HilgertNeeb_2012_StructureGeometryLie} the restriction of the matrix exponential map to $\operatorname{Sym}(d)$ is a diffeomorphism onto the open subset $\operatorname{Sym}_+(d)$ of $\operatorname{Sym}(d)$, and by definition, the matrix logarithm $\log$ is the inverse of $\exp$ on $\operatorname{Sym}_+(d)$.  Therefore, $X\eqdef \log(\Sigma)$ is a well-defined element of $\operatorname{Sym}(d)$. Consequentially, $tX$ belongs to $\operatorname{Sym}(d)$ for every $t\in \mathbb{R}$.  Thus, $\exp(tX)\in \operatorname{Sym}_+(d)$ nd the first claim follows.

    Let now be $t\in\mathbb Z$ with $t\neq 0$ (otherwise the statement is trivial). If $t>0$, then, since $X$ trivially commutes with itself, it easily follows that 
    \[
    \exp(tX) = \underbrace{\exp(X)\dots \exp(X)}_{\mbox{$t$-times}} = (\exp(X))^t.
    \]    
    If on the other hand $t<0$, then, from we have just proved, $\exp(tX) = \exp(-t(-X)) = (\exp(-X))^{\abs{t}}$. But since $\exp(X)$ is invertible with $\exp(-X)=\exp(X)^{-1}$, from the observation at the beginning of this section we infer $\exp(tX) = (\exp(X))^t$

    So, we have that $\exp(tX)=\big(\exp(X)\big)^t$ for every $t\in \mathbb{Z}$.  Plugging in the definition of $X$ and using the fact that $\log$ is the inverse of $\exp$ on $\operatorname{Sym}_+(d)$ yields the conclusion, namely that
    \[
            \exp(t\log(\Sigma)) 
        = 
            \big(\exp(\log(\Sigma))\big)^t 
        = 
            \Sigma^t
    ,
    \]
concluding our proof.
\end{proof}

\section{Proofs}
\label{s:Proofs}

This section contains the proofs of our main results.  
\subsection{Proof of Proposition \ref{prop:NPC_FisherRao}}
In order to prove Proposition \ref{prop:NPC_FisherRao}, we need some preliminary results.
\subsubsection{The global NPC geometry on the space of non-singular Gaussian measures}
\label{app:proof_proposition_NPC_FisherRao}
The following result is well-known in the metric geometry literature; however, parts of its proof were difficult to track down outside that literature, so we record it here.

\begin{lemma}[The Geometry of $\operatorname{Sym}_+$]
\label{lem:Symplus_Hadamard_1}
For every $d \in \mathbb{N}$, $\operatorname{Sym}_+(d )$ is a $d (d +1)/2$-dimensional Cartan-Hadamard manifold whose tangent spaces are identifiable with $\operatorname{Sym}(d )$. The Riemannian metric is given by
\[
    \hat{g}_{\Sigma}(X,Y)\eqdef \operatorname{tr}(\Sigma^{-1}X\Sigma^{-1}Y), \quad \Sigma \in \operatorname{Sym}_+(d), \,X,Y\in \operatorname{Sym}(d).
\]
Moreover, the sectional curvatures are  bounded in $[-1/2,0]$.  
\end{lemma}

\begin{proof}\textbf{of Lemma}~\ref{lem:Symplus_Hadamard_1}
As discussed on \citep[page 314, section 10.31]{BridsonHaefliger_1999Book}, $\operatorname{Sym}_+(d )$ is a smooth manifold of dimension $d (d +1)/2$ with tangent spaces identifiable with $\operatorname{Sym}(d )$ (see \cite[Section 10.2.2]{CookingMatrices_v2105}): a smooth global chart is provided by $\operatorname{vec}: \operatorname{Sym}_+(d)  \to  \operatorname{vec}( \operatorname{Sym}_+(d) )  \subset \mathbb R^ {d(d+1)/2}$ (recall that $\operatorname{vec}:\mathbb{R}^{d(d+1)/2}\rightarrow\operatorname{Sym}(d)$ is the inverse of the vector space isomorphism $\operatorname{sym}$ defined in~\eqref{eq:vecsym_def}); moreover $\operatorname{Sym}_+(d)$ can be endowed with the Riemannian metric $\hat{g}_{\Sigma}(X,Y)\eqdef \operatorname{tr}(\Sigma^{-1}X\Sigma^{-1}Y)$, where $\Sigma \in \operatorname{Sym}_+(d )$ and $X,Y\in \operatorname{Sym}(d )$.

By Theorem \citep[10.39]{BridsonHaefliger_1999Book}, $(\operatorname{Sym}_+,d_{\hat{g}})$ is a $\operatorname{CAT}(0)$ space\footnote{See \citep[Definition II.1.1]{BridsonHaefliger_1999Book}.} and since $(\operatorname{Sym}_+,\hat{g})$ is a Riemannian manifold then \citep[Theorem II.1A.6]{BridsonHaefliger_1999Book}\footnote{Due to \cite{CartanOldBook}.} implies that the sectional curvatures of $(\operatorname{Sym}_+,\hat{g})$ are all bounded-above by $0$.  
On the other hand, \citep[Proposition I.1]{FCM_2022_AcceleratedFirstOrderMethodGDManifolds}\footnote{In their extended appendix in the ArXiV version.} we see that the sectional curvatures in $(\operatorname{Sym}_+,\hat{g})$ are bounded-below by $-1/2$.
\end{proof}

In what follows, for any $\mathfrak{m}\in\mathbb{R}^d$, we let $\mathcal{N}_d^{\mathfrak{m}}\eqdef \{\mathcal{N}_d(\mathfrak m,\Sigma ):\, \Sigma\in \operatorname{Sym}_+(d) \}$.   
Since $\mathcal{N}^{\mathfrak{m}}_{d}$ is a Riemannian submanifold of dimension $d (d +1)/2$ of $(\mathcal{N}_{d },\mathcal{I}^F)$ (with metric pullbacked via the inclusion map $\iota$), by defining
\begin{equation}\label{eq: iso between sym+ and N0d}
    \Phi: \operatorname{Sym}_+(d ) \to \mathcal{N}_d^{\mathfrak m},\quad
    \Sigma\mapsto \mathcal{N}(\mathfrak m,\Sigma)
\end{equation}
it is easy to see that $\Phi$ is an isometry between $(\operatorname{Sym}_+(d),\hat{g})$ and $(\mathcal{N}^{\mathfrak{m}}_{d },\iota^\ast\mathcal{I}^F)$: see also \eqref {eq:FisherRaoMetric2}.  Lemma~\ref{lem:Symplus_Hadamard_1} this implies the following.

\begin{corollary}[{The geometry of $\mathcal{N}^{\mathfrak{m}}_d$}]
\label{lem:Symplus_Hadamard_2}
For every $d \in \mathbb{N}$ and each $\mathfrak{m}\in \mathbb{R}^d$, $\mathcal{N}^{\mathfrak{m}}_{d}$ is a $d (d +1)/2$-dimensional Cartan-Hadamard manifold.
Moreover, their sectional curvatures are bounded in $[-1/2,0]$.  In particular, $(\mathcal{N}^{\mathfrak{m}}_{d},d_{\mathcal{I}^F})$ is a global NPC space.
\end{corollary}

\noindent We are now in place to prove Proposition~\ref{prop:NPC_FisherRao}.

\begin{proof}\textbf{of Proposition}~\ref{prop:NPC_FisherRao}
We begin by determining a suitable Riemannian structure for $\mathcal{N}_d$  before studying its metric and topological characteristics, as well as obtaining estimates on its curvature.
\paragraph{Manifold structure}
We have just seen that $(\mathcal{N}^{\mathfrak{m}}_d,\iota^\ast\mathcal{I}^F)$ is a Riemannian submanifold (of dimension $d(d+1)/2$). Consider $\mathbb R^d$ endowed with its standard Euclidean metric $\boldsymbol{\delta}$, namely the identity matrix.
Define $(\Uu,g)$ to be the product of these two Riemannian manifolds
\[
\Uu\eqdef \mathbb R^d \times \mathcal{N}^0_d, \quad g\eqdef \boldsymbol{\delta} \oplus \iota^\ast\mathcal{I}^F,
\]
where a global chart for $\Uu$ is clearly given by $\kappa(\mathfrak m, \mathcal{N}_d(0,\Sigma)) = 
(\mathfrak m, \operatorname{vec}(\Sigma))$.  Consider the following bijective map $\chi:\mathcal N_d\to \Uu$ so defined 
\[
\chi: \mathcal{N}_d(\mathfrak m,\Sigma) \mapsto (\mathfrak m, \mathcal{N}_d(0,\Sigma)).
\]
It is immediate to see that (refer to Section \ref{sec:InformationGeometryGaussianMeasures} for unexplained notation) $\varphi^{-1}\circ\chi\circ\kappa:\Uu\to\Theta$ is the identity map, and thus smooth. The same holds for the inverse $\chi^{-1}$: we deduce that $\chi$ is a diffeomorphism and that the matrix representation of $\chi_\ast$ is the identity matrix of $\mathbb R^m,\,m=(d^2+3d)/2$. 

We finally pullback $g$ via $\chi$ on $\mathcal N_d$, and we set $\mathfrak{J} \eqdef \chi^\ast g$. It is straightforward to see that at an arbitrary point $\mathcal N_d(\mathfrak m,\Sigma)$~\eqref{eq:Riemannian_metric} is fulfilled.  By construction, $(\mathcal{N}_d,\mathfrak{J})$ is isometric to $(\Uu,g)$. From now on, we will work on $(\Uu,g)$: whence, geometric statements valid for $(\Uu,g)$ will also be valid for $(\mathcal{N}_d,\mathfrak{J})$. 

\paragraph{Topological properties}
Observe that $\operatorname{Sym}_+(d)$ is a convex cone, thus it is contractible\footnote{E.g.\ via the homotopy $(t,\Sigma)\mapsto t\Sigma + (1-t)I_{d }$}.  Therefore, it is connected and simply connected\footnote{See \citep[Theorem 1.10]{spanier1989algebraic}.}.  Consequentially, $\mathcal{N}^0_d$ is connected since it is homeomorphic to $\operatorname{Sym}_+(d)$.  Therefore, $\Uu$ is connected since it is the product of two connected spaces.  Moreover, $\mathcal{N}^0_d$ is also connected and simply connected, because $\operatorname{Sym}_+(d)$ is.

Since $\R^d$ is also path connected, $\Uu$ is the Cartesian product of path connected spaces, then\footnote{See \citep[Exercise 1.G]{spanier1989algebraic}.} the fundamental group $\pi(\Uu)$ of $\Uu$ is isomorphic\footnote{Here, $\cong$ denotes an isomorphism in the category of groups and group homomorphisms.} to the direct product $\pi(\mathbb{R}^d)\oplus \pi(\mathcal{N}^0_d)\cong \{0\}$ of the fundamental groups of $\mathbb R^d$ and $\mathcal N^0_d$, which are both trivial\footnote{I.e.\ homeomorphic, as groups, to the trivial group $\{0\}$.}.  Since $\Uu$ is homeomorphic to $\mathcal{N}_d$ then \citep[Theorem 1.8.8]{spanier1989algebraic} implies that $\pi(\mathcal{N}_d)\cong \{0\}$; thus, $\mathcal{N}_d$ is simply connected. 

\paragraph{Metric structure}
By Lemma~\ref{lem:productmetric}, the distance function on the $\mathcal{U}$ equals to the $2$-product metric given, for any $u_0=(\mathfrak m_0, \mathcal{N}_d(0,\Sigma_0))\in \Uu$ and $u_1=(\mathfrak m_1, \mathcal{N}_d(0,\Sigma_1))\in\Uu$, by
\begin{equation}
\label{eq:product_metric_form}
        d_g((\mathfrak m_0, \mathcal{N}_d(0,\Sigma_0)),
        (\mathfrak m_1, \mathcal{N}_d(0,\Sigma_1))) 
    = 
        \left[
        \norm{\mathfrak m_0 - \mathfrak m_1 }^2_2 + d_{\iota^\ast\mathcal{I}^F}(\mathcal{N}_d(0,\Sigma_0),
        \mathcal{N}_d(0,\Sigma_1))^2
        \right]^{1/2}
    .
\end{equation}
Since $(\mathcal{N}^0_d, \iota^\ast\mathcal{I}^F)$ is \textit{totally geodesic} see \citep[page 214]{Skovgaard_FisherRaoGeometry_NonDegenreateGaussians__1984}, then%%%
\footnote{This can be shown similarly to the Lemma on \citep[181]{KobayashiI}, by instead considering a distance-minimizing geodesic between two given points which exists since $\mathcal{N}_d$ is complete as a metric space (and this geodesically complete by the Hopf-Rinow theorem). }%%%
% \anastasis{find reference to this}
\begin{equation}
\label{eq:totally_geodesic_distanceform}
        d_{\iota^\ast\mathcal{I}^F}(\mathcal{N}_d(0,\Sigma_0),
        \mathcal{N}_d(0,\Sigma_1)) 
    = 
        d_{\mathcal{I}^F}(\mathcal{N}_d(0,\Sigma_0),
        \mathcal{N}_d(0,\Sigma_1))
\end{equation}
for every $\Sigma_0,\Sigma_1\in \operatorname{Sym}_+(d)$. \eqref{eq:totally_geodesic_distanceform} and~\eqref{eq:product_metric_form} imply that for any $(\mathfrak m_0, \mathcal{N}_d(0,\Sigma_0)), (\mathfrak m_1, \mathcal{N}_d(0,\Sigma_1))\in\Uu$
\[
        d_g((\mathfrak m_0, \mathcal{N}_d(0,\Sigma_0)),
        (\mathfrak m_1, \mathcal{N}_d(0,\Sigma_1))) 
    = 
        \left[
        \norm{\mathfrak m_0 - \mathfrak m_1 }^2_2 + 
        \frac12\sum_{i=1}^{d}\, \ln(\lambda_i)^2
        \right]^{1/2}
\]
where $\lambda_1,\dots,\lambda_{d}$ are the eigenvalues of $\Sigma_0^{-1}\Sigma_1$.  Since $(\mathcal{N}_d,\mathfrak J)$ is isometric to $(\Uu,g)$ we find that
\[
d_{\mathfrak J}(\mathcal{N}_d(\mathfrak m_0,\Sigma_0),
\mathcal{N}_d(\mathfrak m_1,\Sigma_1)) = 
\left[
\norm{\mathfrak m_0 - \mathfrak m_1 }^2_2 + 
\frac12\sum_{i=1}^{d}\, \ln(\lambda_i)^2
\right]^{1/2},\quad \mathcal{N}_d(\mathfrak m_0,\Sigma_0),\mathcal{N}_d(\mathfrak m_1,\Sigma_1)\in \mathcal{N}_d. 
\]

\paragraph{Curvature and NPC space}
By %\cite[Corollary 58]{OneillRiemGeo}
\cite[Equation (2.1)]{SurveyPropertiesProductRiemannianManifolds2003}, the product of Riemannian manifolds we have that 
\[
Rm^{\Uu}(X_1+X_2,Y_1+Y_2,Z_1+Z_2,W_1+W_2) = Rm^{\mathbb R^d}(X_1,Y_1,Z_1,W_1) + Rm^{\mathcal{N}^0_d}(X_2,Y_2,Z_2,W_2)
\]
for all $X_1,Y_1,Z_1,W_1$ tangent vectors of $\mathbb R^d$, for all $X_2,Y_2,Z_2,W_2$ tangent vectors of $\mathcal N^0_d$, and where $Rm^\Uu,Rm^{\mathbb R^d},Rm^{\mathcal N^0_d}$ denotes the Riemann curvature tensors of $\Uu,\mathbb R^d$ adn $\mathcal{N}^0_d$ respectively. Since trivially $Rm^{\mathbb R^d}$ is identically zero, the previous identity simplifies to
\begin{equation}\label{eq: Riemannian curvature tensors}
Rm^{\Uu}(X_1+X_2,Y_1+Y_2,Z_1+Z_2,W_1+W_2) = Rm^{\mathcal{N}^0_d}(X_2,Y_2,Z_2,W_2).    
\end{equation}

This fact has striking consequences: in the following $K^\Uu(X_1+X_2,Y_1+Y_2)$ will denote the sectional curve of $\Uu$ spanned by the linearly independent vectors $X_1+X_2$ and $Y_1+Y_2$; similarly, $K^{\mathcal{N}^0_d}(A,B)$ will denote the sectional curve of $\mathcal{N}^0_d$ spanned by the linearly independent vectors $A$ and $B$.

Given $X_1+X_2$ and $Y_1+Y_2$ orthonormal with respect to $g$, we have from \eqref{eq: Riemannian curvature tensors}
\[
K^\Uu(X_1+X_2,Y_1+Y_2) = Rm^{\mathcal{N}^0_d}(X_2,Y_2,Y_2,X_2)
\]
and 
\[
\begin{cases}
    1 = \boldsymbol{\delta}(X_1,X_1) + \iota^\ast\mathcal{I}^F(X_2,X_2)   \\
    1 = \boldsymbol{\delta}(Y_1,Y_1) + \iota^\ast\mathcal{I}^F(Y_2,Y_2)   \\
    0 = \boldsymbol{\delta}(X_1,Y_1) + \iota^\ast\mathcal{I}^F(X_2,Y_2),   
\end{cases}
\]
which implies $\mathcal{I}^F(X_2,X_2),\mathcal{I}^F(Y_2,Y_2)\le 1$.  To obtain the claimed sectional curvature bounds for $(\Uu,g)$, and thus for $(\mathcal{N}_d,\mathfrak J )$ we must consider three distinct cases:
\begin{enumerate}
    \item[\textit{(i)}] If $X_2=0$ or $Y_2=0$ then, $K^\Uu(X_1+X_2,Y_1+Y_2)=0$.
    \item[\textit{(ii)}] If $X_2\neq 0$, $Y_2\neq 0$ and linearly dependent then, $Y_2 = \gamma X_2$ for some $\gamma\neq 0$. We then get $K^\Uu(X_1+X_2,Y_1+Y_2)=\gamma^2 Rm^{\mathcal{N}^0_d}(X_2,X_2,X_2,X_2) = 0$.
    \item[(iii)] If $X_2$ and $Y_2$ are linearly independent (and hence $\neq 0$), then set for brevity
    \[
            t(X_2,Y_2)
        := 
            \iota^\ast\mathcal{I}^F(X_2,X_2)\iota^\ast\mathcal{I}^F(Y_2,Y_2) - 
            \iota^\ast\mathcal{I}^F(X_2,Y_2)^2
    \]
    which is $>0$ by Cauchy-Schwartz inequality. Besides, from above we also have
    \[
            t(X_2,Y_2)
        \le 
            \iota^\ast\mathcal{I}^F(X_2,X_2) \iota^\ast\mathcal{I}^F(Y_2,Y_2) \le 1,
    \]
    and thus $0<t(X_2,Y_2)\le 1$. We then obtain 
    \[
            K^\Uu(X_1+X_2,Y_1+Y_2)
        =
            K^{\mathcal{N}^0_d}(X_2,Y_2) \, t(X_2,Y_2).
    \]
   Since \citep[Proposition I.1]{FCM_2022_AcceleratedFirstOrderMethodGDManifolds}\footnote{In their extended appendix in the ArXiV version.} shows that the sectional curvatures in $\operatorname{Sym}_+(d)$ are bounded in $[-1/2,0]$ and since $\mathcal{N}_d^0$, endowed with $ \iota^\ast\mathcal I^F$, is isometric\footnote{In the Riemannian sense.} by construction to $\operatorname{Sym}_+(d)$, then $K^{\mathcal{N}^0_d}(\cdot,\cdot)\in [-1/2,0]$.
    Therefore, $K^\Uu(X_1+X_2,Y_1+Y_2)\in [-1/2,0]$.
\end{enumerate}
By exhaustion, \textit{(i)-(iii)} imply that all for possible planes in $T_{(\mathfrak m, \mathcal N_d(0,\Sigma) )}\Uu$ one has $K^\Uu(X_1+X_2,Y_1+Y_2)\in [-1/2,0]$. Since the point $(\mathfrak m, \mathcal N_d(0,\Sigma))\in\Uu $ was arbitrary and since $(\Uu,g)$ and $(\mathcal{N}_d,\mathfrak J)$ are isometric, then, $K^{\mathcal{N}_d} \in [-1/2,0]$.
In particular, $(\mathcal{N}_d,\mathfrak J)$ is a Cartan-Hadamard manifold and thus \citep[Proposition 3.1]{Sturm_2003} implies that it is a global NPC space.

\paragraph{{Exponential map at $\mathcal{N}_d(0,I_d)$}}
Lastly, since $(\Uu,g)$ is a product Riemannian manifold, then \citep[Corollary 57]{OneillRiemGeo} implies that a curve $\gamma:[0,1]\rightarrow \mathbb{R}^d\times \mathcal{N}_d^0$ is a geodesic in $(\Uu,g)$ if and only if $\gamma(t)=(\gamma_1(t),\gamma_2(t))$, where $\gamma_1$ is a geodesic in $\mathbb{R}^d$ and $\gamma_2$ is a geodesic in $\mathcal{N}_d^0$.  

Moreover, as shown on \citep[page 47]{pennec2006riemannian}, the (unique) geodesic emanating from $\Sigma\in \operatorname{Sym}_+(d)$ with initial velocity $X\in \operatorname{Sym}(d)$ is given by
\[
        t \mapsto
        \Sigma^{1/2}
        \operatorname{exp}\big(
            t\Sigma^{-1/2}
            X
            \Sigma^{-1/2}
        \big)
        \Sigma^{1/2}.
\]
Since $\operatorname{Sym}_+(d)$ is diffeomorphic to $\mathcal{N}^0_d$ (compare \eqref{eq: iso between sym+ and N0d}), we conclude that 
\[
\gamma(t) = \left(\mathfrak m + t\tilde{\mathfrak{m}},
\mathcal{N}\left(0, \Sigma^{1/2}
        \operatorname{exp}\left(
            t\Sigma^{-1/2}
            X
            \Sigma^{-1/2}
        \right)
        \Sigma^{1/2}\right)
\right) \in \Uu.
\]
Thus, the geodesics emanating from $\mathcal{N}(\mathfrak m,\Sigma)$ with velocity $(\tilde{\mathfrak{m}},X)\in \mathbb R^d\times \operatorname{Sym}(d)$ reads as
\begin{equation}
\label{eq:nameme}
t\mapsto \mathcal{N}\left(
\mathfrak m + t\tilde{\mathfrak{m}},
\Sigma^{1/2}
        \operatorname{exp}\left(
            t\Sigma^{-1/2}
            X
            \Sigma^{-1/2}
        \right)
        \Sigma^{1/2}
\right)\in \mathcal{N}_d.
\end{equation}
The expression of the Riemann exponential map of $\mathcal{N}_d$ follows upon taking $t=1$ in~\eqref{eq:nameme}. 
\end{proof}

\section{Proof of Theorem~\ref{thrm:VanishingMemoryProperty_Qx}}

\label{app:proof_thrm:VanishingMemoryProperty_Qx}
The proof of Theorem~\ref{thrm:VanishingMemoryProperty_Qx} hinges on some preparatory material, we are now going to state and prove.

\subsection{Riemann distance function on a product manifold}

We are now going to prove the following Lemma from Riemannian geometry, for which we have been unable to provide a suitable reference in the literature.

\begin{lemma}
\label{lem:productmetric}
    Let $(M_1,h_1)$ and $(M_2,h_2)$ be two smooth connected Riemannian manifolds, and let $d_1$ and $d_2$ be the respective induced Riemannian distance functions. Let $M\eqdef M_1\times M_2$ be the product manifold, endowed with the canonical Riemannian metric $h\eqdef h_1\oplus h_2$, and let $d$ be the induced Riemannian distance function. Then, it holds
    \[
    d((x_1,x_2),(y_1,y_2)) = \sqrt{d_1(x_1,y_1)^2 + d_2(x_2,y_2)^2},\quad (x_1,x_2),(y_1,y_2)\in M.
    \]
\end{lemma}
\begin{proof}
    Consider $p=(p_1,p_2):[0,1]\to M$ an arbitrary piecewise regular smooth path connecting $(x_1,x_2)$ to $(y_1,y_2)$, with $(x_1,x_2)\neq (y_1,y_2)$. Evidently, also $p_1$ and $p_2$ are piecewise regular smooth path, connecting $x_i$ to $y_i$, $i=1,2$. Let $\ell_i$ denote the length of $p_i$, $i=1,2$. We have by Cauchy-Schwartz inequality
    \[
    \begin{split}
        \operatorname{length}(p)\cdot \sqrt{\ell_1^2+\ell_2^2} &=
        \int_0^1\sqrt{\abs{\dot{p}_1(t)}_{h_1}^2 + \abs{\dot{p}_2(t)}_{h_2}^2 }\,dt \cdot\sqrt{\ell_1^2+\ell_2^2}\\
        &\ge \ell_1\int_0^1 \abs{\dot{p}_1(t)}_{h_1}\,dt + 
        \ell_2\int_0^1 \abs{\dot{p}_2(t)}_{h_2}\,dt\\
        &= \ell_1^2+\ell_2^2\,,
    \end{split}
    \]
    leading to $\operatorname{length}(p)\ge \sqrt{\ell_1^2+\ell_2^2}\ge \sqrt{d_1(x_1,y_1)^2 + d_2(x_2,y_2)^2}$ and hence 
    \[
    d((x_1,x_2),(y_1,y_2)) \ge \sqrt{d_1(x_1,y_1)^2 + d_2(x_2,y_2)^2}\,.
    \]
    To prove the reverse inequality, consider now $\gamma_i:[0,\lambda_i]\to M_i$ arbitrary piecewise regular smooth curve, parametrized by arc-length, joining $x_i$ to $y_i$, $i=1,2$: here $\lambda_i\eqdef\operatorname{length}(\gamma_i)$. 
    
    Then, $\gamma:[0,1]\to M,\, \gamma(t)=(\gamma_1(\lambda_1 t), \gamma_2(\lambda_2 t))$ is a piecewise regular smooth curve connecting $(x_1,x_2)$ to $(y_1,y_2)$, whose length equals $\sqrt{\lambda_1^2+\lambda_2^2}$. We infer
    \[
    \sqrt{\lambda_1^2+\lambda_2^2} \ge d((x_1,x_2),(y_1,y_2)),
    \]
    and hence
    \[
    \sqrt{d_1(x_1,y_1)^2 + d_2(x_2,y_2)^2}\ge
    d((x_1,x_2),(y_1,y_2)).
    \]
    
\end{proof}
\subsubsection{Additional Lemmata}

\begin{lemma}[Control of the perturbed Fisher distance $d_{\mathfrak{J}}$ with constant mean]
\label{lem:Control_GeneralSym+Case}
Let $d \in \mathbb{N}_+$ and $M>0$.  For any $V,\tilde{V}\in \mathbb{R}^{d(d+1)/2}$ satisfying $\|V\|,\|\tilde{V}\|\le M$
and any $\mathfrak{m}\in \mathbb{R}^{d}$ we have 
\[
        d_{\mathfrak{J}}\big(
                \mathcal{N}_{d}(\mathfrak{m},\exp(\sym(V)))
            ,
                \mathcal{N}_{d}(\mathfrak{m},\exp(\sym(\tilde{V})))
        \big)
    \le 
            % \,
            % 2\frac{\sqrt{2}}{
            % M
            % }\sinh(\frac{M}{\sqrt{2}})\,
            % \,
        \frac{2^{3/2}}{
        M
        }\sinh\Big(\frac{
        M
        \,
            c_p
        }{\sqrt{2}}\Big)
            \|V-\tilde{V}\|
    .
\]
\end{lemma}

\begin{proof}\textbf{of Lemma}~\ref{lem:Control_GeneralSym+Case}
First remark that by Proposition \ref{prop:NPC_FisherRao}, specifically~\eqref{prop:NPC_FisherRao__ExponentialNiceForm}, 
\[
        \operatorname{Exp}_p(0,X)
    = 
        \mathcal{N}_d(\mathfrak{m},\exp(X))
\]
where we have set $p \eqdef \mathcal{N}_d(\mathfrak{m},I_{d})$ for brevity.  

By the sectional curvature estimates on $(\mathcal{N}_d,\mathfrak{J})$ in Proposition~\ref{prop:NPC_FisherRao},
we may apply Lemma~\ref{lem:EstimateLocLipExp} to deduce that $\operatorname{Exp}_p\circ\iota_p$ is not only locally-Lipschitz but its Lipschitz constant on the compact set $\overline{\operatorname{Ball}_{(\mathbb{R}^{d(d+1)/2},\|\cdot\|)}(0,M)}$ is bounded-above by
\begin{equation}
\label{eq:lem:Control_GeneralSym+Case__ControlExpLocLipConstant}
        \operatorname{Lip}\big(\operatorname{Exp}_{p}\circ \iota_p\vert \overline{\operatorname{Ball}_{(\mathbb{R}^{d},\|\cdot\|)}(0,M)}\big) 
    \le 
        2c_p
        \,
        \frac{\sqrt{2}}{
        M
        \,
        c_p
        }\sinh\Big(\frac{
        M
        \,
        c_p
        }{\sqrt{2}}\Big)
    =
        \frac{2^{3/2}}{
        M
        }\sinh\Big(\frac{
        M
        \,
        c_p
        }{\sqrt{2}}\Big)
        ,
\end{equation}
where $\iota_p\eqdef 1_{\mathbb{R}^d}\times \sym$. For any $v\in\mathbb R^d$ and $z\in\mathbb R^{d(d+1)/2}$, by definition of $\mathfrak J$
it holds 
\[
\abs{\iota_p(v,z)}_p^2 = \|v\|^2_2 + 
\|\operatorname{sym}(z)\|^2_F \le \|v\|^2_2 + 
\sqrt{2}
\| z\|^2_2 \le \|(v,z)\|^2 
\]
and thus, $\abs{\iota_p}_{op}\le 1$. Therefore,~\eqref{eq:lem:Control_GeneralSym+Case__ControlExpLocLipConstant} yields 
\[
\operatorname{Lip}\big(\operatorname{Exp}_{p}\circ \iota_p\vert \overline{\operatorname{Ball}_{(\mathbb{R}^{d},\|\cdot\|)}(0,M)}\big) 
    \le
            \frac{2^{3/2}}{
        M
        }\sinh\Big(\frac{
        M
        \,
            c_p
        }{\sqrt{2}}\Big)
\]
and hence the claim follows.
\end{proof}

\begin{lemma}[{Main path-wise stability estimate}]
\label{lem:lipshcitzBound_on_RCD}
Fix $t \in \mathbb{N}_+$ and $R>0$. Under Assumption~\ref{ass:uniformboundedness__SProcecess}, there exists $c>0$ such that for  $x_{[0:t]},\tilde{x}_{[0:t]}\in \mathbb{R}^{(1+t)\,d}$ and $s_{[0:t]},\tilde{s}_{[0:t]}\in \overline{\operatorname{Ball}_{(\operatorname{Sym}(d),\|\cdot\|_F )}(\boldsymbol{0}_d,R)}^{1+t}$
\begin{equation}
\label{eq:lem:lipshcitzBound_on_RCD__PreciseBound}
\resizebox{0.9\hsize}{!}{$
        d_{\mathfrak{J}}\Big(
            \Psi_t(x_{[0:t]},\operatorname{vec}\circ s_{[0:t]}),
            \Psi_t( \tilde{x}_{[0:t]},\operatorname{vec}\circ \tilde{s}_{[0:t]})
        \Big)
    \le 
    c
        \biggl(
            \|x_{t}-\tilde{x}_{t}\|
            +
        \sum_{r=0}^t
                \kappa(t,r)
                \big(
                        \|x_r-\tilde{x}_r\|
                    +
                        \|s_r-\tilde{s}_r\|_F
                \big)
        \biggr)
.
$}
\end{equation}
In particular, we have the Lipschitz condition
\[
    d_{\mathfrak{J}}\Big(
    \Psi_t(x_{[0:t]},\operatorname{vec}\circ s_{[0:t]})
    \Psi_t(\tilde{x}_{[0:t]},\operatorname{vec}\circ\tilde{ s}_{[0:t]})        \Big)
    \le 
    c\,
    \big(
        \|x_{[0:t]}-\tilde{x}_{[0:t]}\|
        +
        \| s_{[0:t]}-\tilde{ s}_{[0:t]}\|_F
    \big)
.
\]
Moreover, the Lipschitz constant $c$ is
\[
        c 
    \eqdef 
                2
                \max\{1,L_{\mu}\} 
            +
                \frac{
                    2^{3/2}
                    \max\{1,L_{\sigma}\} 
                }{M+R}\sinh\Big(\frac{M+R}{\sqrt{2}}\Big)
    .
\]
\end{lemma}

\begin{proof}\textbf{of Lemma}~\ref{lem:lipshcitzBound_on_RCD}
By the triangle inequality, we can write 
\allowdisplaybreaks
\begin{align}
\nonumber
    & 
    d_{\mathfrak{J}}\Big(
     \Psi_t(x_{[0:t]},\operatorname{vec}\circ s_{[0:t]}),
\Psi_t(\tilde{x}_{[0:t]},\operatorname{vec}\circ\tilde{ s}_{[0:t]})        \Big)
    \\
    \nonumber
    = &
        d_{\mathfrak{J}}\Big(
            \mathcal{N}_d\big(
                x_t + \operatorname{Drift}(t,x_{[0:t]})
                ,
                \Sigma_{t,x_{[0:t]}, s_{[0:t]}}
                \big)
        ,
            \mathcal{N}_d\big(
                 \tilde{x}_t + \operatorname{Drift}(t,\tilde{x}_{[0:t]}) 
                ,
                \Sigma_{t,\tilde{x}_{[0:t]},\tilde{ s}_{[0:t]}}
            \big)
        \Big)
    \\
    \tag{Mean-Term}
    \label{eq:FisherRaoTriangle_Mean}
    \le &
        d_{\mathfrak{J}}\Big(
            \mathcal{N}_d\big(
                x_t + \operatorname{Drift}(t,x_{[0:t]})
                ,
                \Sigma_{t,x_{[0:t]}, s_{[0:t]}}
                \big)
        ,
            \mathcal{N}_d\big(
                 \tilde{x}_t + \operatorname{Drift}(t,\tilde{x}_{[0:t]}) 
                ,
                \Sigma_{t,x_{[0:t]}, s_{[0:t]}}
            \big)
        \Big)
    \\ 
    \tag{Cov-Term}
    \label{eq:FisherRaoTriangle_Covariance__Simplified}
    &+ 
        d_{\mathfrak{J}}\Big(
            \mathcal{N}_d\big(
                \tilde{x}_t + \operatorname{Drift}(t,\tilde{x}_{[0:t]}) 
                ,
                \Sigma_{t,x_{[0:t]}, s_{[0:t]}}
                \big)
        ,
            \mathcal{N}_d\big(
                 \tilde{x}_t + \operatorname{Drift}(t,\tilde{x}_{[0:t]}) 
                ,
                \Sigma_{t,\tilde{x}_{[0:t]},\tilde{ s}_{[0:t]}}
            \big)
        \Big)
,
\end{align}
where the function $\Sigma_{\cdot}$ is defined to be
\begin{equation}
\label{eq:DEF_SIGMANOTATION}
\Sigma_{t,x_{[0:t]}, s_{[0:t]}} 
    \eqdef 
        \exp\biggl(
            \sum_{r=0}^t
                \kappa(t,r)
                \,
                [\sigma(t,x_r) +  s_r]
        \biggr).
\end{equation}

We first bound~\eqref{eq:FisherRaoTriangle_Covariance__Simplified} term.  Under assumption~\ref{ass:uniformboundedness__SProcecess}, we have that $\sup_{t\in \mathbb{N}_+,x\in\mathbb R^d } \|\sigma(t,x)\|_F\le M$ for some constant $M>0$.  Thus, for every $t\in \mathbb{N}_{+}$, $x\in \mathbb{R}^{d}$, and $s\in \cup_{r=0}^t\,\{s_r,\tilde{s}_r\}$ we have that
\begin{equation*}
        \|
        \sigma(t,x)+s
        \|_F
    \le 
        \|\sigma(t,x)\|_F + \|s\|_F
    \le 
        M+R
    .
\end{equation*}
Consequently, for each $t\in \mathbb{N}_+$, every $x_{[0:t]},\tilde{x}_{[0:t]} \in \mathbb{R}^{(1+t)d}$, and $s_{[0:t],\tilde{s}_{[0:t]}} \in \overline{\operatorname{Ball}_{(\operatorname{Sym}(d),\|\cdot\|_F)}(\boldsymbol{0}_d,R)}^{1+t}$ we have by the triangle inequality that
\allowdisplaybreaks
\begin{align*}
\notag
        \Biggl\|
            \sum_{r=0}^t\,\kappa(t,r)\, [\sigma(t,x_r) + s_r]
        \Biggr\|_F
    \le & 
        M + R,
    \\ 
    \notag
    \Biggl\|
            \sum_{r=0}^t\,\kappa(t,r)\, [\sigma(t,\tilde{x}_r) + \tilde{s}_r]
        \Biggr\|_F \le & M +R,
\end{align*}
because $\kappa$ takes non-negative values whose sum is not bigger than 1.  
Since for any $A\in\operatorname{Sym}(d)$ it holds $\|\operatorname{vec}(A)\|\le \|A\|_F$, we conclude that
\[
    \Biggl\|
        \operatorname{vec}\left( \sum_{r=0}^t\,\kappa(t,r)\, [\sigma(t,x_r) + s_r]\right)
    \Biggr\|,\; \Biggl\|
        \operatorname{vec}\left( \sum_{r=0}^t\,\kappa(t,r)\, [\sigma(t,\tilde{x}_r) + \tilde{s}_r]\right)
    \Biggr\| \le M+R
.
\]
Therefore, upon recalling~\eqref{eq:DEF_SIGMANOTATION}, Lemma~\ref{lem:Control_GeneralSym+Case} applies; whence~\eqref{eq:FisherRaoTriangle_Covariance__Simplified} can be bounded as 
\allowdisplaybreaks
\begin{align}
\label{eq:FisherRaoTriangle_Covariance__CONTROL___BEGIN}
    & 
        d_{\mathfrak{J}}\Big(
            \mathcal{N}_d\big(
                \tilde{x}_t + \operatorname{Drift}(t,\tilde{x}_{[0:t]}) 
                ,
                \Sigma_{t,x_{[0:t]}, s_{[0:t]}}
                \big)
        ,
            \mathcal{N}_d\big(
                 \tilde{x}_t + \operatorname{Drift}(t,\tilde{x}_{[0:t]}) 
                ,
                \Sigma_{t,\tilde{x}_{[0:t]},\tilde{s}_{[0:t]}}
            \big)
        \Big)
    \\
    \notag
    \le &
        {\color{black}{
            \frac{2\sqrt{2}}{M+R}\sinh\Big(\frac{M+R}{\sqrt{2}}\Big)
        }}\,
        \Biggl\|
                \operatorname{vec}\biggl(
                    \sum_{r=0}^t
                    \kappa(t,r)
                    \,
                    [\sigma(t,x_r) +  s_r]
                \biggr)
            -
                \operatorname{vec}\biggl(
                    \sum_{r=0}^t
                    \kappa(t,r)
                    \,
                    [\sigma(t,\tilde{x}_r) +  \tilde{s}_r]
                \biggr)
        \Biggr\|
    \\
    \notag
    \le &
        {\color{black}{
            \frac{2\sqrt{2}}{M+R}\sinh\Big(\frac{M+R}{\sqrt{2}}\Big)
        }}\,
        \Biggl\|
                \biggl(
                    \sum_{r=0}^t
                    \kappa(t,r)
                    \,
                    [\sigma(t,x_r) +  s_r]
                \biggr)
            -
                \biggl(
                    \sum_{r=0}^t
                    \kappa(t,r)
                    \,
                    [\sigma(t,\tilde{x}_r) +  \tilde{s}_r]
                \biggr)
        \Biggr\|_F
    \\
    \notag 
    \le &
        {\color{black}{
            \frac{2\sqrt{2}}{M+R}\sinh\Big(\frac{M+R}{\sqrt{2}}\Big)
        }}\,
        \sum_{r=0}^t\,
        \kappa(t,r)
        \,
        \Biggl\|
                \biggl(
                    [\sigma(t,x_r) +  s_r]
                \biggr)
            -
                \biggl(
                    \,
                    [\sigma(t,\tilde{x}_r) +  \tilde{s}_r]
                \biggr)
        \Biggr\|_F
    \\
    \notag
    \le &
        {\color{black}{
            \frac{2\sqrt{2}}{M+R}\sinh\Big(\frac{M+R}{\sqrt{2}}\Big)
        }}\,
        \sum_{r=0}^t\,
        \kappa(t,r)
        \,
        \max\{1,L_{\sigma}\}\,
            (
                \|x_r-\tilde{x}_r\| + \|s_r-\tilde{s}_r\|_F
            )
    \\
    \label{eq:FisherRaoTriangle_Covariance__CONTROL___END}
    \le &
        {\color{black}{
            \max\{1,L_{\sigma}\}
            \,
            \frac{2\sqrt{2}}{M+R}\sinh\Big(\frac{M+R}{\sqrt{2}}\Big)
        }}\,
        \sum_{r=0}^t\,
        \kappa(t,r)
        \,
            (
                \|x_r-\tilde{x}_r\| + \|s_r-\tilde{s}_r\|_F
            )
    ,
\end{align}
\noindent for every $t\in \mathbb{N}_+$, each $x_{[0:t]},\tilde{x}_{[0:t]}\in \mathbb{R}^{(1+t)\,d}$, and each $s_{[0:t]},\tilde{s}_{[0:t]}\in \overline{\operatorname{Ball}_{(\operatorname{Sym}(d),\|\cdot\|_F) }(\boldsymbol{0}_d,R)}^{1+t}$.

Concerning the \eqref{eq:FisherRaoTriangle_Mean}, in force of Proposition \ref{prop:NPC_FisherRao}, we have
\allowdisplaybreaks
    \begin{align}
    &
    d_{\mathfrak{J}}\Big(
        \mathcal{N}\big(
            x_t + \operatorname{Drift}(t,x_{[0:t]})
            ,
            \Sigma_{t,x_{[0:t]}, s_{[0:t]}}
            \big)
    ,
        \mathcal{N}\big(
            \tilde{x}_t + \operatorname{Drift}(t,\tilde{x}_{[0:t]}) 
            ,
            \Sigma_{t,x_{[0:t]}, s_{[0:t]}}
        \big)
    \Big) 
    \\
    \nonumber
    = &
        \big\|
            (x_t + \operatorname{Drift}(t,x_{[0:t]})) - (\tilde{x}_t + \operatorname{Drift}(t,\tilde{x}_{[0:t]}) ) 
        \big\|
    \\
        = &
        \Biggl\|
            \biggl(
                    x_t
                +
                \sum_{r=0}^t
                    \kappa(t,r)
                    \,
                    \mu(t,x_r)
            \biggr)
            -
            \biggl(
                    \tilde{x}_t
                +
                \sum_{r=0}^t
                    \kappa(t,r)
                    \,
                    \mu(t,\tilde{x}_r)
            \biggr)
        \Biggr\|    
        \\ 
    \nonumber
        \le &
            \|x_t-\tilde{x}_t\|
            +
            L_{\mu}
            \sum_{r=0}^t
                    \kappa(t,r)
                    \|x_r-\tilde{x}_r\|
        \\ 
    \nonumber
        \le &
        \max\{1,L_{\mu}\} 
        \,
            \biggl(
                \|x_t-\tilde{x}_t\|
                +
                \sum_{r=0}^t
                    \kappa(t,r)
                    \|x_r-\tilde{x}_r\|
            \biggr)
    .
    \end{align}
\allowdisplaybreaks
We put it all together: by Cauchy-Schwarz inequality and since 
\[
\left[\sum_{r=0}^t\kappa(t,r)^2\right]^{1/2} \le \left[\sum_{r=0}^t\kappa(t,r)\right]^{1/2} \le 1.
\]
We therefore, find that
\[
\begin{split}
    d_{\mathfrak{J}}\Big(
     &\Psi_t(x_{[0:t]},\operatorname{vec}\circ s_{[0:t]}), 
\Psi_t(\tilde{x}_{[0:t]},\operatorname{vec}\circ\tilde{ s}_{[0:t]}) \Big) \\
&\le 2\left(\max\{1,L_{\mu}\} 
+\max\{1,L_{\sigma}\} 
\frac{\sqrt{2}}{M+R}\sinh\Big(\frac{M+R}{\sqrt{2}}\Big)
\right) \|x_{[0:t]}-\tilde{x}_{[0:t]}\| \\
&\quad + 2 \max\{1,L_{\sigma}\} \frac{\sqrt{2}}{M+R}\sinh\Big(\frac{M+R}{\sqrt{2}}\Big)\| s_{[0:t]}-\tilde{ s}_{[0:t]}\|_F,
\end{split}
\]
because trivially $\|x_{t}-\tilde{x}_{t}\| \le \|x_{[0:t]}-\tilde{x}_{[0:t]}\|$. The Lemma has been proven.
\end{proof}

\indent We now have the following lemma.
\begin{lemma}[Dependence of $x_{[0:t]}\mapsto \mathbb{Q}_{x_{[0:t]}}$ on the realizations of $X_{[0:t]}$]
\label{prop:VMP}
\hfill\\
Suppose that Assumption~\ref{ass:uniformboundedness__SProcecess} and the Volterra kernel $\kappa$ satisfies either of the following decay conditions
\begin{enumerate}
    \item[(i)] \textbf{Exponential decay:}  For some $0<\alpha<1$ and $C>0$, $\kappa(T,r) \le C \, \alpha^{T-r}$ for all integers $0\le r\le T,T>0$;
    \item[(ii)] \textbf{Polynomial decay:} For some $C>0>\alpha$, $\kappa(T,r) \le C(T-r)^{\alpha}$ for all integers $0\le r <T$; 
\end{enumerate}

Fix $T\in \mathbb{N}_+$ and a compact $\mathcal{K}_{[0:T]}\subseteq \mathbb{R}^{(1+T)d}$.  
For every pair of integers $\memory,t$, with $0\le \memory <t \le  T$, there is a Lipschitz function $f^{(t,\memory)}:(\mathbb{R}^{(1+\memory)\,d},\|\cdot\|)\rightarrow (\mathcal{P}_1\mathcal{N}_{d},\mathcal{W}_1)$ 
satisfying 
\[
        \mathcal{W}_1\big(
                \mathbb{Q}_{x_{[0:t]}}
            ,
                f^{(t,\memory)}(x_{[t-\memory:t]})
        \big)
    \le 
        c
            \,
        \operatorname{diam}(\mathcal{K}_{[0:T]})
        \,
        \begin{cases}
            \frac{C\alpha}{\alpha-1}\,(\alpha^t-\alpha^{\memory})
            &
            \mbox{if (i) holds}
            \\
            C\,(\memory+1)^{\alpha}\,(t-\memory) 
            &
            \mbox{if (ii) holds}
        \end{cases}
\]
for each $x_{[0:T]}\in \mathcal{K}_{[0:T]}$, where 
$c>0$ only depends on $M,R>0$ (in Assumption~\ref{ass:uniformboundedness__SProcecess}) and on $L_{\mu}$.
\end{lemma}

\begin{proof}\textbf{of Lemma}~\ref{prop:VMP}
Fix $0\le \memory <T$ and $\memory < t \le T$.  Define the map
$\iota^{(t,\memory)}:\mathbb{R}^{(1+\memory)\,d}\rightarrow \mathbb{R}^{(1+t)\,d}$ sending any $x_{[t-\memory:t]}\in \mathbb{R}^{(1+\memory)\,d}$ to $(0,\dots,0,x_{t-M},\dots,x_t)\in \mathbb{R}^{(1+t)\,d}$
and define the map $f^{(t,\memory)}$ by
\begin{equation}
        f^{(t,\memory)}(x_{[t-\memory:t]}) 
    \eqdef
        \mathbb{Q}_{
            \iota^{(t,\memory)}(x_{[t-\memory:t]})
        }
\end{equation}
for $x_{[t-\memory:t]}\in \mathbb{R}^{(1+\memory)\,d}$.  Since $\mathbb{Q}_{x_{[0:t]}}\in \mathcal{P}_1(\mathcal{N}_{d},\mathfrak{J})$ for each $x_{[0:t]}$ then $f^{(t,\memory)}$ takes values in $\mathcal{P}_1(\mathcal{N}_{d},\mathfrak{J})$.  

Fix an arbitrary $x_{[0:t]}\in \mathbb{R}^{(1+t)d}$.  Since the $1$-Wasserstein distance $\mathcal{W}_1$ between $\mathbb{Q}_{x_{[0:t]}}$
and $f^{(t,\memory)}(x_{[t-\memory:t]})$ is defined as the infimum over all transport plans/couplings having these two laws as marginals, we have
\begin{equation}
\label{eq:prop:VMP__SimpleUpperBoundofW1}
\begin{split}
\mathcal{W}_1\big(
                \mathbb{Q}_{x_{[0:t]}}
            ,
                f^{(t,\memory)}(x_{[t-\memory:t]})
        \big)
    & \le   
        \mathbb{E}\Big[ 
            d_{\mathfrak{J}}\big(
                \Psi_t(x_{[0:t]},\operatorname{vec}\circ\mathbf{S}_{[0:t]})
            ,
                \Psi_t(\iota^{(t,\memory)}(x_{[t-\memory:t]}),\operatorname{vec}\circ\mathbf{S}_{[0:t]})
            \big)
        \Big]
.
\end{split}
\end{equation}
By Assumption~\eqref{ass:uniformboundedness__SProcecess}, we may apply Lemma~\ref{lem:lipshcitzBound_on_RCD} to upper-bound the quantity under expectation in~\eqref{eq:prop:VMP__SimpleUpperBoundofW1} via the estimate~\eqref{eq:lem:lipshcitzBound_on_RCD__PreciseBound}.  This yields
\begin{equation}
\label{eq:prop:Vwarmup}
\resizebox{0.9\linewidth}{!}{$
\begin{aligned}
\mathbb{E}\Big[ 
            d_{\mathfrak{J}}\big(
                \Psi_t(x_{[0:t]},\operatorname{vec}\circ\mathbf{S}_{[0:t]})
            ,
                \Psi_t(\iota^{(t,\memory)} (x_{[t-\memory:t]}),\operatorname{vec}\circ\mathbf{S}_{[0:t]})
            \big)
        \Big]
    \le &
        \mathbb{E}\biggr[
        c
        \sum_{r=0}^{t}\,
                \kappa(t,r)
                \|x_r - \iota^{(t,\memory)}(x_{[t-\memory:t]})_r\|
        \biggl]
\end{aligned}
$}
\end{equation}
because $x_t=\iota^{(t,\memory)}(x_{[t-\memory:t]})_t$.
Since $0\le M<t$, the right-hand side of~\eqref{eq:prop:Vwarmup} can be bounded above as follows
\begin{equation}
\label{eq:prop:VMP__ControlBy___lem:lipshcitzBound_on_RCD__PreciseBound}
\resizebox{0.95\linewidth}{!}{$
\begin{aligned}
&
\mathbb{E}\Big[ 
            d_{\mathfrak{J}}\big(
                \Psi_t(x_{[0:t]},\operatorname{vec}\circ\mathbf{S}_{[0:t]})
            ,
                \Psi_t(\iota^{(t,\memory)} (x_{[t-\memory:t]}),\operatorname{vec}\circ\mathbf{S}_{[0:t]})
            \big)
        \Big]
    \le &
        \mathbb{E}\biggr[
        c
        \sum_{r=0}^{t}\,
                \kappa(t,r)
                \|x_r - \iota^{(t,\memory)}(x_{[t-\memory:t]})_r\|
        \biggl]
    \\
    = &
        \mathbb{E}\biggr[
        c
        \sum_{r=0}^{t-\memory-1}\,
                \kappa(t,r)
                \|x_r\|
        \biggl]
\\
    = &
        c
        \sum_{r=0}^{t-\memory-1}\,
                \kappa(t,r)
                \|x_r\|
\\
    \le &
        c
        \operatorname{diam}(\mathcal{K}_{[0:T]})
        \,
        \sum_{r=0}^{t-\memory-1}\,
                \kappa(t,r)
,
\end{aligned}
$}
\end{equation}
where the compactness of $\mathcal{K}_{[0:T]}$ implies that $\operatorname{diam}(\mathcal{K}_{[0:T]})$ is finite.
Set $C\eqdef c\,\operatorname{diam}(\mathcal{K}_{[0:T]})$. Combining~\eqref{eq:prop:VMP__SimpleUpperBoundofW1} with~\eqref{eq:prop:VMP__ControlBy___lem:lipshcitzBound_on_RCD__PreciseBound} we have
\[
        \mathcal{W}_1\big(
                \mathbb{Q}_{x_{[0:t]}}
            ,
                f^{(t,\memory)}(x_{[t-\memory:t]})
        \big)
    \le 
        C\sum_{r=0}^{t-\memory-1}\,\kappa(t,r)
    ,
\]
for each $x_{[0:T]}\in \mathcal{K}_{[0:T]}$.  
Applying either Lemmata~\ref{lem:RapidlyVanishingMemory_exponential_decay} or~\ref{lem:persistentmemory_polycase} yields the first conclusion.

It only remains to show that $f^{(t,\memory)}$ is Lipschitz.  Let $x_{[t-\memory:t]},\tilde{x}_{[t-\memory:t]}\in \mathbb{R}^{(1+\memory)\,d}$ and $t\in [[T]]$. Arguing similarly as above, using Lemma~\ref{lem:lipshcitzBound_on_RCD}, we find that
\allowdisplaybreaks
\begin{align}
\nonumber
        \mathcal{W}_1\big(
                f^{(t,\memory)}(x_{[t-\memory:t]})
            ,
                f^{(t,\memory)}(\tilde{x}_{[t-\memory:t]})
        \big)
    \le &
        \mathbb{E}\big[
            d_{\mathfrak{J}}\big(
                \Psi_t(\iota^{(t,\memory)} (x_{[t-\memory:t]}),\operatorname{vec}\circ\mathbf{S}_{[0:t]})
            ,
    \\
    \nonumber
               &\quad\quad\quad \Psi_t(\iota^{(t,\memory)} (\tilde{x}_{[t-\memory:t]}),\operatorname{vec}\circ\mathbf{S}_{[0:t]})
            \big)
        \big]
    \\
    \label{eq:lem:lipshcitzBound_on_RCD__PreciseBound__SecondPart}
    \le &
        c\,
        \mathbb{E}\biggl[
                \|\iota^{(t,\memory)} (x_{[t-\memory:t]})-\iota^{(t,\memory)}(\tilde{x}_{[t-\memory:t]})\|
    \nonumber
        \\
            & + 
                \|x_t -\tilde{x}_t\|
        \biggr]
    \\
    \nonumber
    = &
        c\,
        \mathbb{E}\biggl[
                    \|x_{[t-\memory:t]}-\tilde{x}_{[t-\memory:t]}\|
         + 
                \|x_t -\tilde{x}_t\|
        \biggr]
    \\
    \nonumber
    \le &
        2c\,
        \|x_{[t-\memory:t]}-\tilde{x}_{[t-\memory:t]}\|
,
\end{align}
where~\eqref{eq:lem:lipshcitzBound_on_RCD__PreciseBound__SecondPart} follows from the second part of Lemma~\ref{lem:lipshcitzBound_on_RCD}; relabelling $2c$ to $c$ yields the statement.
\end{proof}

Since $\Pi$ is a Lipschitz function of the path $x_{[0:t]}$ then, perturbations to the realized path $x_{[0:t]}$ result in at-most linear changes of the value of $\Pi_{x_{[0:t]}}$. 
To quantify this stability, let $\tilde{\boldsymbol{S}}_{\cdot}$ be another $d\times d$-dimensional matrix-valued processes, also defined on $(\Omega,\mathcal{F},\mathbb{P})$ and satisfying Assumption~\ref{ass:uniformboundedness__SProcecess}.

\begin{proposition}[{Lipschitz stability wrt. perturbations of the paths of $X_{\cdot}$ and of $\mathbf{S}_{\cdot}$}]
\label{lem:Stability_Estimates}
\hfill\\
Fix $t\in \mathbb{N}$ and a compact $\mathcal{K}_{[0:t]}\subseteq \mathbb{R}^{(1+t)\,d}$.  
If both $\mathbf{S}_{\cdot}$ and $\tilde{\boldsymbol{S}}_{\cdot}$ satisfy Assumption~\ref{ass:uniformboundedness__SProcecess} then, for each $x_{[0:t]}\in K_{[0:t]}$, $\beta\big(
                \operatorname{Law}(
                    \Psi_t(x_{[0:t]},\operatorname{vec}\circ \mathbf{S}_{[0:t]})
                )
            \big)$ is well-defined.  Furthermore
\[
    d_{\mathfrak{J}}\Big(
            \beta\big(
                \operatorname{Law}(
                    \Psi_t(x_{[0:t]},\operatorname{vec}\circ \mathbf{S}_{[0:t]})
                )
            \big)
        ,
            \beta\big(
                \operatorname{Law}(
                    \Psi_t(\tilde{x}_{[0:t]},\operatorname{vec}\circ \tilde{\boldsymbol{S}}_{[0:t]})
                )
            \big)
        \Big)
    \lesssim 
        \|x_{[0:t]}-\tilde{x}_{[0:t]}\|
        +
        \mathbb{E}[\| \mathbf{S}_{[0:t]}-\tilde{\boldsymbol{S}}_{[0:t]}\|_F]
,
\]
holds for each $x_{[0:t]},\tilde{x}_{[0:t]}\in \mathcal{K}_{[0:t]}$, where $\lesssim$ suppresses a positive constant depending only on $t$ and on $R$. 
In particular, the following holds for each $x_{[0:t]},\tilde{x}_{[0:t]}\in \mathcal{K}_{[0:t]}$
\[
    d_{\mathfrak{J}}\Big(
            \beta(\mathbb{Q}_{x_{[0:t]}})
        ,
            \beta(\mathbb{Q}_{\tilde{x}_{[0:t]}})
        \Big)
    \lesssim 
        \|x_{[0:t]}-\tilde{x}_{[0:t]}\|
.
\]
\end{proposition}

\begin{proof}\textbf{of Lemma}~\ref{lem:Stability_Estimates}
By Proposition~\ref{prop:NPC_FisherRao} $(\mathcal{N}_d,d_{\mathfrak{J}})$ is a global NPC
and by Assumption~\ref{ass:uniformboundedness__SProcecess} (ii) we may apply Theorem~\ref{thm:sturmthree} to deduce that for every $t\in \mathbb{N}_+$ 
\begin{align}
&
\notag
        d_{\mathfrak{J}}\Big(
            \beta\big(\operatorname{Law}(\Psi_t(x_{[0:t]},\operatorname{vec}(\mathbf{S}_{[0:t]}))\big)
        ,  
            \beta\big(\operatorname{Law}(\Psi_t(\tilde{x}_{[0:t]},\operatorname{vec}(\tilde{\boldsymbol{S}}_{[0:t]}))\big)
        \Big)
\\
\label{eq:PROOF__lem:Stability_Estimates__Definitions___ContractingBarycenterProperty__A}
\le 
&
        \mathcal{W}_{1}\Big(
            \operatorname{Law}(\Psi_t(x_{[0:t]},\operatorname{vec}(S_{[0:t]}))
        ,  
            \operatorname{Law}(\Psi_t(\tilde{x}_{[0:t]},\operatorname{vec}(\tilde{S}_{[0:t]}))
        \Big)
\\
\label{eq:PROOF__lem:Stability_Estimates__Definitions___ContractingBarycenterProperty__B}
\le 
&
    \mathbb{E}\big[
        d_{\mathfrak{J}}\big(
            \Psi_t(x_{[0:t]},\operatorname{vec}(S_{[0:t]})
        ,  
            \Psi_t(\tilde{x}_{[0:t]},\operatorname{vec}(\tilde{S}_{[0:t]})
        \big)
    \big]
,
\end{align}
where~\eqref{eq:PROOF__lem:Stability_Estimates__Definitions___ContractingBarycenterProperty__A} follows from the definition of the $1$-Wasserstein distance on $(\mathcal{N}_d,\mathfrak{J})$.  
Applying Lemma~\ref{lem:lipshcitzBound_on_RCD} to the right-hand side of~\eqref{eq:PROOF__lem:Stability_Estimates__Definitions___ContractingBarycenterProperty__B}, allows us to deduce the estimate
\begin{align}
\notag
& 
        d_{\mathfrak{J}}\Big(
            \beta\big(\operatorname{Law}(\Psi_t(x_{[0:t]},\operatorname{vec}(\mathbf{S}_{[0:t]}))\big)
        ,  
            \beta\big(\operatorname{Law}(\Psi_t(\tilde{x}_{[0:t]},\operatorname{vec}(\tilde{\boldsymbol{S}}_{[0:t]}))\big)
        \Big)
\\
\notag
\lesssim & 
    \mathbb{E}\Big[
        \,
        \big(
            \|x_{[0:t]}-\tilde{x}_{[0:t]}\|
            +
            \| \mathbf{S}_{[0:t]}-\mathbf{S}_{[0:t]}\|_F
        \big)
    \Big]
\\
= &
    \notag
    \big(
        \|x_{[0:t]}-\tilde{x}_{[0:t]}\|
            +
        \mathbb{E}\big[
            \| \mathbf{S}_{[0:t]}-\mathbf{S}_{[0:t]}\|_F
        \big]
    \big)
.
\end{align}
This completes the proof of the first claim; the second claim directly follows by definition of $\mathbb{Q}_{x_{[0:t]}}$.
\end{proof}

\begin{proof}\textbf{of Theorem}~\ref{thrm:VanishingMemoryProperty_Qx}
Since the conditions of Lemma~\ref{prop:VMP} are met, for each pair of non-negative integers $\memory,T$ with $\memory < T$ and every compact $\mathcal{K}_{[0:T]}\subseteq \mathbb{R}^{(1+T)d}$ there is a constant $C>0$, depending only on $\mathcal{K}_{[0:T]}$, and functions $f^{(t,\memory)}:(\mathbb{R}^{(1+\memory)\,d},\|\cdot\|)\rightarrow(\mathcal{P}_1(\mathcal{N}_{d},\mathfrak{J}),\mathcal{W}_1)$, for $T-\memory \le \memory$, satisfying 
\begin{equation}
\label{eq:tail_decay}
        \mathcal{W}_1\big(\mathbb{Q}_{x_{[0:t]}}
            ,
                f^{(t,\memory)}(x_{[t-\memory:t]})
        \big)
    \le 
        c
            \,
        \operatorname{diam}(\mathcal{K}_{[0:T]})
        \,
        \begin{cases}
            \frac{C\alpha}{\alpha-1}\,(\alpha^t-\alpha^{\memory})
            &
            \mbox{if (i) holds}
            \\
            C\,(\memory+1)^{\alpha}\,(t-\memory) 
            &
            \mbox{if (ii) holds}
        \end{cases}
\end{equation}
where $c>0$ depends only on $M$ and on $R$ and where $C\ge 0$ depends only on $K$, $T$, and on which of (i) or (ii), holds in Lemma~\ref{prop:VMP}. 

In particular, $\tilde{C}$ only depends on $\mathcal{K}_{[0:T]}$ and on $T$. By Proposition \ref{prop:sturmone}, $(\mathcal{N}_d,d_{\mathfrak{J}})$ is a non-positively curved metric space since it is complete, simply connected, and it has everywhere non-positive sectional curvatures.  Thus, Theorem \ref{thm:sturmthree} implies that the barycenter map $\beta:\big(\mathcal{P}_1(\mathcal{N}_d,\mathfrak{J}),\mathcal{W}_1\big)\rightarrow (\mathcal{N}_d,d_{\mathfrak{J}})$ is $1$-Lipschitz.  
Therefore, the existence and $1$-Lipschitz regularity of $\beta$, together with estimate in~\eqref{eq:tail_decay} imply 
\allowdisplaybreaks
\begin{align*}
\label{eq:PROOF_thrm:VanishingMemoryProperty_Qx__barycentricity}
        d_{\mathfrak{J}}\big(
            \beta(\mathbb{Q}_{x_{[0:t]}})
            \big)
        ,
            \beta\circ f^{(t,\memory)}(x_{[t-\memory:t]})
        \big)
    \le &
    \mathcal{W}_1\Big(
                \operatorname{Law}\big(\mathbb{Q}_{x_{[0:t]}})\big)
            ,
                f^{(t,\memory)}(x_{[t-\memory:t]})
        \Big)
    \\
    \le &
        \,
        c
        \,
        \operatorname{diam}(\mathcal{K}_{[0:T]})
        \begin{cases}
            \frac{C\alpha}{\alpha-1}\,(\alpha^t-\alpha^{\memory})
            &
            \mbox{if (i) holds}
            \\
            C\,(\memory+1)^{\alpha}\,(t-\memory) 
            &
            \mbox{if (ii) holds}
        \end{cases}
,
\end{align*}
for each $x_{[0:T]}\in \mathcal{K}_{[0:T]}$ and each $T-\memory\le t\le T$.
Lastly, for each $T-\memory\le t\le T$, set $\tilde{f}^{(t,\memory)}\eqdef \beta\circ f^{(t,\memory)}$ to obtain the conclusion.  Relabelling $c\eqdef C_{M,R,\mu}$ yields the conclusion.
\end{proof}

Finally, we demonstrate the smoothness of the Gaussian random projection.
\begin{proof}\textbf{of Proposition}~\ref{cor:DeterministicHighRegularity}
    Fix any $t\in \mathbb{N}_+$ and any $x_{[0:t]}\in \mathbb{R}^{(1+t)d}$.  Observe that
    \begin{equation}
    \label{eq:PROOF_cor:DeterministicHighRegularity__Simplification}
            \Pi_{x_{[0:t]}}
        = 
            \beta\big(\operatorname{Law}(\Psi_t(x_{[0:t]},\boldsymbol{0}))\big)
        =
            \beta\big(\delta_{\Psi_t(x_{[0:t]},\boldsymbol{0})}\big)
        =
            \delta_{\Psi_t(x_{[0:t]},\boldsymbol{0})}
    ,
    \end{equation}
where $\boldsymbol{0}=(0,\dots,0)\in \operatorname{Sym}(d)^{1+t}$ and where the right-hand side of~\eqref{eq:PROOF_cor:DeterministicHighRegularity__Simplification} follows from the fact that $\mathbf{S}_{[0:t]}=0\,$ $\mathbb{P}$-a.s.\ and that $\beta$ is a (contracting) barycenter map then $\beta(\delta_{\mu})=\mu$ for every $\mu\in \mathcal{N}_{d}$.  
Since $\psi$ is real-analytic (whence it is $C^{\infty}$; i.e.\ smooth) and since the composition and sums of $C^{\infty}$ (smooth) maps defined a $C^{\infty}$ (smooth) map then each $\Psi_t$ is $C^{\infty}$ (smooth) map if $\mu$ and $\sigma$ are.  Therefore, the map $x_{[0:T]}\mapsto \Psi_t(x_{[0:t},\boldsymbol{0})=\Pi_{x_{[0:t]}}$ is smooth.
\end{proof}

\section{{Proof of universality in the static case - Theorem~\ref{theorem:optimal_GDN_Rates__ReLUActivation} and Corollary~\ref{cor:Universality}}}
\label{s:Proof_Static_Case}

Before proceeding with the proof of Theorem~\ref{theorem:optimal_GDN_Rates__ReLUActivation}, we give some preparatory material. First, we state and prove a special case of the universal approximation theorem of \cite{galimberti2022designing}, Theorem 1; see Subsection \ref{s:Approx_Euclidean__ss:Feedforward}. Second, we provide a Universal Approximation result for a general $f$ and finite $\mathcal{K}$; see Subsection \ref{s:Universal_general_f_finite_K}.
\subsection{A brief review of feedforward theory} 
\label{s:Approx_Euclidean__ss:Feedforward}
Lemma \ref{lem:ReLU_Approximation} is a multi-dimensional consequence of the main results of \cite{ShenYangHaizhaoZhang_OptimalReLU_2022} and \cite{LuShenYangZhang_2021_UATRegularTargets}.
We record the results here, since they can be obtained by modifying only one step of the proof of Theorem 1 in \cite{galimberti2022designing}, by decoupling the condition that the width and depth parameters are set to be equal. In particular, in this formulation, the ReLU feed-forward network's depth is kept as a hyperparameter.  This is useful since it will allow us to synchronize the depths of our convolutional neural networks, and to avoid parity issues or the need to deploy a trainable PReLU activation function \citep{Acciaio2022_GHT} capable of implementing the identity.  
%%%
\begin{lemma}[Universal approximation for ReLU feedforward Neural Networks]
\label{lem:ReLU_Approximation}
Fix positive integers $d^{\star}, D$, let $\mathcal{K}$ be a non-empty compact of $\mathbb R^{d^{\star}}$. Consider $f:\mathcal{K}\rightarrow \mathbb{R}^D$ such that either:
\begin{enumerate}
    \item[(i)] $f$ extends to an $\lambda$-Lipschitz $C^{k}$ function $F:\mathbb{R}^{d^{\star}}\rightarrow \mathbb{R}^{D}$,
    \item[(ii)] $f$ extends to a $C^{\alpha}$-H\"{o}lder function $F:\mathbb{R}^{d^{\star}}\rightarrow \mathbb{R}^{D}$ for some $0<\alpha \le 1$ with H\"{o}lder constant $\lambda$. 
\end{enumerate}
Then, for every ``approximation error'' $\varepsilon$ and every ``depth hyperparameter'' $J\in \mathbb{N}_+$ there exists a ReLU feedforward neural network $\hat{f}:\mathbb{R}^{d^{\star}}\rightarrow \mathbb{R}^{D}$ satisfying the uniform estimate
\[
        \max_{x\in \mathcal{K}}\,
            \big\|
                f(x)
                -
                \hat{f}(x)
            \big\|
        <
            \varepsilon
    .
\]
Furthermore the depth and width of $\hat{f}$ are recorded in Table~\ref{tab:ReLU_Rates_Summary}.  
\end{lemma}

\begin{table}[ht!]%[H]
    \centering
	\ra{1.3}
    \caption{\textbf{Approximation Rates - ReLU feedforward Neural Network with fixed depth} 
    \hfill\\
    %%  Smooth Case
    In case (i), the constants in Table~\ref{tab:ReLU_Rates_Summary} are $C_0\eqdef d^{\star}(D-1)$, $
            C_1
        \eqdef 
            %%%% Depth of Individual Network
            17k^{d^{\star}+1}3^{d^{\star}}$ and $C_{\bar{F}}\eqdef \max_{i=1,\dots,D}\, \| F_i\circ W\|_{C^k([0,1]^d)}$; where $W$ is any bijective affine self-map on $\mathbb{R}^{d^{\star}}$ map satisfying $W(K)\subseteq [0,1]^{d^{\star}}$.  %
    \hfill\\
    %%  Holder Case
    In case (ii), the constants in Table~\ref{tab:ReLU_Rates_Summary} are $C_0\eqdef d^{\star}(D-1)$, $C_1\eqdef 3^{d^{\star}+3}$, $C_2\eqdef 19 + 2d^{\star}$% I added the extra parallelization 1 here
    , $C_3\eqdef (131 (d^{\star}D)^{1/2})$, and $V:[0,\infty)\rightarrow [0,\infty)$ is the ``special function'' defined as the inverse of $t\mapsto t^2\log_3(t+2)$.%
    \hfill\\
    %% No-Regularity
    In case (iii), 
    the constant in Table~\ref{tab:ReLU_Rates_Summary} is $C=
    (1+\memory)d(D-1)
            +
        9
        \big(
\operatorname{diam}(f(\mathcal{K}_{[0:\memory]}))
        \,
        \Big(
            \frac{
                (1+\memory)d
                }{2(1+\memory)d+2}
        \Big)^{1/2}
        \,
        % C_{N,D}
        DN
    $.}
    %%%
    \label{tab:ReLU_Rates_Summary}
    % \resizebox{\columnwidth}{!}{%
    \begin{tabular}{@{}lll@{}}
    \cmidrule[0.3ex](){1-2}
	\textbf{Hyperparam.} & \textbf{Exact Quantity - High Regularity } (Case (i))\\
    \midrule
        Depth & 
            $
            %%% Cost of Deep Parallelization
            D\Big(1 
                +
                %%%% Depth of Individual Network
                18 k^2(J+2) \log_2(4J) + 2d^{\star}
            \Big)
            $
        \\
        Width & 
          $
            %%% CONSTANT
            %%% Cost of Deep Parallelization
            C_0
            +
            %     %%%% Depth of Individual Network
            %     17k^{d^{\star}+1}3^{d^{\star}}\, 
            C_{1}
                %%% RATE
                \Big(2 + 
                    %N
                    %% Definition of N
                    \Big\lceil 
                        J\Big(\frac{C_{\bar{f}} D^{1/2}}{\varepsilon}\Big)^{d^{\star}/(2k)}
                    \Big\rceil 
                \Big)\log_2\Big(8
                    %N
                    \Big\lceil 
                        J\Big(\frac{C_{\bar{f}} D^{1/2}}{\varepsilon}\Big)^{d^{\star}/(2k)}
                    \Big\rceil 
                \Big)
            $
    \\
    \arrayrulecolor{lightgray}
    \cmidrule[0.3ex](){1-2}
    \textbf{Hyperparam.} &  \textbf{Exact Quantity - Low Regularity -}  (Case (ii)) 
    \\
    \midrule
        Depth &  
            %%%% 
            $
                %%% Cost of Deep Parallelization
                D
                \Big(
                % 1+
                    %%%% Depth of Individual Network
                    11 J + C_2
                \Big)
            $
        \\
        Width
        &
            $
            %%% Cost of Deep Parallelization
                    C_0
                +
            %%%% Width of Individual Network
                C_1\,
                \max\big\{
                    d^{\star}
                    \big\lfloor
                    \big(
                    % Expression for N 
                    %%%%
                            \big\lceil
                            V\big(
                                J^{-2}\,
                                (
                                    C_1 
                                    \varepsilon
                                    /\lambda
                                )^{d^{\star}/\alpha}
                            \big)
                        \big\rceil
                    %%%%
                    \big)^{1/d^{\star}}
                    \big\rfloor
                ,
                    2 +
                    % Expression for N 
                        %%%%
                                \big\lceil
                                V\big(
                                    J^{-2}\,
                                    (
                                        C_1 
                                        \varepsilon
                                        /\lambda
                                    )^{d^{\star}/\alpha}
                                \big)
                            \big\rceil
                        %%%%
            \big\}
            $
    \\
\arrayrulecolor{lightgray}
    \cmidrule[0.3ex](){1-2}
    \textbf{Hyperparam.} &  \textbf{Exact Quantity - No Regularity - Finite} $\#\mathcal{K}=N$  (Case (iii)) 
    \\
    \midrule
      Depth & $
      D\,\big(
        28\,N^2 \,J
    \big)
    $
\\
Width & 
    $\lceil (\varepsilon^{-1}\,C)^{1/(2NJ)}\rceil - 1
        \big)
        +2N-8
    $
    \\
    \arrayrulecolor{lightgray}\hline
    		\end{tabular}
    		% }% END Resize box
\end{table}

\subsection{{Static universality: The case of general \texorpdfstring{$f$}{f} and finite \texorpdfstring{$\mathcal{K}$}{domain}}}\label{s:Universal_general_f_finite_K}

\begin{lemma}[Universal approximation: General $f$ and finite $\mathcal{K}$]
\label{lem:ReLUMLPFiniteSetsNoRegularityTarget}
Let $M,d,D,N\in \mathbb{N}$. 
Fix any $N$-point subset $\mathcal{K}_{[0:\memory]}\subseteq \mathbb{R}^{(1+\memory)d}$ and any function $f:\mathbb{R}^{(1+\memory)d}\rightarrow \mathbb{R}^D$. For every ``approximation error'' $\varepsilon>0$ and every ``depth hyperparameter'' $J\in \mathbb{N}_+$, there exists a ReLU neural network $\hat{f}:\mathbb{R}^{(1+\memory)d}\rightarrow\mathbb{R}^D$ satisfying the uniform estimate
\[
    \max_{x\in \mathcal{K}_{[0:\memory]}}\,
            \big\|
                    f(x)
                -
                    \hat{f}(x)
            \big\|
        <
            \varepsilon
.
\]
Moreover, the depth and width of $\hat{f}$ are:
\begin{enumerate}
    \item \textbf{Depth:} $D\,\big(28\,N^2 \,J\big)$,
    \item \textbf{Width:} $D\,\big((1+\memory)d(D-1)+ 9(\lceil (\varepsilon^{-1}\,C)^{1/(2NJ)}\rceil - 1)+2N-8\big)$,
    \item \textbf{Number of trainable parameters:} $
                    % Correction Factor
                            \big(
                                (11/4)\,(1+\memory)dD -1
                            \big)
                            % Multiple from Depth
                            \\
                            \quad \times
                            \big(
                                28\,N^2 \,J + 1
                            \big)
                            % Parameters from Width
                            \Big(
                                (1+\memory)d(D-1)
                                    +
                                9\big(\lceil (\varepsilon^{-1}\,C)^{1/(2NJ)}\rceil - 1\big)+2N-8
                            \Big)^2
    +
        2(1+\memory)\,d + 2D$
\end{enumerate}
where $C>0$ is a constant independent of $J$ and of $\varepsilon$.
\end{lemma}
\begin{proof}\textbf{of Lemma~\ref{lem:ReLUMLPFiniteSetsNoRegularityTarget}}
The proof consists of several \textit{Steps}.\\
\textbf{Step 1: Re-scaling to the unit cube}
\hfill\\
We begin by enumerating $\mathcal{K}_{[0:\memory]} = \big\{x_{[0:\memory]}^{(n)}\big\}_{n=0}^{N-1}$.  
Clearly, there exists some {invertible} affine maps $A:\mathbb{R}^{(1+\memory)d}\rightarrow \mathbb{R}^{(1+\memory)d}$ and $\tilde{A}:\mathbb{R}^D\rightarrow \mathbb{R}^D$ ``re-normalizing the linearized data to the unite cubes''; i.e.\ satisfying
\begin{equation}
\label{PROOF_lem:UA_Finite_Data__renormalization}
    A\big(
    \{x^{(n)}\}_{n=1}^{N}
    \big) \subseteq [0,1]^{(1+\memory)d}
\qquad 
\mbox{ and }
\qquad 
    \tilde{A}\big(
    \{f(x^{(n)})\}_{n=1}^{N}
    \big) \subseteq [0,1]^D
.
\end{equation}
Accordingly, for each $n=1,\dots,N$ we label these ``re-normalized datasets'' by
\begin{equation}
\label{PROOF_lem:transformed_data}
{u^{(n)}}  \eqdef A (x^{(n)})
\qquad 
\mbox{ and }
\qquad 
{w^{(n)}} \eqdef \tilde{A} (f(x^{(n)}))
.
\end{equation}
Furthermore, Jung's Theorem, see e.g.~\citep[Theorem A]{LangSchroeder_Jung_1997_AnnGlobAnalGeom} (in the case where the Alexandrov curvature $\kappa=0$), implies that there exists $y_0\in \mathbb{R}^D$ such that
\begin{equation}
\label{eq:Jungs_Theorem_Simple_AffineMap}
    \tilde{A}(y)
        =
        \frac1{\operatorname{diam}(f(\mathcal{K}_{[0:\memory]}))}
        \,
        \biggl(
            \frac{2(1+\memory)d+2}{
                (1+\memory)d
                }
        \biggr)^{1/2}
        \,
        (y-y_0)
    \eqdef 
        D_{\tilde{A}}y
        +
        \tilde{a}
\end{equation}
for each $y\in \mathbb{R}^D$.  A similar expression of $A$ is also implied by Jung's Theorem.  In particular, the affine map $\tilde{A}=D_{\tilde{A}}\cdot + \tilde{a}$ (resp. $A=D_A\cdot + a$) is defined by multiplication against a square diagonal matrix $D_{\tilde{A}}$ (resp. $D_A$) and a shift by a vector $\tilde{a}\in \mathbb{R}^{(1+\memory)d}$ (resp. $a\in \mathbb{R}^D$) and therefore 
\begin{equation}
\label{eq:sparseparametercount_shiftandscale}
        \|\tilde{A}\|_0+\|\tilde{a}\|_0 \le 2(1+\memory)d
    \mbox{ and }
        \|A\|_0+\|a\|_0 \le 2D
.
\end{equation}
We will come back the the parameter count in~\eqref{eq:sparseparametercount_shiftandscale} towards the end of the proof.
\hfill\\
\textbf{Step 2: A polynomial extension of the target function on the transformed dataset}
\hfill\\
Note that the squared Euclidean norm $\|\,\cdot\,\|$ is a polynomial function.  Define the polynomial function $F:\mathbb{R}^{(1+\memory)d}\rightarrow \mathbb{R}^{D}$ of degree $2N$ by 
\[
\begin{aligned}
    F(u) 
        & \eqdef 
    \sum_{\tilde{n}=1}^N
        \left(
            \prod_{
            \underset{n\neq \tilde{n}}{n=1,\dots,N}
            }\,
                \Biggl(
                    \frac{
                        \|{u^{(n)}} - u\|^2
                    }{
                        \|{u^{(n)}}-u^{(\tilde{n})}\|^2
                    }
                \Biggr)
        \right)
            \cdot
            {w^{(n)}}
    \\
        & = 
    \sum_{\tilde{n}=1}^N
        \Biggl(
            c_{\tilde{n}}\,
            \prod_{
            \underset{n\neq \tilde{n}}{n=1,\dots,N}
            }\,
                \biggl(
                        \sum_{k=1}^{d^{\star}}\,
                            (u^{(n)}_k-u_k)^2
                \biggr)
        \Biggr)
            \cdot
            {w^{(n)}}
    ;
\end{aligned}
\]
where $c_{\tilde{n}}\eqdef \prod_{n=1,\dots,N;\,n\neq \tilde{n}}\,(\|{u^{(n)}}-u^{(\tilde{n})}\|^2)^{-1}$.
By construction, the polynomial function $F$ interpolates the re-normalized datasets in~\eqref{PROOF_lem:transformed_data}; that is: for each $n=1,\dots,N$ we have
\begin{equation}
\label{PROOF_lem:interpolator__data}
        F({u^{(n)}}) 
    = 
        {w^{(n)}} 
.
\end{equation}
Depending now on a few parameters, we approximate the function $F$ by a ``small'' ReLU feed-forward neural network.  
\hfill\\
\textbf{Step 3: Efficient approximation of the target functions polynomial extension}
\hfill\\
Fix a ``width hyperparameter'' $w$ and a ``depth hyperparameter'' $J$, both of which are positive integers and such that the value of $w$ will be determined at the end of the proof.
By construction, each of the components $F(x)_i,i=1,\dots, D$ of the function $F$ are polynomial.  
Thus, \citep[Proposition 4.1]{LuShenYangZhang_2021_UATRegularTargets} implies that for each $i=1,\dots,D$ there are deep ReLU networks $f^{(i)}:\mathbb{R}^{(1+\memory)d}\rightarrow \mathbb{R}$ each having width $9w+(2N)-8$ and depth $7(2N)^2J$, and each satisfying the uniform estimate
\begin{equation}
\label{PROOF_lem:polynomial_estimates}
    \max_{i=1,\dots , D}\,
        \max_{x\in [0,1]^{(1+\memory)d}}\, 
            \abs{f^{(i)}(x) - [F(x)]_i} 
    \le 
        C_N
        \,
        \frac1{
            (w+1)^{(2N)\,J}
        }
;
\end{equation}
where $C_N\eqdef 9(2N) \ge 0$; in particular, $C_N$ is independent of $w$ and $J$.  
We would like to efficiently ``parallelize'' these deep feedforward networks; that is, we seek a deep ReLU feedforward neural network $\tilde{f}:\mathbb{R}^{(1+\memory)d}\rightarrow \mathbb{R}^D$ satisfying
\begin{equation}
\label{PROOF_lem:parallelization}
        \tilde{f}(u) 
    = 
        \sum_{i=1}^{D}\, 
        f^{(i)}(u)\cdot e_i,
\end{equation}
for all $u\in \mathbb{R}^{(1+\memory)d}$, where $e_1,\dots,e_{D}$ denotes the standard orthonormal basis of $\mathbb{R}^{D}$.  
Since the (activation) function $\operatorname{ReLU}$ has the $2$-identity property (see \citep[Definition 4]{cheridito2021efficient}) then \cite[Proposition 5]{cheridito2021efficient} implies that there exists a deep ReLU feedforward neural network $f:\mathbb{R}^{(1+\memory)d}\rightarrow \mathbb{R}^{D}$ satisfying~\eqref{PROOF_lem:parallelization} and whose
\[
\begin{aligned}
\mbox{\textit{\textbf{Depth}}} & \mbox{ is at-most }  D\,\big(
                                                    28\,N^2 \,J
                                                \big)
\\
\mbox{\textit{\textbf{Width}}} & \mbox{ is at-most }  D\,\big(
                                                (1+\memory)d(D-1)
                                                        +
                                                    9w+2N-8
                                                \big)
\\
\mbox{\textit{\textbf{N. Trainable parameters}} } & \le \,
                            % Correction Factor
                            \big(
                                (11/4)\,(1+\memory)dD -1
                            \big)
                            % Multiple from Depth
                            \\
                            & \quad \times
                            \big(
                                28\,N^2 \,J + 1
                            \big)
                            % Parameters from Width
                            \big(
                                (1+\memory)d(D-1)
                                    +
                                9w+2N-8
                            \big)^2
.
\end{aligned}
\]
Therefore, together~\eqref{PROOF_lem:polynomial_estimates} and~\eqref{PROOF_lem:parallelization} imply that
\begin{equation}
\begin{aligned}
\label{PROOF_lem:polynomial_estimates__parallelized__Extension}
        \max_{u\in [0,1]^{(1+\memory)d}}\, 
            \Big\|
                \tilde{f}(u) - F(u)
            \Big\|
    = &
        \max_{u\in [0,1]^{(1+\memory)d}}\, 
            \biggl\|
                \sum_{i=1}^D
                    \, 
                    (f^{(i)}(u) - [F(u)]_i)e_i
            \biggr\| 
    \\
    \le & 
        \max_{u\in [0,1]^{(1+\memory)d}}\, \sum_{i=1}^{D}
            \abs{f^{(i)}(u) - [F(u)]_i}
    \\
    \le &
        D\,
        \max_{u\in [0,1]^{(1+\memory)d}}\, 
            \Big\|
                f(u) - F(u)
            \Big\|
    \\ 
    \le & 
        D
        N
        C_N
        \,
        \frac1{
            (w+1)^{2N J}
        }
    .
\end{aligned}
\end{equation}
Restricting the estimate in~\eqref{PROOF_lem:polynomial_estimates__parallelized__Extension} from all of $[0,1]^{(1+\memory)d}$ down to the subset $\{u^{(n)}\}_{n=1}^N \subseteq [0,1]^{(1+\memory)d}$ and recalling the interpolation property of $F$ on the pairs $\{(u^{(n)},w^{(n)})\}_{n=1}^N$ in~\eqref{PROOF_lem:interpolator__data}, we deduce that
\begin{equation}
\begin{aligned}
\label{PROOF_lem:polynomial_estimates__parallelized}    
    \max_{n=1,\dots,N}\, 
        \|\tilde{f}(u^{(n)}) - w^{(n)}\|
    \le & 
        C_{N,D}
        \,\,
        \frac1{
            (w+1)^{2N\,J}
        }
    ,
\end{aligned}
\end{equation}
where $C_{N,D}\eqdef DN C_N $ and we note that the constant $C_{N,D}> 0$ is independent of $w$ and of $J$.

Next, we observe that the pre-composition and post-composition of any deep ReLU neural network by invertible affine self-maps is again a deep ReLU neural network of the same depth and width.  Therefore, 
\[
        \hat{f}
    \eqdef 
        \tilde{A}^{-1}\circ \tilde{f}\circ A
\]
is a ReLU neural network with the same depth and width as $\tilde{f}$.  Furthermore, is defined by at-most
\begin{equation}
\label{eq:Parameters_flipped_map}
        P
    +
        2(1+\memory)\,d + 2D
\end{equation}
trainable parameters, due to the sparse representation of $\tilde{A}^{-1}$ and $A$ and the parameter count in~\eqref{eq:sparseparametercount_shiftandscale}; here $P$ denotes the number of trainable parameters of $\tilde{f}$.  

Putting it all together, we find that
\allowdisplaybreaks
\begin{align}
\notag
        \max_{n=1,\dots,N}\,
            \|\hat{f}(x_n)-f(x_n)\|
    = & 
        \max_{n=1,\dots,N}\,
            \|\tilde{A}^{-1}\circ \tilde{f}\circ A(x_n)-f(x_n)\|
\\
\notag
    = & 
        \max_{n=1,\dots,N}\,
            \|\tilde{A}^{-1}\circ \tilde{f}\circ A(x_n)-\tilde{A}^{-1}\circ \tilde{A}\circ f(x_n)\|
\\
\notag
    = & 
        \max_{n=1,\dots,N}\,
            \|\tilde{A}^{-1}\circ \tilde{f}(u_n)-\tilde{A}^{-1}(w^{(n)})\|
\\
\notag
    \le & 
        \operatorname{Lip}(\tilde{A}^{-1})
        \max_{n=1,\dots,N}\,
            \|\tilde{f}(u_n)-w^{(n)}\|
\\
\notag
    \le & 
        \operatorname{Lip}(\tilde{A}^{-1})
        C_{N,D}
        \,\,
        \frac1{
            (w+1)^{2N\,J}
        }
\\
\label{eq:Simplifiation_Via_ClosedFormExpressionUsingJungsTheorem}
    = & 
        \operatorname{diam}(f(\mathcal{K}_{[0:\memory]}))
        \,
        \Big(
            \frac{
                (1+\memory)d
                }{2(1+\memory)d+2}
        \Big)^{1/2}
        \,
        C_{N,D}
        \,\,
        \frac1{
            (w+1)^{2N\,J}
        }
\\
\label{eq:justlabelingconstants}
    = & 
        C
        \,
        \frac1{
            (w+1)^{2N\,J}
        }
\end{align}
where~\eqref{eq:Simplifiation_Via_ClosedFormExpressionUsingJungsTheorem} follows from the expression for $\tilde{A}$ in~\eqref{eq:Jungs_Theorem_Simple_AffineMap} which implies the expression and Lipschitz constant of its inverse $\tilde{A}^{-1}(w)\mapsto
\operatorname{diam}(f(\mathcal{K}_{[0:\memory]}))
\,
\Big(
    \frac{
        (1+\memory)d
        }{2(1+\memory)d+2}
\Big)^{1/2}
\,
y + y_0
$ and where~\eqref{eq:justlabelingconstants} follows by definition of the constant $C\eqdef \operatorname{diam}(f(\mathcal{K}_{[0:\memory]}))
        \,
        \Big(
            \frac{
                (1+\memory)d
                }{2(1+\memory)d+2}
        \Big)^{1/2}\,C_{N,D}$.  We note that $C> 0$ and that it does not depend on $J$ nor does it depend on $w$.
Upon setting $w\eqdef \lceil (\varepsilon^{-1}\,C)^{1/(2NJ)}\rceil - 1$, 
the right-hand side of~\eqref{eq:justlabelingconstants} is bounded-above by $\varepsilon$ and we obtain 
the conclusion and $\hat{f}$ is as in Table~\ref{tab:ReLU_Rates_Summary} case (iii).
\end{proof}

The diagram in Figure~\ref{fig:Thrm_Reference} is used to illustrate the workflow used in the proof of Theorem~\ref{theorem:optimal_GDN_Rates__ReLUActivation}.
\begin{figure}[H]
    \centering
    \begin{tikzcd}
    \mathcal{K}_{[0:\memory]} \arrow[r, hook] \arrow[d, "\varphi", red] & \mathcal{N}^{1+\memory} \arrow[r, "f"] \arrow[d, "\varphi", red]
    & \mathcal{M} \\
    \prod_{m=0}^M\,\operatorname{Log}_{\bar{x}_m}(\widetilde{\mathcal{K}}_{m}) \arrow[r, hook] & \mathbb{R}^{(1+\memory)d}\, \arrow[r, "\tilde{f}"] 
    & \mathbb{R}^{D} \arrow[u, "\rho", blue]
    \end{tikzcd}
    \caption{Diagram chase in the proof of Theorem~\ref{theorem:optimal_GDN_Rates__ReLUActivation}.}
    \label{fig:Thrm_Reference}
\end{figure}

\begin{proof}\textbf{of Theorem}~\ref{theorem:optimal_GDN_Rates__ReLUActivation} 
The proof is divided into three steps, highlighted in the following.\\
{We first begin by introducing some convenient notation.  Fix $x^{\star}\in \mathcal{N}$, $y^{\star}\in \mathcal{Y}$, and define
\[
\begin{aligned}
\varphi:\mathcal{N}^{1+\memory}\ni
x_{[0:\memory]} & \mapsto \big(\operatorname{Log}_{\bar{x}_m}^h\circ \iota_{\mathcal{N},\bar{x}_m}(x_m)\big)_{m=0}^M \in \mathbb{R}^{(1+\memory)D}\\
\rho:\mathbb{R}^d\ni y & \mapsto \operatorname{Exp}^g_{\bar{y}}\circ \iota^{-1}_{\mathcal{M},y}(y) \in \mathcal{M}
.
\end{aligned}
\]
where $\bar{x}_{[0:\memory]}$ and $\bar{y}$ are defined as in~\eqref{eq:HopfRinow_Encoding}.
}

\noindent\textbf{Step 0: Preliminary definitions}\hfill\\
Fix a non-empty compact subset $\mathcal{K}_{[0:\memory]} \subseteq \mathcal{N}^{1+\memory}$, a map $f : \mathcal{N}^{1+\memory} \rightarrow \mathcal{M}$, and some pre-specified ``approximation error'' $0< \varepsilon \leq 1$. Chasing the diagram in Figure~\ref{fig:Thrm_Reference}, we define the following subset of $\mathcal{M}$:
\begin{equation*}
\widetilde{\mathcal{K}} \eqdef \rho\biggl(\Big\{ u\in \mathbb{R}^{D}:\, \inf_{v\in \rho^{-1} \circ f(\mathcal{K}_{[0:\memory]})}\,\|u-v\|\le 1
    \Big\}\biggr)
,
\end{equation*}
which is the closure of the $1$-fattening/thickening of the set $\rho^{-1}\circ f(\mathcal{K}_{[0:\memory]})$ with respect to the Euclidean distance on $\mathbb{R}^D$. By construction, $\widetilde{\mathcal{K}}$ is a compact subset of $\mathcal{M}$, since $\{ u\in \mathbb{R}^{D}:\, \inf_{v\in \rho^{-1} \circ f(\mathcal{K}_{[0:\memory]})}\,\|u-v\|\le 1\}$ is compact in $\mathbb{R}^D$ and $\rho$ is a homeomorphism, {$\tilde{\mathcal{K}}$} contains $\rho^{-1} \circ f(\mathcal{K}_{[0:\memory]})$, and it does not depend on the approximation error $\varepsilon$.
\\
\textbf{Comment:} \textit{The role of the set $\widetilde{\mathcal{K}}$ is to define a compact region in $\mathcal{M}$, containing what will be the image of our approximation of $\rho^{-1}\circ f$ on $\mathcal{K}_{[0:\memory]}$, in which the local Lipschitz constant of the map $\rho$ may be controlled independently of $\varepsilon$.  Note, we need to do so, since $\rho$ is not globally Lipschitz.}\\
\noindent\textbf{Step 1: Reduction to a function approximation problem between Euclidean spaces}\hfill\\
For any map $F:\mathbb{R}^{(M+1) \,{D}}\rightarrow \mathbb{R}^{{d}}$
, the following preliminary estimates holds true. Since $\rho$ and $\varphi$ are global diffeomorphisms, then, in particular, they admit smooth inverses, and these smooth inverses are respectively defined on all of $\mathcal{M}$ and over all of $\mathcal{N}$. 

Chasing the diagram in Figure~\ref{fig:Thrm_Reference}, again, we define the induced maps $\tilde{f}:\mathbb{R}^{(1+\memory)D}\rightarrow \mathbb{R}^{{d}}$ and $\tilde{{F}}:\mathcal{N}^{1+\memory}\rightarrow \mathcal{M}$ by 
\[
\begin{aligned}
        \tilde{f}\eqdef \rho^{-1}\circ f\circ \varphi^{-1}
    \mbox{ and }
        \tilde{{F}}\eqdef \rho\circ {F}\circ \varphi
    .
\end{aligned}
\]
\noindent We deduce that:
\allowdisplaybreaks
\begin{align}
\notag
        \sup_{x_{[0:\memory]}\in \mathcal{K}_{[0:\memory]} } \, 
            d_g\big(f(x),\tilde{{F}}(x)\big)
        = & 
        \sup_{x_{[0:\memory]}\in \mathcal{K}_{[0:\memory]} } \, 
            d_g\big(
                    (\rho\circ \rho^{-1})\circ f\circ (\phi^{-1}\circ \phi)(x)
                ,
                    (\rho\circ \rho^{-1})\circ \tilde{{F}}\circ (\phi^{-1}\circ \phi)(x)
            \big)
\\
\notag
    = & 
        \sup_{x_{[0:\memory]}\in \mathcal{K}_{[0:\memory]} } \, 
            d_g\big(
                    \rho\circ \tilde{f}\circ \phi(x)
                ,
                    (\rho\circ \rho^{-1})\circ \tilde{{F}}\circ (\phi^{-1}\circ \phi)(x)
            \big)
\\
\notag
    = & 
        \sup_{x_{[0:\memory]}\in \mathcal{K}_{[0:\memory]} } \, 
            d_g\big(
                    \rho\circ \tilde{f}\circ \phi(x)
                ,
                    \rho\circ (\rho^{-1}\circ \rho)\circ {F} \circ (\phi \circ \phi^{-1})\circ \phi(x)
            \big)
\\
\notag
    = & 
         \sup_{x_{[0:\memory]}\in \mathcal{K}_{[0:\memory]} } \, 
            d_g\big(
                    \rho\circ \tilde{f}\circ \phi(x)
                ,
                    \rho\circ {F}\circ \phi(x)
            \big)
\\
\label{eq:1}
    = & 
        \sup_{u_{[0:\memory]}\in \phi ( \mathcal{K}_{[0:\memory]} )}\, 
            d_g\big(
                    \rho\circ \tilde{f}(u)
                ,
                    \rho\circ {F}(u)
            \big)
    .
\end{align}
Notice that if $g$ is any function satisfying the following uniform estimate on $\phi(\mathcal{K}_{[0:\memory]})$
\begin{equation}\label{eq:Prior_g_approximation}
\sup_{u \in \varphi(\mathcal{K}_{[0:\memory]})}\|\tilde{f}(u)-g(u) \| \leq 1,  
\end{equation}
then, ${F}\circ \varphi(\mathcal{K}_{[0:\memory]})\subseteq \rho^{-1}[\widetilde{\mathcal{K}}]$, and since $\rho$ is bijective, this implies that $\rho\circ {F} \circ \varphi(\mathcal{K}_{[0:\memory]}) \subseteq \widetilde{\mathcal{K}}$. Moreover, by construction, $\rho\circ\tilde{f} \circ \varphi(\mathcal{K}_{[0:\memory]}) \subseteq \widetilde{\mathcal{K}}$. Therefore the compact set $\widetilde{\mathcal{K}}$ contains $\rho \circ {F}(\phi ( \mathcal{K}_{[0:\memory]} )) \bigcup \rho\circ \tilde{f}(\phi (\mathcal{K}_{[0:\memory]} ))$.
Since smooth functions are locally-Lipschitz and since $\rho$ is smooth, then the restriction $\rho\vert {\widetilde{\mathcal{K}}}$ of $\rho$ to the compact set $\widetilde{\mathcal{K}}$ is Lipschitz.  
Denote its local-Lipschitz constant of $\rho$ thereon by $\operatorname{Lip}(\rho\vert {\widetilde{\mathcal{K}}})$.  Thus, under the Assumption that ${F}$ satisfies~\eqref{eq:Prior_g_approximation}, the right-hand side of~\eqref{eq:1} can be bounded-above as
\allowdisplaybreaks
\begin{align}
\label{eq:2}
    \sup_{u_{[0:\memory]}\in \phi ( \mathcal{K}_{[0:\memory]} )}\, 
            d_g\big(
                    \rho\circ \tilde{f}(u)
                ,
                    \rho\circ {F}(u)
            \big)
    \le & 
        \operatorname{Lip}\big(
                \rho\vert {\widetilde{\mathcal{K}}}
        \big)
        \,
        \sup_{u_{[0:\memory]}\in \phi ( \mathcal{K}_{[0:\memory]} )}\, 
            \|
                    \tilde{f}(u)
                -
                    {F}(u)
            \|
    .
\end{align}
Note that $\operatorname{Lip}(
                    \rho\vert{\widetilde{\mathcal{K}}}
\big)$ cannot be $0$ since it is a diffeomorphism and $\mathcal{M}$ was assumed to be of dimension at-least $1$; thus, set $
\bar{\varepsilon}
\eqdef 
\varepsilon
\,
\min\big\{
1
,\,
            {\operatorname{Lip}\big(
                    \rho
                \vert
                    \widetilde{\mathcal{K}}
            \big)}^{-1}
\big\}$.  

\noindent\textbf{Step 2: Universal approximation by ReLU feedforward Neural Networks:}\hfill\\
By Lemma~\ref{lem:ReLU_Approximation}, there exists a ReLU feedforward neural network
${F}:\mathbb{R}^{(M+1)\,{D}}\rightarrow \mathbb{R}^{{d}}$ satisfying the following uniform estimate: 
\begin{equation}
\label{eq:NN_Bound_OtherPaper}
        \sup_{u_{[0:\memory]}\in \phi ( \mathcal{K}_{[0:\memory]} )}\,
            \|
                    \tilde{f}(u)
                -
                    {F}(u)
            \|
    \le 
        \bar{\varepsilon}
    .
\end{equation}
In particular,~\eqref{eq:NN_Bound_OtherPaper} implies that ${F}$ satisfies~\eqref{eq:Prior_g_approximation} since $\varepsilon\le 1$.  
Inputting the estimate in~\eqref{eq:NN_Bound_OtherPaper} into the right-hand side of~\eqref{eq:2}, and thus~\eqref{eq:1}, yields 
\[
\begin{aligned}
        \sup_{x_{[0:\memory]}\in \mathcal{K}_{[0:\memory]} } \, 
            d_g\big(
                    f(x)
                ,
                   \tilde{{F}}(x)
            \big)
    \le &
        \operatorname{Lip}(\rho\vert \widetilde{\mathcal K})\, 
            \varepsilon 
        \min\Big\{
                1
            ,
                \frac1{\operatorname{Lip}(\rho\vert \widetilde{\mathcal{K}})}
            \big\}
    \\
    = &\,
        \varepsilon 
            \min\{
                \operatorname{Lip}(\rho\vert \widetilde{\mathcal{K}})
            ,
                1
            \}
    \\
    \le &\, 
        \varepsilon
    .
\end{aligned}
\]
Furthermore, the depth and width of $g$ are given in Table~\ref{tab:ReLU_Rates_Summary} with their specifications
$d^{\star}= (1+\memory)\,{D}$, $\varepsilon$ in their notation is set to $\varepsilon\min\{1,\operatorname{Lip}(\rho|\tilde{K}))$, and
$
        \lambda
    \eqdef 
        \operatorname{Lip}_{\alpha}( \rho^{-1}\circ f \circ \varphi^{-1}  \vert \phi(\overline{B(\mathcal{K}_{[0:\memory]},1)})
        )
    ,
$  is the $\alpha$-H\"{o}lder constant of $\rho^{-1}\circ f\circ \varphi^{-1}$ on the $1$-thickening of $\mathcal{K}_{[0:\memory]}$.\\
\indent In the case where $f$ is smooth, then since the composition of smooth functions is again smooth and all smooth functions are locally-Lipschitz then $\tilde{f}$ is smooth on $\mathbb{R}^{(1+\memory)D}$ and it is Lipschitz on the compact set $\varphi(\mathcal{K}_{[0:\memory]})$.  
% {\color{green}{
The depth and width of ${F}$ is recorded in Table~\ref{tab:ReLU_Rates_Summary} with $d^{\star}= (1+\memory)\,{D}$, $\varepsilon$ in their notation is set to $ \varepsilon \min\{1,\operatorname{Lip}(\rho|\tilde{K})\}$, $F= \rho^{-1}\circ f\circ \varphi^{-1}$, and $W$ is any bijective affine self-map on $\mathbb{R}^{(1+\memory)\,{D}}$ sending the compact set $\varphi(
\overline{B(\mathcal{K}_{[0:\memory]},1)}
)$ to $[0,1]^{(1+\memory){D}}$.
\end{proof}

\begin{proof}\textbf{of Corollary}~\ref{cor:Universality}
     Since $\mathcal{N}^{1+\memory}$ endowed with the product metric is Polish, we have that $\mathbb{P}\eqdef (X_1)_{\#}\mu$ is a Radon measure on $\mathcal{N}^{1+\memory}$. We consider the Borel function $\rho^{-1}\circ f:\mathcal{N}^{1+\memory}\to \mathbb R^D$. By Lusin theorem, we may find $\ell_\epsilon: \mathcal{N}^{1+\memory}\to \mathbb R^D$ continuous such that 
     \[
     \mathbb P\left[ 
        z\in \mathcal{N}^{1+\memory};\; \rho^{-1}\circ f(z) = \ell_\varepsilon(z)
     \right] > 1 -\epsilon
     .
     \]
     Therefore, the function $\rho\circ\ell_\epsilon:\mathcal{N}^{1+\memory}\to\mathcal{M}$ is continuous and $\mathbb P [z\in \mathcal{N}^{1+\memory};\; f(z) = \rho\circ\ell_\varepsilon(z) ] > 1 -\epsilon$. Since $\mathbb P$ is Radon, we may find $\mathcal{K}_{\epsilon} \subset \{z\in \mathcal{N}^{1+\memory};\; f(z) = \rho\circ\ell_\varepsilon(z)\}$ compact and such that $\mathbb P[ \mathcal{K}_\epsilon]>1-\epsilon$. 
     Applying Theorem~\ref{cor:Universality}, we obtain a GDN $\Hat{f}:\mathcal{N}^{1+\memory}\to\mathcal M$ such that
     \[
     d_h(\hat{f}(z),\rho\circ \ell_\epsilon(z))<\epsilon, \quad z\in\mathcal K_\epsilon
     .
     \]
     That is, we have
     \[
     d_h(\hat{f}(z),f(z))<\epsilon, \quad z\in\mathcal K_\epsilon.
     \]
     Since trivially, $\mathcal{K}_\epsilon\subset [z\in\mathcal N^{1+\memory};\;d_h(\hat{f}(z),f(z))<\epsilon ]$, the thus result follows. 
\end{proof}

\section{Proof of universality in the dynamic case - Theorem~\ref{thrm:approx_HGCNN}}
\label{s:Proof_Dynamic_Case}

\begin{proof}[{Theorem~\ref{thrm:approx_HGCNN}}]
Fix $Q,J\in \mathbb{N}_+$, define the time horizon $T\eqdef T_{\delta,Q}\eqdef \lfloor \delta^{-Q}\rfloor$, and let $\varepsilon>0$.
\hfill\\\textbf{Step 1: Obtaining proxy target functions via the finite virtual memory of $f$}
\hfill\\
In each case (1), (2), and (3), $f$ was assumed to have finite virtual memory with virtual memory $r\ge 0$ (see Definition~\ref{def:causal_maps}).  Therefore, there is an $\tilde{\memory}\eqdef 
\memory(\varepsilon,T,\mathcal{K})
\in \mathbb{N}$ with $
\tilde{\memory} \lesssim (\varepsilon/2)^{-r}
$ 
and functions $
f_1,\dots f_T \in C(\mathcal{N}^{\memory}, \mathcal{M})$ satisfying
\begin{equation}
\label{eq:causal_maps__insideproofs___epsover2}
    \max_{t \in [[T]]}
    \sup_{x \in \mathcal{K}} \,
        d_g\big(
                f(x)_t
            , 
                f_t(x_{(t-\memory: t]})
        \big)
    <
        \frac{\varepsilon}{2}
.
\end{equation}
Define $\memory\eqdef \max\{T,\tilde{\memory}\}$.
\hfill\\
\textbf{Comment:} \textit{Next, we approximate each $f_t$ with an ``expert'' GDN specialized at that time $t$.}
\hfill\\
\noindent
We consider three cases, depending on which of the assumptions on $f$ held.
\begin{enumerate}
    \item[(1)] If $f$ is $(r,k,\lambda)$-smooth then, 
    each $\{f_t\}_{t=0}^T\in C^{k,\lambda}(\mathcal{N}^{\memory}, \mathcal{M})$,
    \item[(2)] If $f$ is $(r,\alpha,\lambda)$-H\"{o}lder then, 
    each $\{f_t\}_{t=0}^T\in C^{\alpha,\lambda}(\mathcal{N}^{\memory}, \mathcal{M})$,
    \item[(3)] If each $\mathcal{K}_t$ is finite, then no additional condition on the $f_t$ is required.
\end{enumerate}
%%%
\noindent
\textbf{Step 2: Parallelized ``expert'' approximation on separate time-windows}
\hfill\\ 
Fix $\tilde{x}^{\star}\in \mathcal{N}$, {define $r\eqdef 
    \max_{x\in \mathcal{K}_{[0:T]}}
    \max_{t=0,\dots,T}
    \,
    d_h\big(
            x_t
        ,
            \tilde{x}^{\star}
    \big)$ and construct the compact subset $\mathcal{K}^{\star:T}$ of $\mathcal{N}^{1+T}$, containing $\mathcal{K}_{[0:T]}$ as follows:
\[
\mathcal{K}^{\star:T}
\eqdef 
    \begin{cases}
        \big\{
            x_{[0:T]}\in \mathcal{N}^{1+T}
            :
            \,
                \max_{t=0,\dots,T}\,
                    d_h(x_t,x^{\star})\le r
        \big\}
    & \mbox{ if } \#\bigcup_{t=0}^T\, \mathcal{K}_t = \infty 
    \\
        \big\{
            x_{[0:T]}\in \mathcal{N}^{1+T}
            :
            \,
                (\forall s=0,\dots,T)\,
                    x_s\in \cup_{t=0}^T\, \mathcal{K}_t
        \big\}
    & \mbox{ if } \#\bigcup_{t=0}^T\, \mathcal{K}_t < \infty.
    \end{cases}
\]
Observe, that $\mathcal{K}^{\star:T}$ is finite, if $\mathcal{K}_{[0:T]}$ is finite.  Furthermore, by the stars-and-bars problem from elementary combinatorics, we know that each $t\in \{\memory,\dots,T\}$
\begin{equation}
\label{eq:combinatorial_cardinality_bound}
\#K^{\star:T}_{(t-\memory:t]} 
\le 
N^{\star}_M
\eqdef
    \binom{N^{\star}_M + M -1}{M}
\mbox{ and }
    N^{\star:T}_{\memory}
    \eqdef \max_{t=0,\dots,T}\, \mathcal{K}_t
\end{equation}
where $K^{\star:T}_{(t-\memory:t]} = \{(x_{(t-\memory,\dots,t]}:\, \exists z_{[0:T]}\in \mathcal{K}^{\star:T}\, \forall t-\memory\le s\le t\,(x_{(t-\memory:t]})_s=(x_{[0:T]})_s\}$.
}

\noindent
Once for each $t\in [[T]]$, we apply Theorem~\ref{theorem:optimal_GDN_Rates__ReLUActivation}, to deduce that for each $t \in [[T]]$ there is are GDNs $\hat{f}_{\theta_t}$, satisfying 
\begin{equation}
\label{PROOF__thrm:approx_HGCNN__ApplyingApproximationLemma}
{
        \max_{x_{\cdot} \in \mathcal{K}^{\star:T}}\,
        \max_{t=\memory,\dots,T}
        \,
            d_g\big(
                f_t(x_{(t-\memory: t]}
                )
                        ,
                \hat{f}_{\theta_t}(x_{(t-\memory:t]})
            \big)
}       
        < 
            \varepsilon/2
,
    \end{equation}
where, for $t=\memory,\dots,T$, we have set the {distinguished base-point in~\eqref{eq:HopfRinow_Encoding} to be $x^{\star}_m=\tilde{x}^{\star}$, for each $m=0,\dots,\memory$.}  
Furthermore, the depth and width of each GDN $\hat{f}_{\theta_{T-M}},\dots,\hat{f}_{\theta_{T}}$ depends only {on:
\begin{enumerate}
    \item[(1)] \textbf{If $f$ is $(r,k,\lambda)$-Smooth}: the quantity $(r,k,\lambda)$ and on the diameter of $\mathcal{K}^{\star:T}$,
    \item[(2)] \textbf{If $f$ is $(r,\alpha,\lambda)$-H\"{o}lder}: the quantity $(r,\alpha,\lambda)$ and on the diameter of $\mathcal{K}^{\star:T}$,
    \item[(3)] \textbf{If $\mathcal{K}_{[0:T]}$ has finite Cardinality}: on $N^{\star:T}_{\memory}$ 
    which bounds the cardinality of $\mathcal{K}^{\star:T}$ above.
\end{enumerate}
as detailed in Table~\ref{tab:casestudies}, in each respective case.}
Note that each GDN $\hat{f}_{\theta_t}$ has the same domain and each $f_t$ has the same regularity {(i.e.\ as in cases $1$-$3$ above)}; therefore {Theorem~\ref{theorem:optimal_GDN_Rates__ReLUActivation} implies that the depth, width, multi-index $[\mathbf{d}]$, and parameter space $\mathbb{R}^{P([\mathbf{d}])}$ of each $\hat{f}_{\theta_{\memory}},\dots,\hat{f}_{\theta_T}$ are all the same.}
The quantities $\ell$ and $w$ are recorded in Table~\ref{tab:casestudies} with the following relevant constants
\begin{enumerate}
    \item[(i)] \textbf{Smooth case:} Since {$f$ is $(r,k,\lambda)$-smooth, then each $f_t$ belongs to $C^{k,\lambda}(\mathcal{N}^{1+\memory},\mathcal{M})$} define
    \[
            C_{\bar{f}}
        \eqdef 
            \max_{
            \underset{i=1,\dots,D}{t=T-\memory,\dots,T}
            }\, 
                    \frac{
                        \| (\rho^{-1}\circ 
                        f_t
                        \circ \phi^{-1}\circ W_t)_i\|_{C^k([0,1]^{d})}
                    }{
                    \max\{1,
                        \operatorname{Lip}\big(
                            \rho\vert 
                            f_t
                            \circ 
                            \varphi( 
                                \mathcal{K}^{\star:T}
                                )
                        \big)
                    \}
                    }
        ,
    \]
    where {for each $t=\memory,\dots,T$}, $W_{M},\dots,W_T:\mathbb{R}^{(1+\memory)\,d}\to \mathbb{R}^{(1+\memory)\,d}$ are affine bijections satisfying $W_t(\phi(\mathcal{K}^{\star:T}))\subseteq [0,1]^{(1+\memory)\,d}$.  
    {In particular, the definition of $\mathcal{K}^{\star:T}$, Lemma~\ref{lem:regularity_phi} (ii), and Jung's Thereom (see \cite{Jung1901}), there are $\tilde{x}_t\in \mathbb{R}^{(1+\memory)\,d}$ such that: for each $x\in \mathbb{R}^{d(1+\memory)}$
    \begin{equation}
    \label{eq:aff_recentring}
        W_t(x) \eqdef 
        \frac1{C_{\bar{x}}\sqrt{M}\,2r}\, 
        \sqrt{
            \frac{2d(1+\memory)+2}{d(1+\memory)}
        }
        \,(x-\tilde{x}_t)
    .
    \end{equation}}
    \item[(ii)] \textbf{H\"{o}lder case:} {Since $f$ is $(r,\alpha,\lambda)$-H\"{o}lder, then each $f_t$ belongs to $C^{\lambda}_{\alpha}(\mathcal{N}^{1+\memory},\mathcal{M})$}, set
    \[
            C_f
        \eqdef 
        \min_{t=T-\memory,\dots,\memory}\,
            \operatorname{Lip}_{\alpha}( \rho^{-1}\circ 
            f_t \circ 
            \varphi^{-1}  
        \vert 
        \phi(\mathcal{K}^{\star:T}))
    ,
    \]
    \textit{(note this is because we are maximizing $1/\operatorname{Lip}_{\alpha}( \rho^{-1}\circ f_t \circ \varphi^{-1}  \vert \phi(\mathcal{K}^{\star:T}))$ across $t=T-\memory,\dots,T$.}
\end{enumerate}

Combining the estimates in~\eqref{eq:causal_maps} and the family of estimates in~\eqref{PROOF__thrm:approx_HGCNN__ApplyingApproximationLemma} (indexed by $[[T]]$) we deduce that: for each $x\in \mathcal{K}$ and every $t=T-\memory,\dots,T$ the following holds
\begin{equation}
\begin{aligned}
        d_g\big(
            f(x)_t
        ,
            \hat{f}_{\theta_t}(x^{(n)}_{(t-\memory:t]})
        \big)
    \le &\,
        d_g\big(
            f(x)_t
        ,
            f_t(x_{(t-\memory:t]})
        \big)
    \\
    & +
        d_g\big(
            f_t(x_{(t-\memory:t]})
        ,
            \hat{f}_{\theta_t}(x_{(t-\memory:t]})
        \big)
    \\
    < &\,
        \varepsilon/2
        + 
        \varepsilon/2
    .
\end{aligned}
\end{equation}
\hfill\\
\textbf{Step 3: Weaving the approximators together}
\hfill\\
Fix a $Q\in \mathbb{N}_+$, with $Q+P([\bar{J}]) \ge \max\{12,P([\bar{J}])+1\}$.  Upon  applying the ``dynamic weaving lemma'' \cite[Lemma 5]{galimberti2022designing}, we deduce that there exists a $z \in \mathbb{R}^{P([\bar{J}])+Q}$, a linear map $L:\mathbb{R}^{P([\bar{J}])+Q}\rightarrow \mathbb{R}^{P([\bar{J}])}$, and a deep ReLU FFNN $\hat{f}:\mathbb{R}^{P([\bar{J}])+Q}\rightarrow \mathbb{R}^{P([\bar{J}])+Q}$ with 
\begin{enumerate}
    \item \textbf{Width:} $\big(P([\bar{J}]) +Q\big)T + 12$,
    \item \textbf{Depth:} $
        \mathcal{O}\Biggr(
        T\biggl(
        1+\sqrt{T\log(T)}\,\Big(1+\frac{\log(2)}{\log(T)}\,
        \biggl[
            C + \frac{\Big(\log\big(T^2\,2^{1/2}\big)-\log(\delta)\Big)}{\log(2)}
        \biggr]_+\Big)
        \biggr)
        \Biggr)$,
    \item \textbf{N. Parameters:} 
    \[
    \resizebox{1\linewidth}{!}{$
        \mathcal{O}\Bigl(
        T^3(P([\mathbf{d}])+Q)^2\,\left(1+(P([\mathbf{d}])+Q)
        \sqrt{T\log(T)}\,\left(1+\frac{\log(2)}{\log(T)}\,
            \left[
            C_d+
                \frac{\Big(\log\big(T^2\,2^{1/2}\big)-\log(\delta)\Big)}{\log(2)}
            \right]_+\right)\right)
        \Bigr)
    ,
    $}
    \]
\end{enumerate}
satisfying the interpolation
\begin{equation}
\label{PROOF__thrm:approx_HGCNN__Applying_Interpolation}
    \theta_t = L\big(h^{\circ t}(z)\big) \qquad \mbox{ for }t=T-\memory,\dots,T
;
\end{equation}
where $h^{\circ t}$ denotes the $t$-hold composition of $h$ with itself. Upon setting $d^{\star}\eqdef P([\bar{J}])+Q$ we obtain the conclusion.  
\end{proof}

\section{Memory persistence estimates for Volterra kernels}
In this appendix, we derive estimates of the error incurred by truncating the realized path in a Volterra series when predicting its evolution.  We consider three types of Volterra kernels, those which decay fast (exponentially) %, moderately (polynomially), 
and slowly (logarithmically).  
%%%
\begin{lemma}[Rapidly (exponential) vanishing memory]
\label{lem:RapidlyVanishingMemory_exponential_decay}
Suppose there are $C>0$, $0<\alpha{<1}$ such that for every pair of integers $0\le r\le t$ with $t\neq 0$
\[
        \kappa(t,r)
    \le 
        C\,\alpha^{t-r}
    .
\]
For every $t\in \mathbb{N}_+$ and each integer $0\le \memory < t$ one has
\[
                \sum_{r=0}^{t-\memory-1}\,\kappa(t,r) 
    \le 
        \frac{C\,\alpha}{\alpha-1} (\alpha^t-\alpha^{\memory})
\]
\end{lemma}
\begin{proof}
We have
\allowdisplaybreaks
\begin{equation}
        \sum_{r=0}^{t-\memory-1}\,
        \kappa(t,r) 
    \le 
        \sum_{r=0}^{t-\memory-1}\,
            C\,\alpha^{t-r}
=
        \frac{C\alpha}{\alpha-1}\,(\alpha^t-\alpha^{\memory})
\label{eq:case_by_base}
.
\end{equation}
\end{proof}

\begin{lemma}[Moderately (Polynomially) vanishing memory]
\label{lem:persistentmemory_polycase}
Suppose there is $\alpha\in\R, C>0$ such that for each pair of integers $0\le r< t$ 
\[
        \kappa(t,r)
    \le 
        C(t-r)^{\alpha}
.
\]
Then, for each $t\in \mathbb{N}_+$ and each integer $0\le \memory \le t$ one has
\begin{equation*}
        \sum_{r=0}^{t-\memory-1}
        \,
        \kappa(t,r)
    \le 
        C
        \,
        (\memory+1)^{\alpha}\,(t-\memory)
\end{equation*}
\end{lemma}
\begin{proof}
Employing the growth assumption $\kappa(t,r)\le C(t-r)^{\alpha}$ and the fact that $t-\memory-1<t$ we find that
\allowdisplaybreaks
\begin{align}
        \sum_{r=0}^{t-\memory-1}
        \,
        \kappa(t,r)
    \le & 
        \sum_{r=0}^{t-\memory-1}
        \,
            C(t-r)^{\alpha}
\\
    = &
        {C}\,
        \sum_{u=\memory+1}^{t}
        \,
            u^{\alpha}
.
\end{align}
Since $\alpha<0$ then, $\sum_{u=M+1}^{t}u^{\alpha}\le (\memory+1)^{\alpha}\,(t-\memory)$.
\end{proof}

\section{Experiment Details}
\label{a:Experiment_Details}
We include the details of the experiments in Section~\ref{s:Numerics}. 
We first explicitly describe the algorithms we used to implement generate the sample paths from the Volterra process in~\eqref{eq:Volterra__ExperimentsS}, to compute its Gaussian random projections, and to train the HGN from this data.  
We then provide details on the hardware used to train these models.

\subsection{Experiment and Compute Details}
\label{s:ExperimentDetails}
The architecture used to tackle those problems (identical for all the models) had six layers with a maximum size of 512 in the GDN part and eight layers with a maximum size of 1024 in the hypernetwork (ignoring the input-output layers). Since the base parameter's $T$ value was 200, we had to train 200 GDNs and 1 hypernetwork. Considering the architecture mentioned, we had around 500 thousand parameters in each GDN and 750 million in the hypernetwork to train. Despite their outrageous look, these calculations are highly parallelizable. We trained our models within a reasonable time by exploiting the most basic optimizations and employing the graphic processing unit (GPU) to run the computations within each model in parallel. We know that there are some obvious ideas to improve the performance of these networks, like training GDNs in parallel using multiple GPUs simultaneously. However, there was no need to do that since our training time was short enough, and more GPUs would be needed to achieve that.

We ran experiments on the Vector Institute for Artificial Intelligence's computing cluster. Since each problem (process created with a specific set of parameters) is entirely independent, we used 30 machines, one for each problem in parallel. All the machines had the same config with 6 CPU cores, 1 Nvidia T4 GPU, 20GB of RAM, and 40GB of SSD memory. The run-time limit for the instances was 12 hours, although all machines finished their jobs in less than 8.5 hours. The problem with the base parameter set took about 4 hours and 44 minutes, and the overall average was around 6 hours. Note that these differences might be seen because not all machines were on the same host computing node. Thus, the CPU and RAM models and clock frequencies were not the same for all the machines.

\subsection{Algorithm Descriptions}
\label{a:Experiment_Details__ss:AlgoDescriptions}
The following algorithm is used to generate sample paths of the Volterra process defined in~\eqref{eq:Volterra__ExperimentsS}.  

\begin{algorithm}[H]%[ht]%
\caption{Construct $\boldsymbol{X}$}
\label{alg:get_X}
\begin{algorithmic}
\SetAlgoLined
\Require{Number of training samples $N\in \mathbb{N}_+$, time-horizon $T\in \mathbb{N}_+$, dynamics $\mu$, $\sigma$, and $\varsigma$, ``noise'' parameter $0\le \lambda$, memory $1\le w\le 0$}.
\DontPrintSemicolon
    \State \SetKwBlock{ForParallel}{For $n:0,\dots,N-1$ in parallel}{end}
        \ForParallel{\State 
        \For{$t:0,\dots,T$}{
            \State 
                \If{$t=0$}{
                        \State $x^n_{-1}\leftarrow 0$
                        \State $x^n_0\leftarrow 0$
                        \tcp*{Get Initial States}
                    }
                \Else{
                    \State Sample: $Z\sim N_d(0,I_d)$ \tcp*{Generate Gaussian Noise}
                    \State Sample: $B\sim \operatorname{Binom}(\{0,1\};1/2)$ \tcp*{Generate Hidden Process}
                    \State 
                    $x \leftarrow w \mu(x_{t-1}^n) + (1-w) \mu(x_t^n)$ 
                    \tcp*{Get Drift}
                    \State 
                    $y\leftarrow 
                        (\varsigma(x_t^n)\cdot \sigma + B\,\lambda \,I_d) Z
                    $ 
                    \tcp*{Get Diffusion}
                    \State  $x_{t+1}^n\leftarrow x_t^n + x + y$
                    \tcp*{Update Diffusion}
                }
                \State $x^n\leftarrow (x_t^n)_{t=0}^T$
                \tcp*{Save Sample Path}
            }
        }
    \State  $\boldsymbol{X}\leftarrow \{x^n\}_{n=0}^{N-1}$
    \tcp*{Compile Dataset}
    \State \Return $\boldsymbol{X}$
\end{algorithmic}
\end{algorithm}

The next algorithm (Algorithm~\ref{alg:SampleXstar}) implements the Gaussian projection of the Volterra process in~\eqref{eq:Volterra__ExperimentsS}.  This is given by the closed-form expression derived in~\eqref{eq:GaussianProjection__Experiments}.

\begin{algorithm}[H]%[ht!]%
\caption{{Generate $\Pi_{x_{[0:t]}}$ Given Paths}}
\label{alg:SampleXstar}
\begin{algorithmic}
\SetAlgoLined
\Require Time-Horizon, finite set of paths $\boldsymbol{X}\subseteq \mathbb{R}^{(2+T)d}$, drift $\mu$, diffusion parameters $\sigma$ and $\varsigma$, and a ``randomness'' $\lambda\ge 0$.
\DontPrintSemicolon
    \State \SetKwBlock{ForParallel}{For $x\eqdef x_{[-1:T]}\in \boldsymbol{X}$ in parallel}{end for} 
        \ForParallel{
        \State 
        \For{$t:0,\dots,T$}{
            \State 
                $\mu_{t}^x \leftarrow x_t + \operatorname{Drift}(x_{[t-1:x_t]})$
                \tcp*{Get Mean of Projection}
            \State $
                \sigma_{t}^x 
                    \leftarrow
                \varsigma(x_t)\cdot 
                \sigma^2(\sigma^{-2}(
                    \lambda\, I_d + \varsigma(x_t)\cdot \sigma
                    )^2)^{1/2}
            $
            \tcp*{Get Covariance of Projection}
        }
            \State $
                y^{x}
                    \leftarrow 
                \big(\mu_{t}^x,\operatorname{vec}(\sigma_{t}^x)\big)_{t=0}^T
                $
            \tcp*{Get Outputs}
    }
    \State 
    \Return $\boldsymbol{Z}\leftarrow \{(x,y^x)\}_{x\in \boldsymbol{X}}$ 
    \tcp*{Return Array of Input-Output Pairs $\{(x,y^x)\}_{x\in \boldsymbol{X}}$}
\end{algorithmic}
\caption*{Given a set of paths $\mathbf{X}$ in $\mathbb{R}^{(2+T)d}$, this algorithm returns an array of input-output pairs, whose elements are pairs of paths $x_{[-1:T]}$ in $\mathbf{X}$ paired with the parameters determining the path of Gaussian distributions $y^x\eqdef (\mathbb{Q}_{x_{[-1:t]}})_{t=0}^T$ traced out by sequentially applying the barycenter map to the process $(\mathbb{Q}_{x_{[-1:t]}})_{t=0}^T$.}
\end{algorithm}

Once the input data has been generated using Algorithm~\ref{alg:get_X} and the corresponding Gaussian random projections have been computed using Algorithm~\ref{alg:SampleXstar}, then we have input-output with which the HGN model can be trained.  Observe that the training scheme which we used and  which encodes the proof strategy of the dynamic universal approximation theorem~\ref{thrm:approx_HGCNN}, avoids backpropagating in time (as with RNNs).  Thus, even if the HGN has a recursion, it can be trained without recursion similarly and thus enjoys some of the training stability properties of transformers which RNNs do not share; namely, no backpropagation in time.

\begin{algorithm}[ht]%[H]
\caption{Construct HG-CNN}\label{alg:train_CNO}
\begin{algorithmic}
\SetAlgoLined
\Require A dataset $\boldsymbol{Z}\leftarrow \{(x,y^x)\}_{x\in \boldsymbol{X}}$, GDN depth and widths, (ReLU) hypernetwork dimensions $[\mathbf{d}]$.
\DontPrintSemicolon
    \State \SetKwBlock{ForParallel}{For $t:0,\dots,T$ in parallel}{end}
        \ForParallel{\State $\hat{f}_{\theta_{t}} 
        \leftarrow \underset{\hat{f}_{\theta}\in \mathcal{G}-\mathcal{CNN}_{[S],[L]}}{\operatorname{argmin}}\,
                \,
                \sum_{x\in \boldsymbol{X}}
                \,
                \big\|
                        \hat{f}_{\theta}(x_{[t-1:t]})
                    -
                        y^x_t
                \big\|^2
        $ \tcp*{Optimize G-CNN Nodes}
        \State $z_{t} \leftarrow (\theta_{t},t)$
        \tcp*{Separate Parameters}
        }
    \State \tcc{Create Recurrence/ Encode Causality}
    \State $\hat{h}
        \leftarrow 
            \underset{h \in \mathcal{N}\mathcal{N}^{\operatorname{ReLU}}_{[\mathbf{d}]}}{\operatorname{argmin}}\,
                \sum_{t=0}^T
                \,
                \|h(z_t) - z_{t+1}\|_2^2
        $
    \State \tcc{Server receives parameters of optimized neural filters for each time window}
    \State $L:\mathbb{R}^{P([\mathbf{d}])}\times \mathbb{R}^Q \rightarrow \mathbb{R}^{P([\mathbf{d}])}$ projection onto first component
    % \\
    \State \Return Trained HG-CNN: $(\hat{f},z_0,L)$.
\end{algorithmic}
\end{algorithm}

\subsection{Additional Loss Curves}
\label{a:additional_pots}
We plot the loss curves, in the test set, of a representative subset of the experiments reported in Tables~\ref{tab:loss_mu},~\eqref{tab:loss_memory},~\eqref{tab:loss_lambda},~\eqref{tab:loss_varsigma}, and~\eqref{tab:loss_varsigma}.  
Figure~\ref{fig:summaryplot} shows that the behaviour illustrated in Figures~\ref{fig:gap_HGN} and~\eqref{fig:no_gap__HGN} persists across our experiments.

\begin{figure}[H]%[t!]
    \centering
    \includegraphics[width=1\textwidth]{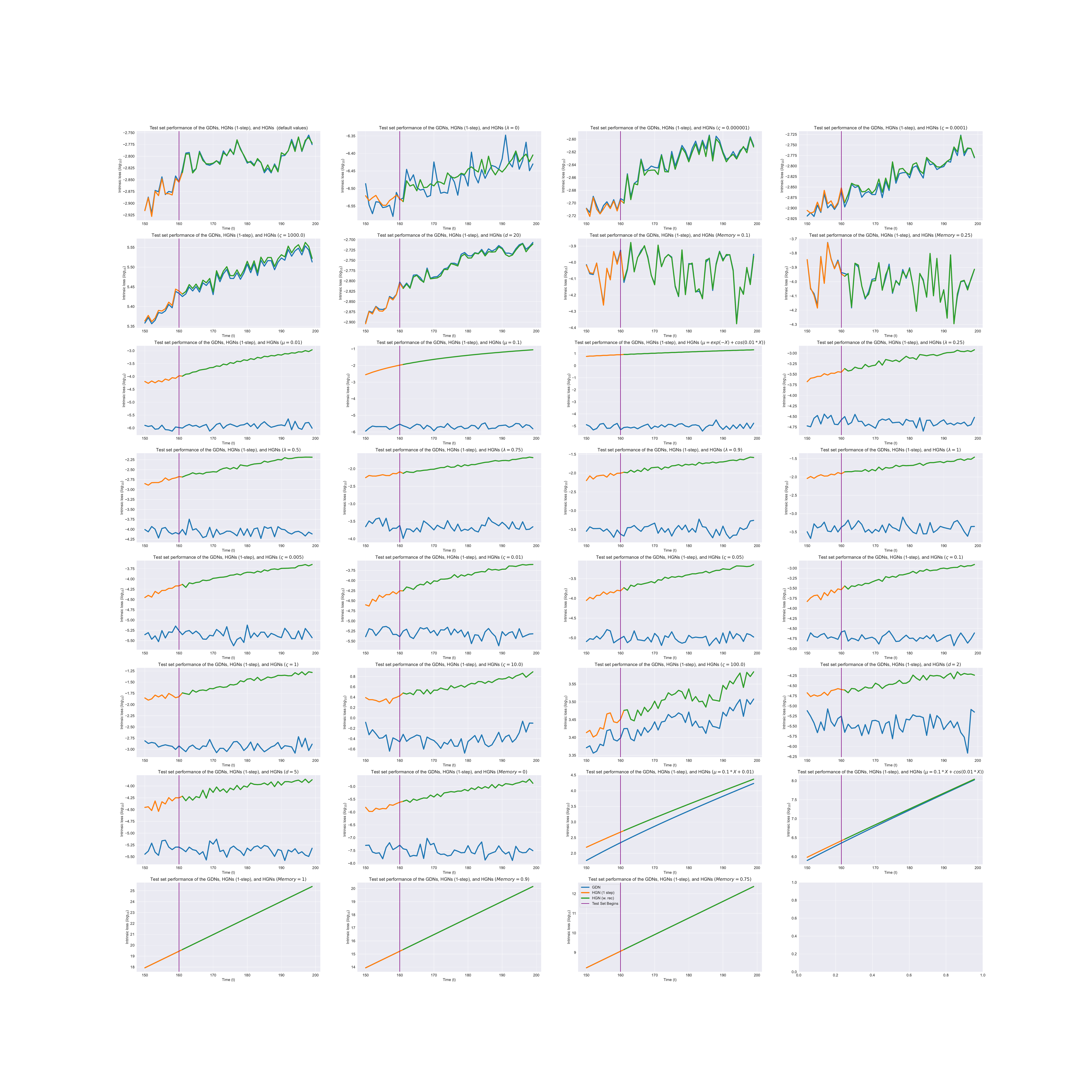}
    \caption{Typical Learning Curves - Including Cases With Exploding Gradients}
    \label{fig:summaryplot}
\end{figure}

\subsection{Exploding Gradients due to Small Eigenvalues}
\label{ss:ExplodingGrads}
There is one additional case, illustrated in Figure~\ref{fig:summaryplot}, where the HNN or HGN suffers from exploding gradients during training.  We included some examples of learning curves illustrating the exploding gradients phenomenon that occurs.  

A prior, there is nothing particular about the target function being learned when this occurs.  For instance, the experiments in Table~\ref{tab:loss_mu} where $x\mapsto \frac1{2}(\frac1{100}-x)$ is essentially the same as when $x\mapsto \frac1{10}\,x + \frac1{100}$; thus both should be equally easy to learn.  However, in the latter experiment, recorded in Table~\ref{tab:exploding_gradients}, the gradients exploded during training, resulting in a loss (both in and out of the sample).   Similarly, the drift $x\mapsto e^{-x}+ \cos(\frac{x}{100})$ and $x\mapsto \frac{x}{10} + \cos(\frac{x}{100}$ are essentially the same; but again the latter is not being learned due to exploding gradients during training (see Table~\ref{tab:exploding_gradients} again) while the former is not due since gradients did not explode (see Table~\ref{tab:loss_mu} again).  Similarly, the value of $w$ is similar to those considered in Table~\ref{tab:loss_memory}; however, the loss in situations where gradients exploded during training is several magnitudes larger.

\begin{table}[H]
\centering
\caption{\textbf{Examples of Exploding Gradients:} 
}
\label{tab:exploding_gradients}
\resizebox{1\textwidth}{!}{% %
\begin{tabular}{@{}lccccc@{}}
\toprule
\textbf{$\mu$} & \textbf{GDN Loss Mean} & \textbf{GDN Loss 95\% C.I.} & \textbf{HGN Loss Mean} & \textbf{HGN Loss 95\% C.I.} \\ \midrule
% 0.1 $x$ + 0.01
$\frac1{10}\,x + \frac1{100}$
& $4.19\times{10^{+3}}$ & $[2.75, 5.64]\times{10^{+3}}$ & $6.22\times{10^{+3}}$ & $[4.24, 8.21]\times{10^{+3}}$ \\
% 0.1 $x$ + cos(0.01 $x$)
$\frac{x}{10} + \cos(\frac{x}{100})$
& $2.82\times{10^{+7}}$ & $[1.92, 3.71]\times{10^{+7}}$ & $2.98\times{10^{+7}}$ & $[2.05, 3.91]\times{10^{+7}}$ \\
\arrayrulecolor{black!30}\midrule
\textbf{$w$} & \textbf{GDN Loss Mean} & \textbf{GDN Loss 95\% C.I.} & \textbf{HGN Loss Mean} & \textbf{HGN Loss 95\% C.I.} \\ \midrule
0.75 & $3.32\times{10^{+11}}$ & $[1.55, 5.09]\times{10^{+11}}$ & $3.32\times{10^{+11}}$ & $[1.55, 5.09]\times{10^{+11}}$ \\
0.9 & $1.42\times{10^{+19}}$ & $[0.45, 2.40]\times{10^{+19}}$ & $1.42\times{10^{+19}}$ & $[0.45, 2.40]\times{10^{+19}}$ \\
\arrayrulecolor{black}\bottomrule
\end{tabular}
}% %
\end{table}

A closer look at the error logs shows that the exploding gradient occurs due to rounding errors in the Riemannian distance function~\eqref{eq:totally_geodesic_distanceform} when the logarithm is applied to small eigenvalues of the relevant positive-definite matrix.  Though gradient clipping typically solves this issue, it occasionally resurfaces, and we thus report it here.

\bibliography{Bookkeaping/2_References}
\end{document}

%% file: Bookkeaping/1_mathcommands.tex
\usepackage{graphicx}
\usepackage[T1]{fontenc}    % use 8-bit T1 fonts
\usepackage{booktabs}       % professional-quality tables

\usepackage{amsfonts,amsmath,amssymb,bbm}
\usepackage{enumitem}
\usepackage{graphicx}
\usepackage[table,dvipsnames]{xcolor}
    \usepackage{adjustbox}
\usepackage[UKenglish]{isodate}
\usepackage{graphicx}
\usepackage[font=small,labelfont=it]{caption}
\usepackage{subcaption}
\usepackage{hyperref} 
    \hypersetup{
    	colorlinks = true,
    	linkcolor = blue,
    	anchorcolor = blue,
    	citecolor = blue,
    	filecolor = blue,
    	urlcolor = blue
    }
    \usepackage{cleveref}
\usepackage{float}
\usepackage{multirow}
\usepackage{fancyvrb}
\usepackage{rotating}
    \usepackage{tikz}
    \usepackage{tikz-cd} 
    \pgfdeclarelayer{edgelayer}
    \pgfdeclarelayer{nodelayer}
    \pgfsetlayers{edgelayer,nodelayer,main}
    \tikzstyle{new style 0}=[fill={rgb,255: red,255; green,94; blue,247}, draw=black, shape=circle]
    \tikzstyle{pointy}=[fill=white, draw=black, shape=circle]
    \tikzstyle{pointy}=[->]

\usepackage{algpseudocode}
\usepackage[linesnumbered,lined,boxed,commentsnumbered,ruled,longend]{algorithm2e}

\SetCommentSty{mycommfont}

\newcommand{\ra}[1]{\renewcommand{\arraystretch}{#1}}
\usepackage{lscape}

\renewcommand{\phi}{\varphi}

\newcommand{\R}{\mathbb{R}}

\newcommand{\Uu}{\mathcal{U}}

\newcommand{\eqdef}{\ensuremath{\stackrel{\mbox{\upshape\tiny def.}}{=}}}

\newtheorem{result_intro}{Informal Results}%{Result}
\usepackage{appendix}

\definecolor{darkcerulean}{rgb}{0.03, 0.27, 0.49}
\definecolor{darkmidnightblue}{rgb}{0.0, 0.2, 0.4}
\definecolor{darkcyan}{rgb}{0.0, 0.55, 0.55}
\definecolor{mediumseagreen}{rgb}{0.24, 0.7, 0.44}
\definecolor{lightseagreen}{rgb}{0.13, 0.7, 0.67}
\definecolor{electricultramarine}{rgb}{0.25, 0.0, 1.0}
\definecolor{darkgreen}{rgb}{0.0, 0.2, 0.13}
\definecolor{deepjunglegreen}{rgb}{0.0, 0.29, 0.29}
\definecolor{darkcandyapplered}{rgb}{0.64, 0.0, 0.0}
\definecolor{darkred}{rgb}{0.55, 0.0, 0.0}
\definecolor{darkscarlet}{rgb}{0.34, 0.01, 0.1}
\definecolor{jasper}{rgb}{0.84, 0.23, 0.24}
    \definecolor{deepcarrotorange}{rgb}{0.91, 0.41, 0.17}
\definecolor{amethyst}{rgb}{0.6, 0.4, 0.8}
\definecolor{darkviolet}{rgb}{0.58, 0.0, 0.83}
\definecolor{darkpurple}{rgb}{0.58, 0.0, 0.83}
\definecolor{electricpurple}{rgb}{0.75, 0.0, 1.0}
\definecolor{darkjazzberryjam}{rgb}{0.45, 0.04, 0.37}
\definecolor{MidnightBlue}{RGB}{25,25,112}
\definecolor{MidnightBlueComplementingGreen}{RGB}{25,112,25}
\definecolor{MidnightBlueComplementingPurple}{RGB}{112,25,112}
\definecolor{MidnightBlueComplementingRed}{RGB}{112,25,69}

\definecolor{lightgray}{rgb}{0.83, 0.83, 0.83}

\definecolor{lightkhaki}{rgb}{0.94, 0.9, 0.55}
\definecolor{cadmiumorange}{rgb}{0.93, 0.53, 0.18}
\definecolor{darkseagreen}{rgb}{0.56, 0.74, 0.56}
\definecolor{babyblueeyes}{rgb}{0.63, 0.79, 0.95}
\definecolor{deepjunglegreen}{rgb}{0.0, 0.29, 0.29}

\usepackage{tikz,tikz-cd}
\usepackage{tkz-euclide}
\usetikzlibrary{matrix,cd}
\usepackage{smartdiagram}

\usepackage{tablefootnote}

\NewDocumentCommand{\good}{}{{\color{green}\boldsymbol{\checkmark}}}
\NewDocumentCommand{\bad}{}{{\color{red}\boldsymbol{\times}}}

\usepackage{marvosym}

\def\sym{{\operatorname{sym}}}

\newcounter{totenum}
\let\Oldenumerate\enumerate
\def\enumerate{\refstepcounter{totenum}\Oldenumerate}

\newcommand{\norm}[1]{\left\lVert#1\right\rVert}
\newcommand{\abs}[1]{\left |#1\right|}

\DeclareMathOperator{\memory}{{H}}

%% file: Bookkeaping/3_Editing.tex
\definecolor{darkcyan}{rgb}{0.0, 0.55, 0.55}
\definecolor{midnightgreen_eaglegreen}{rgb}{0.0, 0.29, 0.33}
\definecolor{shamrockgreen}{rgb}{0.0, 0.62, 0.38}
\definecolor{phthalogreen}{rgb}{0.07, 0.21, 0.14}
\NewDocumentCommand{\anastasis}{mo}{
    \IfValueF{#2}{
                        {{\scriptsize
                            \textcolor{phthalogreen}{ 
                            \textbf{A:}
                            \textit{{#1}}
                            }
                        }}
        }
    \IfValueT{#2}{
                        \marginnote{{\scriptsize
                            \textcolor{phthalogreen}{ 
                            \textbf{A:}
                            \textit{{#1}}
                            }
                        }}
        }
                    }
\definecolor{iris}{rgb}{0.35, 0.31, 0.81}
\NewDocumentCommand{\reza}{mo}{
    \IfValueF{#2}{
                        {{\scriptsize
                            \textcolor{iris}{ 
                            \textbf{RA:}
                            \textit{{#1}}
                            }
                        }}
        }
    \IfValueT{#2}{
                        \marginnote{{\scriptsize
                            \textcolor{iris}{ 
                            \textbf{RA:}
                            \textit{{#1}}
                            }
                        }}
        }
                    }
                    
\NewDocumentCommand{\luca}{mo}{
    \IfValueF{#2}{
                        {{\scriptsize
                            \textcolor{deepjunglegreen}{ 
                            \textbf{L:}
                            \textit{{#1}}
                            }
                        }}
        }
    \IfValueT{#2}{
                        \marginnote{{\scriptsize
                            \textcolor{deepjunglegreen}{ 
                            \textbf{L:}
                            \textit{{#1}}
                            }
                        }}
        }
                    }
\NewDocumentCommand{\giulia}{mo}{
    \IfValueF{#2}{
                        {{\scriptsize
                            \textcolor{red}{ 
                            \textbf{GL:}
                            \textit{{#1}}
                            }
                        }}
        }
    \IfValueT{#2}{
                        \marginnote{{\scriptsize
                            \textcolor{red}{ 
                            \textbf{GL:}
                            \textit{{#1}}
                            }
                        }}
        }
}

\NewDocumentCommand{\john}{mo}{
    \IfValueF{#2}{
                        {{\scriptsize
                            \textcolor{amethyst}{ 
                            \textbf{JA:}
                            \textit{{#1}}
                            }
                        }}
        }
    \IfValueT{#2}{
                        \marginnote{{\scriptsize
                            \textcolor{amethyst}{ 
                            \textbf{L:}
                            \textit{{#1}}
                            }
                        }}
        }
                    }
                    
\usepackage{soul}